\documentclass[10pt,a4paper,twoside,reqno]{amsart}
\usepackage[a4paper,top=3cm,bottom=3cm,inner=2.5cm,outer=2.5cm]{geometry}
\usepackage[utf8]{inputenc}
\usepackage[english]{babel}

\usepackage{xcolor}

\usepackage{amssymb,mathtools,amsmath,amsthm,stmaryrd}
\usepackage[all,cmtip,2cell]{xy}

\usepackage[draft=false, hidelinks, linkcolor = blue]{hyperref}
\usepackage{cleveref}
\usepackage{tensor}
\usepackage{bbm}
\usepackage{url}
\usepackage{bookmark}
\usepackage{footmisc}
\usepackage{quiver}
\usepackage{adjustbox}
 
\usepackage{graphicx}
\usepackage{lmodern}
\usepackage{enumitem}
\usepackage{array}

\def\defthm#1#2#3#4{
  \newtheorem{#1}[theorem]{#3}
  \newtheorem*{#1*}{#3}
  \newtheorem{#2}[theorem]{#4}
  \newtheorem*{#2*}{#4}
  \crefname{#1}{#3}{#4}
  \crefname{#2}{#4}{#4}  
}

\newtheoremstyle{mythm}%
{10pt}
{}
{\itshape}
{}
{\bf}
{.}
{.5em}
{}%

\newtheoremstyle{mydef}%
{10pt}
{3pt}
{}
{}
{\bf}
{.}
{.5em}
{}%

\newtheoremstyle{myrmk}%
{10pt}
{3pt}
{}
{}
{\bf}
{.}
{.5em}
{}%

\theoremstyle{mythm}
\newtheorem{theorem}{Theorem}[section]
\newtheorem*{theorem*}{Theorem}

\defthm{corollary}{corollaries}{Corollary}{Corollaries}
\defthm{lemma}{lemmata}{Lemma}{Lemmata}
\defthm{proposition}{propositions}{Proposition}{Propositions}
\defthm{axiom}{axioms}{Axiom}{Axioms}
\defthm{propdef}{props/defs}{Proposition/Definition}{Prositions/Definitions}
\defthm{defcor}{defscors}{Corollary/Definition}{Corollaries/Definitions}
\defthm{conjecture}{conjectures}{Conjecture}{Conjectures}

\theoremstyle{mydef}
\defthm{definition}{definitions}{Definition}{Definitions}

\theoremstyle{myrmk}
\defthm{acknowledgment}{acknowledgments}{Acknowledgment}{Acknowledgments}
\defthm{remark}{remarks}{Remark}{Remarks}
\defthm{motivation}{motivations}{Motivation}{Motivations}
\defthm{example}{examples}{Example}{Examples}
\defthm{question}{questions}{Question}{Questions}
\defthm{observation}{observations}{Observation}{Observations}
\defthm{claim}{claims}{Claim}{Claims}
\defthm{notation}{notations}{Notation}{Notations}
\defthm{terminology}{terminologies}{Terminology}{Terminologies}
\defthm{construction}{constructions}{Construction}{Constructions}
\defthm{conclusion}{conclusions}{Conclusion}{Conclusions}
\defthm{warning}{warnings}{Warning}{Warnings}

\newtheorem*{replemmax}{\reptitle}
 {\end{replemmax}}

\newtheorem*{repthmx}{\reptitle}
 {\end{repthmx}}

\newtheorem*{repcorx}{\reptitle}
 {\end{repcorx}}

\newcommand{\sprime}{^{\prime}}

\newcommand{\pbs}{\scalebox{1.5}{\rlap{$\cdot$}$\lrcorner$}}

\newcommand*{\twocat}{2\text{-Cat}}

\newcommand*{\hotwocat}{2\text{-Cat}^{\times}}

\newcommand{\xdblrightarrow}[2][]{%
  \xrightarrow[#1]{#2}\mathrel{\mkern-14mu}\rightarrow
}

\begin{document}
\title[The higher algebra and geometry of bicategories]{The higher algebra and geometry of monoidal bicategories} 

\author[R. Stenzel] {Raffael Stenzel}

\begin{abstract}
We show that braided, sylleptic and symmetric monoidal bicategories are precisely the $\mathsf{E}_k$-monoids in the cartesian monoidal
$(\infty,1)$-category of bicategories for respective integers $k$. To manage the underlying computations, we use the geometry of the
little cubes operads to mediate between the 2-dimensional algebra underlying the former and the $\infty$-categorical algebra underlying the 
latter.

%
\end{abstract}

\date{\today}

\maketitle

\section{Introduction}

Categorification is the program of generalising the study of algebraic structures defined on sets --- that is, the $0$-dimensional algebra of 
monoids, groups, rings etc.\ --- 
to the study of algebraic structures defined on (potentially higher dimensional) categories. It plays a major role across algebra, topology, 
geometry, physics and logic, particularly in the fields of topological quantum field theories, K-theory and representation theory 
\cite{baezdolan_cat, schommerpries_thesis, lurieha}. The central basic notions in this program are the corresponding instances of a monoid 
together with their corresponding instantiations of commutativity and invertibility.
In this context, the foundation of 1-dimensional algebra --- that is, the theory of (commutative) monoidal categories --- has 
become an essential machinery that underpins a plethora of mathematical constructions that arose in the past decades \cite{joyalstreet}. 
However, various algebraic structures arising in the context of topology and geometry are often naturally higher dimensional, and certain 
important features of these structures cannot be captured in this way. Thus, Kapranov and Voevodsky
\cite{kapranovvoevodskyI,kapranovvoevodskyII} introduced a theory of (commutative) monoidal 2-dimensional categories with some of those 
applications in mind. This consequently led to a long arc of research in 2-dimensional algebra from many different angles
\cite{johnsonyau,fioregambinohyland}.
On the other side of the spectrum, a theory of $\infty$-categorified algebra has more recently been developed following the operadic 
tradition of topology \cite{may_loop, lurieha}. 

This paper constitutes the foundation for an at least -- and very possibly precisely -- two part series on the $\infty$-categorical algebra 
of bicategories. The ultimate aim of this series is to use the results of higher algebra as pivotally developed in \cite{lurieha} to give 
fairly straightforward proofs of otherwise computationally long-winded results in low dimensional category theory. Thus, in a sense, this 
project is an application of topological algebra in the tradition of K-theory, opposed to algebraic topology in the tradition of homological 
algebra going the other way round. Here we make crucial use of the rather basic but essential fact that all objects in 2-dimensional algebra 
are defined by imposing the existence of invertible higher cells only; they are hence captured by the $\infty$-category of 
bicategories.\footnote{Here, ``$\infty$'' is to be read ``$(\infty,1)$'' as usual. On that note it should be added that Day and Street 
\cite{daystreet} have set up their definitions to capture lax morphisms and
non-invertible transformations between 2-dimensional algebraic objects as well. These however exceed the scope of $\infty$-category theory, 
and would require an $(\infty,2)$-categorical treatment instead. This aspect of 2-dimensional algebra will hence not be subject of this 
paper.}

The main contribution of this paper is the identification of the $\infty$-categorified 
algebra of \cite{lurieha} when evaluated in the cartesian monoidal $\infty$-category $\mathrm{BiCat}^{\times}$ of bicategories, and the
2-dimensional algebra of \cite{daystreet} and \cite{crans_centers} --- in turn extending \cite{kapranovvoevodskyI,kapranovvoevodskyII} --- to 
the extent specified in Theorem~\ref{thm_main}.
Against the background of the ``tortuous history of symmetric monoidal bicategories'' \cite{schommerpries_thesis}, it hence serves as a 
proof that the algebraic definitions of braidings, syllepses and symmetries on monoidal bicategories as they ultimately have been defined in 
\cite{baezneuchl} are the intended ones. Indeed, the following theorem has been 
expected to hold for some time, see e.g.\ \cite{haugseng_hanotes}, but any argument towards its merit has been missing to date.

\begin{theorem}\label{thm_main}
There are the following bijections between sets of respective equivalence classes:
\begin{enumerate}
\item $\{$monoidal bicategories$\}\cong\{\mathsf{E}_1$-algebras in $\mathrm{BiCat}^{\times}\}$;
\item $\{$braided monoidal bicategories$\}\cong\{\mathsf{E}_2$-algebras in $\mathrm{BiCat}^{\times}\}$;
\item $\{$sylleptic monoidal bicategories$\}\cong\{\mathsf{E}_3$-algebras in $\mathrm{BiCat}^{\times}\}$;
\item $\{$symmetric monoidal bicategories$\}\cong\{\mathsf{E}_{\infty}$-algebras in $\mathrm{BiCat}^{\times}\}$.
\end{enumerate}
These bijections are natural with respect to the obvious forgetful maps from bottom to top. The same applies to the corresponding 
notions of morphisms between them.
\end{theorem}
The corresponding notions of equivalence that these sets are quotients of are in each case the natural ones, and will be made precise in the 
body of the paper. To put the theorem in perspective, we recall that monoidal categories are precisely the $\mathsf{E}_1$-algebras in the 
cartesian monoidal $\infty$-category $\mathrm{Cat}^{\times}$ of categories. Joyal and Street \cite[Section 5]{joyalstreet} showed that 
braidings on a monoidal category $\mathcal{C}^{\otimes}$ stand in 1-1 correspondence to monoidal structures on the multiplication
$m\colon\mathcal{C}\times\mathcal{C}\rightarrow\mathcal{C}$ of $\mathcal{C}^{\otimes}$. As observed in
\cite[Example 5.1.2.4]{lurieha}, this leads to a 1-1 correspondence between braided 
monoidal categories and $\mathsf{E}_2$-algebras in $\mathrm{Cat}^{\times}$. As subsequently observed in loc.\ cit., this 1-1 correspondence 
further restricts to an equivalence between braided monoidal categories whose braiding is symmetric 
and $\mathsf{E}_{\infty}$-algebras in $\mathrm{Cat}^{\times}$. Theorem~\ref{thm_main} is hence a higher dimensional generalization of 
\cite[Example 5.1.2.4]{lurieha}, where the results of Joyal and Street about monoidal categories \cite{joyalstreet} are replaced 
by the corresponding results of Day and Street about monoidal 2-categories \cite{daystreet}.

In brief, its proof proceeds 
as follows. In one direction, we observe that every (sylleptic) braiding on a Gray monoid defines one half of an initial segment of 
a corresponding $\mathsf{E}_{k}$-monoid in $\mathrm{BiCat}^{\times}$ directly by way of the results of Day and Street. We then show that 
first algebraically reversing and then geometrically reflecting these (sylleptic) braidings yields another half of an initial 
segment. These two can be pasted together to yield a full $\mathsf{E}_{k}$-monoid. In the other direction, we use an Eckmann--Hilton type 
corollary of the Additivity Theorem together with a series of semi-strictification results which use both homotopical and algebraic 
techniques.

\paragraph{Outline}
A development of the necessary $\infty$-categorical tools is given in Section~\ref{sec_infty}. Here, we remind the reader of various facts 
about (non-unital) $\mathsf{E}_k$- and $\mathsf{A}_n$-algebras, give an Eckmann--Hilton type argument for the $\infty$-operads $\mathsf{E}_k$ 
(Proposition~\ref{prop_eh}), and define the diagonal reflection of an $\mathsf{E}_k$-algebra (Section~\ref{sec_sub_revmon_infty}). In 
Section~\ref{sec_sub_iniseg} we define the $\infty$-category of initial segments of an $\mathsf{E}_k$-algebra in a symmetric monoidal
$\infty$-category $\mathcal{C}^{\otimes}$, and give precise bounds for the homotopy type of the space of full $\mathsf{E}_k$-algebra extensions
of an initial segment provided that $\mathcal{C}$ is an $n$-category. 
Section~\ref{sec_modbicat} recalls the 2-dimensional algebra of bicategories in terms of Gray monoids, defines the reverse of a braiding 
and that of a syllepsis in this context, and discusses their associated homotopy theory. Section~\ref{sec_translation} combines the results 
of the prior two sections to give a proof of Theorem~\ref{thm_main}.

\paragraph{Related Work}
Gurski and Gurski--Osorno \cite{gurski_braid,gurskiosorno_sym} provide a series of coherence results for braided monoidal 
bicategories and for symmetric monoidal bicategories, respectively. As an application, the authors show that every 
$\mathsf{E}_k$-space gives rise to a monoidal (braided/sylleptic/symmetric) bigroupoid by way of a homotopy-bigroupoid trifunctor
$\Pi\colon\mathrm{Top}^{\times}\rightarrow\mathrm{BiCat}^{\times}$ for respective integers $k$. They also show that every symmetric 
monoidal bicategory gives rise to a special $\Gamma$-bicategory, that is, essentially, an $\mathsf{E}_{\infty}$-monoid in the tricategory
$\mathrm{BiCat}^{\times}$. We note that these applications regarding the functor $\Pi$ arise as corollaries of Theorem 1.1 as well.

\paragraph{Future work} One of many natural follow-up questions that ensue is whether the center $Z(\mathcal{C})$ of a 
monoidal bicategory $\mathcal{C}$ as defined (in algebraic manner across many pages) by Baez--Neuchl \cite{baezneuchl} is 
indeed the center thereof in the cartesian monoidal $\infty$-category $\mathrm{BiCat}^{\times}$, defined in \cite[Section 5.3.1]{lurieha} 
by way of a slick universal property. Both are essentially constructed to be a moduli space for braidings on $\mathcal{C}$. The
1-categorical version of this equivalence is again shown in \cite[Example 5.3.1.18]{lurieha}. 

\paragraph{Possible alternative approaches}
The proof presented in this paper is an amalgamation of algebraic, homotopical, combinatorial and geometric arguments. One proof that lends 
itself more naturally to category theorists might proceed as follows. The tricategory of monoidal bicategories is 
the tricategory of pseudomonoids in the tricategory of bicategories. Analogously, one ought to be able to show that 
the higher category of braided (sylleptic/symmetric) bicategories is the higher category of pseudomonoids in the higher category of 
(braided/sylleptic) monoidal bicategories. Then one ought to be able to further construct a
well-behaved nerve functor from the higher category of (monoidal) higher categories in which all cells of dimension $n>1$ are invertible to 
the $\infty$-category of (monoidal) $\infty$-categories. This nerve ought to map the higher category of pseudomonoids in 
any monoidal higher category with invertible higher cells to the $\infty$-category of $\mathsf{E}_1$-algebras in its associated nerve.
The usual properties of nerve functors then should imply the theorem. This strategy is not pursued
primarily given
the amount of bulky machinery that needs to be developed as part of the strategy, especially as ``higher'' may well exceed three dimensions
depending on the implementation.

\begin{acknowledgments*}
This paper was written as part of a project initiated by Nicola Gambino; without his impetus and continuous support the paper would not have 
come to be. The larger project ultimately has been inspired by suggestions of Clark Barwick and James Cranch. The author 
would like to thank John Bourke for his insightful comments, Amar Hadzihasanovic for his kind invite to Tallinn to exchange ideas, and 
Adrian Miranda and Mike Shulman for many helpful discussions. This material is based upon work supported by the US Air Force Office for 
Scientific Research under award number FA9550-21-1-0007, as well as upon work supported by the Air Force Office of Scientific Research under 
award number FA9550-21-1-0009.
\end{acknowledgments*}

\section{Monoidal $(\infty,1)$-categories and iterated associative structures}\label{sec_infty}

Section~\ref{sec_sub_prelims} is a very brief reminder of the theory of $\infty$-operads. Section~\ref{sec_sub_comm} discusses the
$\infty$-operads $\mathsf{E}_k$, their Additivity Theorem and an Eckmann--Hilton type corollary. Section~\ref{sec_sub_assoc} discusses 
(non-unital) $\mathsf{A}_m$-algebras, and Section~\ref{sec_sub_revmon_infty} introduces the diagonal reflection of an
$\mathsf{E}_k$-algebra in a symmetric monoidal $\infty$-category $\mathcal{C}^{\otimes}$. Section~\ref{sec_sub_iniseg} introduces the 
notion of an initial segment of an $\mathsf{E}_{k+1}$-algebra, and gives a precise bound for the homotopy type of the space of extensions 
to a full $\mathsf{E}_{k+1}$-algebra structure thereof whenever $\mathcal{C}$ is an $n$-category. In Section~\ref{sec_sub_iniseg_app} we 
apply the results of the prior section to small $n$. This recovers some known criteria for the construction of $\mathsf{E}_{k+1}$-algebras 
in the case of $n=1,2$, and yields the criteria which we will apply to the cartesian $3$-category of bicategories in 
Section~\ref{sec_translation}.

\subsection{Preliminaries}\label{sec_sub_prelims}
\paragraph{$\infty$-Operads}

The notion of an $\infty$-operad generalises the notion of a (coloured) symmetric multicategory in essentially the same way as the notion 
of an $\infty$-category generalises that of a category. More precisely, let $\mathrm{Fin}_{\ast}$ denote the category of finite pointed 
sets and base-point preserving functions. We consider $\mathrm{Fin}_{\ast}$ as an $\infty$-category. We recall that an $\infty$-operad
$\mathcal{O}\rightarrow\mathrm{Fin}_{\ast}$ is defined to be a morphism of simplicial sets that satisfies the Segal condition, and that 
further exhibits cocartesian lifts to all inert maps in $\mathrm{Fin}_{\ast}$ \cite[Definition 2.1.1.10]{lurieha}. Every such morphism
$\mathcal{O}\rightarrow\mathrm{Fin}_{\ast}$ is in particular an isofibration, and so the domain $\mathcal{O}$ of an $\infty$-operad is 
always an $\infty$-category. The latter is sometimes referred to as the $\infty$-category of operators associated to the $\infty$-operad. 
We often will refer to an $\infty$-operad simply by way of its $\infty$-category of operators.
The definition implies that the fibers $\mathcal{O}(\langle n\rangle)$ of an $\infty$-operad are canonically equivalent to the $n$-ary 
product $\prod_{1\leq i\leq n}\mathcal{O}(\langle 1\rangle)$; in particular, an element in $\mathcal{O}(\langle n\rangle)$ is essentially 
an $n$-tuple of elements in $\mathcal{O}(\langle 1\rangle)$. Given an $n$-tuple $\vec{x}\in\mathcal{O}(\langle n\rangle)$ and a single 
object $y\in\mathcal{O}([1])$, the definition further allows to think of the mapping space $\mathcal{O}(\vec{x},y)$ as the space of 
multi-morphisms from $\vec{x}$ to $y$ in the $\infty$-operad $\mathcal{O}$. For general objects $\vec{y}\in\mathcal{O}(\langle m\rangle)$, 
the mapping space $\mathcal{O}(\vec{x},\vec{y})$ simply decomposes into the product of the mapping spaces $\mathcal{O}(\vec{x},y_i)$. In  
particular, whenever the $\infty$-operad is single coloured, i.e.\ whenever $\mathcal{O}(\langle 1\rangle)$ is contractible with center of 
contraction $x$, then the mapping space $\mathcal{O}((x,\dots,x),x)$ may be thought to consist of the $n$-ary operations associated to the
$\infty$-operad.

We further recall that a symmetric monoidal $\infty$-category $\mathcal{C}^{\otimes}\twoheadrightarrow\mathrm{Fin}_{\ast}$ is an
$\infty$-operad that exhibits cocartesian lifts to all morphisms in
$\mathrm{Fin}_{\ast}$ \cite[Example 2.1.2.18]{lurieha}. In particular, a symmetric monoidal $\infty$-category is a cocartesian fibration
$\mathcal{C}^{\otimes}\twoheadrightarrow\mathrm{Fin}_{\ast}$. By way of the $\infty$-categorical Grothendieck construction, it can be 
equivalently described as a Segal object $\mathrm{Fin}_{\ast}\rightarrow\mathrm{Cat}_{\infty}$ with a contractible $\infty$-category of 
objects. Essentially, that is, as an $\infty$-category object in $\mathrm{Cat}_{\infty}$ with a single object and commutative composition.
The underlying $\infty$-category of a symmetric monoidal $\infty$-category $\mathcal{C}^{\otimes}$ is given by the fiber
$\mathcal{C}:=\mathcal{C}^{\otimes}(\ast)$ over the unit $\ast\in\mathrm{Fin}_{\ast}$.

\begin{example}\label{exple_cartmon}
Every $\infty$-category $\mathcal{C}$ with finite products gives rise to a cartesian symmetric 
monoidal $\infty$-category $\mathcal{C}^{\times}\twoheadrightarrow\mathrm{Fin}_{\ast}$ \cite[Section 2.4.1]{lurieha}. In particular, there 
is a cartesian monoidal $\infty$-category $\mathrm{Cat}_{\infty}^{\times}$ of $\infty$-categories. Furthermore, every finite product 
preserving functor $F\colon\mathcal{C}\rightarrow\mathcal{D}$ between $\infty$-categories with finite products induces a symmetric monoidal 
functor $F\colon\mathcal{C}^{\times}\rightarrow\mathcal{D}^{\times}$ \cite[Corollary 2.4.1.8]{lurieha}.
\end{example}

\paragraph{Algebras for $\infty$-operads}

Given an $\infty$-operad $\mathcal{O}$, the class of cocartesian morphisms in $\mathcal{O}$ over inert maps in $\mathrm{Fin}_{\ast}$ is 
called the class of inert morphisms in $\mathcal{O}$. Given another $\infty$-operad $\mathcal{V}$, one may consider the full
sub-$\infty$-category $\mathrm{Alg}_{\mathcal{O}}(\mathcal{V})\subseteq\mathrm{Fun}_{\mathrm{Fin}_{\ast}}(\mathcal{O},\mathcal{V})$ spanned 
by those morphisms of simplicial sets over $\mathrm{Fin}_{\ast}$ that preserve inert morphisms. This $\infty$-category is often referred to 
as the $\infty$-category of morphisms of $\infty$-operads from $\mathcal{O}$ to $\mathcal{V}$, or equivalently, as the $\infty$-category of 
$\mathcal{O}$-algebras in $\mathcal{V}$. 

\begin{remark}
The class of $\infty$-operads is the class of fibrant objects in the slice category
$(\mathrm{S}^+/(\mathrm{Fin}_{\ast},\mathrm{Inert}))$ of marked simplicial sets equipped with the model structure
$\mathrm{Op}_{\infty}$ for $\infty$-operads \cite[Section 2.1.4]{lurieha}. Then, $\mathcal{O}$-algebras in $\mathcal{V}$ are simply 
morphisms between the fibrant objects $\mathcal{O}$ and $\mathcal{V}$ in this model category.
\end{remark}

\begin{proposition}\label{prop_bv}
Let $\mathcal{C}^{\otimes}$ be a symmetric monoidal $\infty$-category and $\mathcal{V}$ be an $\infty$-operad. Then the $\infty$-category
$\mathrm{Alg}_{\mathcal{V}}(\mathcal{C}^{\otimes})$ inherits a symmetric monoidal structure from that of $\mathcal{C}^{\otimes}$.
\end{proposition}
\begin{proof}
\cite[Example 3.2.4.4]{lurieha}.
\end{proof}

\begin{example}\label{exple_bv}
Proposition~\ref{prop_bv} allows to construct $\infty$-categories of algebras iteratively whenever we start with a symmetric monoidal
$\infty$-category $\mathcal{C}^{\otimes}$. That is to say, given $\infty$-operads $\mathcal{O}$ and $\mathcal{V}$, we may construct
the $\infty$-category $\mathrm{Alg}_{\mathcal{O}}(\mathrm{Alg}_{\mathcal{V}}(\mathcal{C}^{\otimes}))$ of $\mathcal{O}$-algebras 
in $\mathcal{V}$-algebras in $\mathcal{C}^{\otimes}$. Moreover, the $\infty$-category of $\infty$-operads comes equipped with a closed 
symmetric monoidal structure itself \cite[Section 2.2.5]{lurieha}. The corresponding tensor product
$\mathcal{O}\otimes_{\mathrm{BV}}\mathcal{V}$ of two $\infty$-operads $\mathcal{O}$ and $\mathcal{V}$ is an $\infty$-operadic version of 
the (derived) Boardman-Vogt tensor product. It induces a natural equivalence
\[\mathrm{Alg}_{\mathcal{O}}(\mathrm{Alg}_{\mathcal{V}}(\mathcal{C}^{\otimes}))\simeq\mathrm{Alg}_{\mathcal{O}\otimes_{\mathrm{BV}}\mathcal{V}}(\mathcal{C}^{\otimes})\]
of symmetric monoidal $\infty$-categories.
\end{example}

\begin{example}
Central to this paper are the little cubes $\infty$-operads $\mathsf{E}_n$ for $0\leq n \leq \infty$. We refer the reader to 
\cite[Definition 5.1.0.2]{lurieha} for a precise definition. Given a 
symmetric monoidal $\infty$-category $\mathcal{C}^{\otimes}$, the $\infty$-category $\mathrm{Alg}_{\mathsf{E}_1}(\mathcal{C}^{\otimes})$ is 
(essentially) the $\infty$-category of associative and unital monoids in $\mathcal{C}^{\otimes}$. The $\infty$-operad $\mathsf{E}_{\infty}$ 
is given by the $\infty$-category $\mathrm{Fin}_{\ast}$ itself (or rather by the identity to itself). The $\infty$-category
$\mathrm{Alg}_{\mathsf{E}_{\infty}}(\mathcal{C}^{\otimes})$ is the $\infty$-category of commutative, associative and unital monoids in
$\mathcal{C}^{\otimes}$. We will say more about the intermediate integers $1<n<\infty$ below.
\end{example}

\begin{notation}
Let $\mathcal{C}^{\otimes}$ be a symmetric monoidal $\infty$-category. Whenever $\mathcal{C}^{\otimes}$ is cartesian, 
$\mathsf{E}_k$-algebras in $\mathcal{C}^{\otimes}$ are commonly referred to as $\mathsf{E}_k$-monoids. Furthermore, generally, the
$\infty$-category $\mathrm{Alg}_{\mathsf{E}_k}(\mathcal{C}^{\otimes})$ is 
itself again symmetric monoidal again by Proposition~\ref{prop_bv}. In particular, given an $\mathsf{E}_k$-algebra $A$ in
$\mathcal{C}^{\otimes}$, we may consider powers $A^{\otimes n}$ in $\mathrm{Alg}_{\mathsf{E}_k}(\mathcal{C}^{\otimes})$ thereof. 
If by $m\colon A_1\otimes A_1\rightarrow A_1$ we denote the underlying multiplication of $A$ on its underlying object $A_1$ in
$\mathcal{C}$, then we will denote the induced multiplication of the $\mathsf{E}_k$-algebra $A^{\otimes n}$ by
$m^{\otimes}\colon A^{\otimes n}_1\otimes A^{\otimes n}_1\rightarrow A^{\otimes n}_1$.
\end{notation}

\begin{remark}\label{rem_algfib}
Let $\mathcal{C}^{\otimes}$ be a symmetric monoidal $\infty$-category. Any monomorphism
$\iota\colon\mathcal{O}\hookrightarrow\mathcal{V}$ of $\infty$-operads induces an isofibration
\[\iota^{\ast}\colon\mathrm{Alg}_{\mathcal{V}}(\mathcal{C}^{\otimes})\twoheadrightarrow\mathrm{Alg}_{\mathcal{O}}(\mathcal{C}^{\otimes})\]
of $\infty$-categories. In particular, whenever we are given a $\mathcal{V}$-algebra $A$ in $\mathcal{C}^{\otimes}$ and we are given an 
equivalence $f\colon\iota^{\ast}(A)\simeq B$ of $\mathcal{O}$-algebras in $\mathcal{C}^{\otimes}$, then $B$ can be extended to a
$\mathcal{V}$-algebra $\bar{B}$ together with an equivalence $\bar{f}\colon A\rightarrow\bar{B}$ such that $\iota^{\ast}(\bar{f})=f$.
This holds more generally for every inclusion $\iota\colon (S,E)\hookrightarrow (T,F)$ of marked simplicial sets over $\mathrm{Fin}_{\ast}$. 
We will apply this repeatedly (and sometimes implicitly) to inclusions of the form $B_k\hookrightarrow \mathsf{E}_k$ to extend a concrete set of 
partial data to a full $\mathsf{E}_k$-algebra structure provided we know that this data is equivalent to the partial data of another $\mathsf{E}_k$-algebra.
\end{remark}

\begin{example}
The $\infty$-category $\mathrm{SMonCat}_{\infty}:=\mathrm{Alg}_{\mathsf{E}_{\infty}}(\mathrm{Cat}_{\infty}^{\times})$ of
$\mathsf{E}_{\infty}$-monoids in the cartesian monoidal $\infty$-category $\mathrm{Cat}_{\infty}^{\times}$ is (equivalent to) the
$\infty$-category of symmetric monoidal $\infty$-categories. Indeed, its objects are the morphisms of $\infty$-operads from
$\mathrm{Fin}_{\ast}$ to $\mathrm{Cat}_{\infty}^{\times}$ by definition, and so they present symmetric monoidal
$\infty$-categories by way of the $\infty$-categorical Grothendieck construction as referred to above. One may think of morphisms in
$\mathrm{SMonCat}_{\infty}$ as the strong symmetric monoidal functors between symmetric monoidal $\infty$-categories. In contrast, general 
morphisms of $\infty$-operads between symmetric monoidal $\infty$-categories are the lax symmetric monoidal functors.
\end{example}

\begin{example}\label{exple_moncat}
Similarly, the $\infty$-category $\mathrm{MonCat}_{\infty}:=\mathrm{Alg}_{\mathsf{E}_{1}}(\mathrm{Cat}_{\infty}^{\times})$ of $\mathsf{E}_1$-monoids 
in the cartesian monoidal $\infty$-category $\mathrm{Cat}_{\infty}^{\times}$ is (equivalent to) the $\infty$-category of monoidal
$\infty$-categories. The morphisms in $\mathrm{MonCat}_{\infty}=\mathrm{Alg}_{\mathsf{E}_1}(\mathrm{Cat}_{\infty}^{\times})$ are again the 
strong monoidal functors between monoidal $\infty$-categories (rather than their lax versions).
To relate the definition of a monoidal $\infty$-category to a notion of monoid in $\mathrm{Cat}_{\infty}$ more familiar to low dimensional 
category theory, we recall that there is an embedding $\Delta^{op}\hookrightarrow\mathrm{Fin}_{\ast}$ by assigning to a natural number 
$n$ its  associated ``cut'' \cite[Section 4.1.2]{lurieha}. Pullback of the cocartesian fibration
$\mathrm{Cat}_{\infty}^{\times}\twoheadrightarrow\mathrm{Fin}_{\ast}$ along this embedding yields a cocartesian fibration
$\mathrm{Cat}_{\infty}^{\times}\twoheadrightarrow\Delta^{op}$. This defines the ``planar'' monoidal $\infty$-category 
associated to $\mathrm{Cat}_{\infty}^{\times}$ by forgetting its symmetry. Now, the $\infty$-category $\mathrm{MonCat}_{\infty}$ is 
the full sub-$\infty$-category
$\mathrm{Alg}_{\Delta^{op}}(\mathrm{Cat}_{\infty}^{\times})\subseteq\mathrm{Fun}_{\Delta^{op}}(\Delta^{op},\mathrm{Cat}_{\infty})$ spanned 
by the morphisms of planar $\infty$-operads. Again, the latter are defined as those sections
$\Delta^{op}\rightarrow\mathrm{Cat}_{\infty}^{\times}$ that preserve inert morphisms; or in this case, equivalently, as those functors
$\Delta^{op}\rightarrow\mathrm{Cat}_{\infty}$ that map $[0]$ to the terminal $\infty$-category, and that satisfy the Segal conditions. 
These are, essentially, $\infty$-category objects in $\mathrm{Cat}_{\infty}$ with a single object.
\end{example}

\subsection{Fragments of commutativity as iterated monoidal structures}\label{sec_sub_comm}
As stated above, an $\mathsf{E}_1$-algebra in a symmetric monoidal $\infty$-category $\mathcal{C}^{\otimes}$ is to be thought of as a
homotopy-coherently associative and unital algebra in $\mathcal{C}^{\otimes}$. An extension of a given $\mathsf{E}_1$-algebra $A$ in
$\mathcal{C}^{\otimes}$ to an $\mathsf{E}_m$-algebra for $2\leq m\leq\infty$ is generally understood as a degree of homotopy-coherent 
commutativity, which only reflects full commutativity for $m=\infty$. Indeed, the $\infty$-category of $\mathsf{E}_{\infty}$-algebras in a 
symmetric monoidal $\infty$-category $\mathcal{C}^{\otimes}$ is the limit of a tower
\[\mathrm{Alg}_{\mathsf{E}_{\infty}}(\mathcal{C}^{\otimes})\rightarrow\dots\rightarrow\mathrm{Alg}_{\mathsf{E}_{m+1}}(\mathcal{C}^{\otimes})\rightarrow\mathrm{Alg}_{\mathsf{E}_{m}}(\mathcal{C}^{\otimes})\rightarrow\dots\rightarrow\mathrm{Alg}_{\mathsf{E}_1}(\mathcal{C}^{\otimes})\]
of forgetful functors \cite[Corollary 5.1.1.5]{lurieha}.

By virtue of the Additivity Theorem \cite[Theorem 5.1.2.2]{lurieha}, the $\infty$-operad $\mathsf{E}_m$ is the $m$-fold Boardman-Vogt 
tensor power of the $\infty$-operad $\mathsf{E}_1$. More precisely, the theorem states that the 
canonical algebra $\rho\colon \mathsf{E}_{k_1}\otimes_{\mathrm{BV}}\mathsf{E}_{k_2}\rightarrow \mathsf{E}_{k_1+k_2}$ is an equivalence of
$\infty$-operads for any pair of positive integers $k_1$ and $k_2$, see \cite[Section 5.1.2]{lurieha} for details. Informally, for 
$k_1=k_2=1$ it states that to give two interacting homotopy-coherent unital and associated monoid structures on an object is the same as to 
give a single $\mathsf{E}_2$-commutative monoid structure on the object. This constitutes an $\infty$-categorical version of the classical 
Eckmann-Hilton argument. To see this, in the following we make precise the fact that $\mathsf{E}_2$-algebra structures introduce an 
interaction of the given multiplication of its underlying $\mathsf{E}_1$-algebra structure with itself rather than to introduce a second 
new multiplication. Therefore, we recall that the 
classical Eckmann-Hilton argument essentially shows that to give an $\mathsf{E}_1$-algebra structure on top of an already determined
$\mathsf{E}_1$-algebra $A$ in a symmetric monoidal category $\mathcal{C}^{\otimes}$ is not to give an arbitrary new multiplication subject 
to further coherence laws, but further coherence laws directly on the already provided multiplication (Remark~\ref{rem_eh_2}, 
Remark~\ref{rem_eh_1}). In the general $\infty$-categorical case, we recall from \cite[Section 2.2.5]{lurieha} the symmetric monoidal 
structure on $\mathrm{Fin}_{\ast}$ given by the smash product
\[\wedge\colon\mathrm{Fin}_{\ast}\times \mathrm{Fin}_{\ast}\rightarrow\mathrm{Fin}_{\ast}.\]
The element $\{\langle 1\rangle\}\colon\ast\rightarrow \mathsf{E}_1$ induces two embeddings
\[\xymatrix{
\mathsf{E}_1\ar@<.5ex>[r]^(.4){(1,\{\langle 1\rangle\})}\ar@<-.5ex>[r]_(.4){(\{\langle 1\rangle\},1)}\ar@{->>}[d] & \mathsf{E}_1\times \mathsf{E}_1\ar[d]\ar@{^(->}[r]^(.45){\sim} &  
\mathsf{E}_1\otimes_{\mathrm{BV}}\mathsf{E}_1\ar@{->>}[d] \\
\mathrm{Fin}_{\ast}\ar@<.5ex>[r]^(.4){(1,\{\langle 1\rangle\})}\ar@<-.5ex>[r]_(.4){(\{\langle 1\rangle\},1)} & \mathrm{Fin}_{\ast}\times\mathrm{Fin}_{\ast}\ar[r]_(.6){\wedge} & \mathrm{Fin}_{\ast}
}\]
of $\infty$-categories. The bottom two composite horizontal maps are both the identity on $\mathrm{Fin}_{\ast}$, and so the top two composite horizontal maps give rise to embeddings
\begin{align}\label{equ_eh}
\iota_i\colon \mathsf{E}_1\rightarrow \mathsf{E}_1\otimes_{\mathrm{BV}}\mathsf{E}_1
\end{align}
of $\infty$-operads for $i=0,1$. 
For any other given $\infty$-operad $\mathcal{V}$, the two embeddings (\ref{equ_eh}) induce two projections
\[U,\mathrm{Alg}_{\mathsf{E}_1}(U)\colon\mathrm{Alg}_{\mathsf{E}_1}(\mathrm{Alg}_{\mathsf{E}_1}(\mathcal{V}))\rightarrow\mathrm{Alg}_{\mathsf{E}_1}(\mathcal{V}).\]
The former projection $U$ maps an $\mathsf{E}_1$-algebra $A$ in $\mathrm{Alg}_{\mathsf{E}_1}(\mathcal{V})$ to its underlying, say \emph{vertical}, $\mathsf{E}_1$-algebra $U(A):= A_1\in\mathrm{Alg}_{\mathsf{E}_1}(\mathcal{V})$.
The latter projection $\mathrm{Alg}_{\mathsf{E}_1}(U)$ 
maps an $\mathsf{E}_1$-algebra $A$ in $\mathrm{Alg}_{\mathsf{E}_1}(\mathcal{V})$ to the \emph{horizontal} $\mathsf{E}_1$-algebra underlying 
$A$ given pointwise by $\mathrm{Alg}_{\mathsf{E}_1}(U)(A)_n:= A_{n,1}$. Both $U(A)$ and $\mathrm{Alg}_{\mathsf{E}_1}(U)(A)$ hence equip the 
underlying object (that is, the underlying trivial algebra) $A:=U(A)_{1,1}=\mathrm{Alg}_{\mathsf{E}_1}(U)(A)_{1,1}$ in $\mathcal{V}$ with 
an $\mathsf{E}_1$-algebra structure.

In the following we show that the two $\infty$-operadic embeddings (\ref{equ_eh}) are homotopic over $\mathrm{Fin}_{\ast}$ as a 
consequence of the Additivity Theorem. In fact we show more generally that all $(k+1)$-many inclusions
\begin{align*}
s_i\colon(\mathsf{E}_1)^{\otimes_{\mathrm{BV}}^{k}}\rightarrow (\mathsf{E}_1)^{\otimes_{\mathrm{BV}}^{k+1}}\\
\end{align*}
defined as $1^{\otimes^i}\otimes(\mathsf{E}_0\xrightarrow{s_0}\mathsf{E}_1)\otimes 1^{\otimes^{k-i-1}}$ for $0\leq i\leq k$ (and thus 
omitting $\mathsf{E}_1$ in the $(i+1)$-st component) are coherently equivalent over $\mathrm{Fin}_{\ast}$. By the Additivity Theorem,
the canonical $k$-fold bifunctor $\rho\colon (\mathsf{E}_1)^{\otimes_{\mathrm{BV}}^{k}}\rightarrow \mathsf{E}_{k}$ of $\infty$-operads from 
\cite[Definition 2.2.5.3]{lurieha} is an equivalence of $\infty$-operads for every $0\leq k$. It thereby suffices to show that 
all $(k+1)$-many inclusions $s_i\colon \mathsf{E}_k\rightarrow \mathsf{E}_{k+1}$, $1\leq i\leq k+1$ which insert the identity on the 
interval $(-1,1)$ at stage $1\leq i+1\leq k+1$, are coherently equivalent over $\mathrm{Fin}_{\ast}$. To show this, we note that
$\mathsf{E}_{\bullet}$ can be organised into a co-semi-simplicial $\infty$-operad. Indeed, let $\Delta_s\subset\Delta$ be the wide 
subcategory of $\Delta$ spanned by the surjections. Let
\[\mathsf{E}_{\bullet}\colon\Delta_s^{op}\rightarrow\mathrm{Op}_{\infty}\]
be the diagram given by the canonical embeddings
\begin{align}\label{def_E_bullet}
\xymatrix{
\mathsf{E}_0\ar[r]^{s_0} & \mathsf{E}_1\ar@<.5ex>[r]^{s_0}\ar@<-.5ex>[r]_{s_1} & \mathsf{E}_2\ar@<.5ex>@/^/[r]^{s_0}\ar[r]|{s_1}\ar@<-.5ex>@/_/[r]_{s_2} & \mathsf{E}_3\ar@{-->}[r] & \dots
}
\end{align}
where $s_i\colon \mathsf{E}_k\rightarrow \mathsf{E}_{k+1}$ inserts the identity of the interval $(-1,1)$ at stage $i+1$. 

\begin{proposition}\label{prop_eh}
Let $1\leq k<\infty$ and $0\leq i<k$. Then the two $\infty$-operadic embeddings $s_i,s_{i+1}\colon \mathsf{E}_k\rightarrow \mathsf{E}_{k+1}$ are homotopic 
by way of a homotopy $\phi^i_k\colon \mathsf{E}_k\rightarrow P(\mathsf{E}_{k+1})$ in the model category $\mathrm{Op}_{\infty}$ such that
\begin{enumerate}
\item  $\phi^i_k$ is a homotopy relative to the subobject $s_i\colon \mathsf{E}_{k-1}\rightarrow \mathsf{E}_k$:
\[\xymatrix{
 & \mathsf{E}_{k+1}\ar[r]^c & P(\mathsf{E}_{k+1})\ar@{->>}[d]^{(s,t)} \\
\mathsf{E}_{k-1}\ar[r]^{s_i}\ar[ur]^{s_is_i}& \mathsf{E}_{k}\ar[ur]^{\phi^i_k}\ar@<.5ex>[r]^(.4){s_i}\ar@<-.5ex>[r]_(.4){s_{i+1}} & \mathsf{E}_{k+1}\times \mathsf{E}_{k+1};
}\]
\item for all $i+1<j<k+2$ the diagram 
\[\xymatrix{
\mathsf{E}_{k-1}\ar[r]^{s_i}\ar[d]_{s_{j-2}} & \mathsf{E}_{k}\ar[d]_{s_{j-1}}\ar@<.5ex>@/^/[r]^{s_i}\ar@{}[r]|{\phi^i_k}\ar@<-.5ex>@/_/[r]_{s_{i+1}} & \mathsf{E}_{k+1}\ar[d]^{s_j} \\
\mathsf{E}_{k}\ar[r]_{s_i} & \mathsf{E}_{k+1}\ar@<.5ex>@/^/[r]^{s_i}\ar@{}[r]|{\phi^i_{k+1}}\ar@<-.5ex>@/_/[r]_{s_{i+1}} & \mathsf{E}_{k+2} \\
}\]
commutes, i.e.\ $s_j\ast \phi^i_k= \phi^{i}_{k+1}\ast s_{j-1}$.
\end{enumerate}

\end{proposition}
\begin{proof}
We denote the $k$-dimensional cube $(-1,1)^k$ by $\Box^k$. Let us first give the proof for $k=2$ to illustrate. 
Here, on associated topological categories the embeddings
$s_0,s_{1}\colon \mathsf{E}_1\rightarrow \mathsf{E}_2$ are locally simply induced by the two obvious orthogonal embeddings
\[(f\colon\Box^1\rightarrow\Box^1)\mapsto(f,1_{(-1,1)})\colon\Box^{2}\rightarrow\Box^{2}),\]
and 
\[(f\colon\Box^1\rightarrow\Box^{1})\mapsto(1_{(-1,1)},f)\colon\Box^{2}\rightarrow\Box^{2}).\]
Here each map $f\colon (-1,1)\rightarrow (-1,1)$ is an affine map of the form $x\mapsto ax+b$. These two embeddings are naturally homotopic 
by first turning the horizontal rectangle $(ax+b,y)$ to suitably sized squares by way of the homotopy $(ax+b,(1-t(1-a))y+tb)$, $t\in[0,1]$, 
and then turning those to correspondingly sized vertical rectangles by way of the homotopy $((a-t(a-1))x+(1-t)b,y)$. This fixes 
the common boundary $s_0\colon \mathsf{E}_{0}\rightarrow \mathsf{E}_{1}$ given by the identity $(x,y)\mapsto (1x+0,y)$ as stated in Part (1).
The same construction works for any given finite set $S$ and any rectilinear embedding $f\colon\Box^1\times S\rightarrow\Box^1$ as the images 
are rectilinear at all times and do not intersect.

Generally, on associated topological categories the embeddings $s_i,s_{i+1}\colon \mathsf{E}_k\rightarrow \mathsf{E}_{k+1}$ are locally similarly induced by 
the two obvious orthogonal embeddings of $\Box^k$ in $\Box^{k+1}$ induced by the maps
\[(f_1,\dots,f_k)\colon\Box^k\rightarrow\Box^{k})\mapsto(f_1,\dots,f_i,1_{(-1,1)},f_{i+1},\dots, f_k)\colon\Box^{k+1}\rightarrow\Box^{k+1}),\]
and 
\[(f_1,\dots,f_k)\colon\Box^k\rightarrow\Box^{k})\mapsto(f_1,\dots,f_{i+1},1_{(-1,1)},f_{i+2},\dots, f_k)\colon\Box^{k+1}\rightarrow\Box^{k+1}).\]
Here, each $f_i\colon (-1,1)\rightarrow (-1,1)$ is an affine map of the form $x\mapsto a_ix+b_i$. These two embeddings are naturally 
homotopic again by first turning the horizontal rectangles to suitably sized squares, and then turning those to 
correspondingly sized vertical rectangles. This can be done fixing the common boundary
$s_i\colon \mathsf{E}_{k-1}\rightarrow \mathsf{E}_{k}$ as stated in Part (1), and is easy to be seen to be coherent as in Part (2).
\end{proof}

\begin{remark}\label{rem_eh_twist}
The homotopy $\phi_k^i\colon s_i\simeq s_{i+1}$ in Proposition~\ref{prop_eh} can be visualised as folding the 
(co)face $s_i$ of the $k$-dimensional cube onto the (co)face $s_{i+1}$ along the diagonal hypersurface determined the common axis 
$s_{i+1}s_i=s_is_i$. 
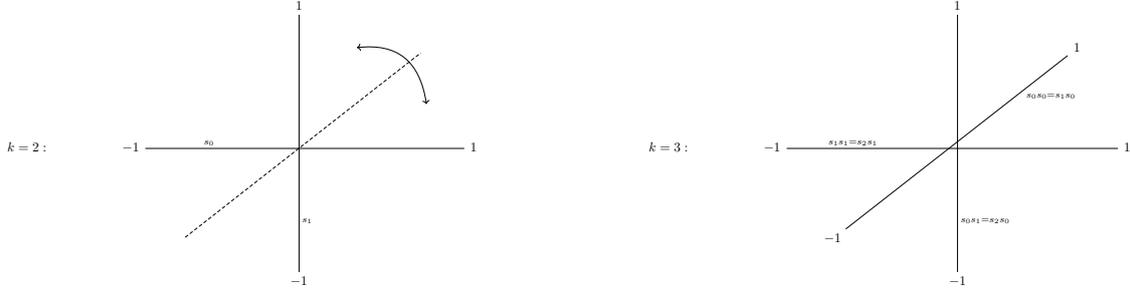
\[
\adjustbox{scale=0.5}{%
\begin{tikzcd}
	&&&&&& {1} &&&&&&&&&&&&&&& {1} \\
	&&&&&&& {} && {} &&&&&&&&&&&&&&& 1 \\
	\\
	&&&&&&&&& {} \\
	{k=2:} && {-1} &&&&&&&& 1 &&&&& {k=3:} && {-1} &&&&&&&& {1} \\
	\\
	\\
	&&& {} &&&&&&&&&&&&&&& {-1} \\
	&&&&&& {-1} &&&&&&&&&&&&&&& {-1}
	\arrow["{s_1}"{pos=0.8}, no head, from=1-7, to=9-7]
	\arrow[curve={height=-24pt}, tail reversed, from=2-8, to=4-10]
	\arrow["{s_0}"{pos=0.2}, no head, from=5-3, to=5-11]
	\arrow["{s_1s_1=s_2s_1}"{pos=0.2}, no head, from=5-18, to=5-26]
	\arrow[dashed, no head, from=8-4, to=2-10]
	\arrow["{s_0s_0=s_1s_0}"'{pos=0.8}, no head, from=8-19, to=2-25]
	\arrow["{s_0s_1=s_2s_0}"'{pos=0.2}, no head, from=9-22, to=1-22]
\end{tikzcd}}\]
The same geometric folding construction can be used to relate the embeddings $s_i,s_j\colon \mathsf{E}_k\rightarrow \mathsf{E}_{k+1}$ for 
all pairs $0\leq i<j\leq k$ directly; this however does generally not yield a relative homotopy as in Proposition~\ref{prop_eh}. For 
instance, folding the embedding $s_0\colon \mathsf{E}_2\rightarrow \mathsf{E}_3$ onto the embedding
$s_2\colon \mathsf{E}_2\rightarrow \mathsf{E}_3$ along the axis $s_2 s_0=s_0 s_1$ gives a homotopy from
$s_0\colon \mathsf{E}_2\rightarrow \mathsf{E}_3$ to $s_2\circ\mathrm{swap}\colon \mathsf{E}_2\rightarrow \mathsf{E}_3$, the latter of which 
maps a pair $(f_1,f_2)$ to the triple $(f_2,f_1,1)$ (instead to $(f_1,f_2,1)$ as $s_2$ does). Proposition~\ref{prop_eh} however does yield 
a composite homotopy $s_0\simeq s_1\simeq s_2$ by folding first along the $s_1s_0=s_0s_0$-axis and then along the $s_2s_1=s_1s_1$-axis.
\end{remark}

\begin{remark}\label{rem_eh_2}
In particular, for every symmetric monoidal $\infty$-category $\mathcal{C}^{\otimes}$ the two forgetful functors
\[\xymatrix{
\mathrm{Alg}_{\mathsf{E}_1}(\mathrm{Alg}_{\mathsf{E}_1}(\mathcal{C}^{\otimes}))\ar@<.5ex>[rr]^U\ar@<-.5ex>[rr]_{\mathrm{Alg}_{\mathsf{E}_1}(U)}\ar@/_/[dr]_U & & \mathrm{Alg}_{\mathsf{E}_1}(\mathcal{C}^{\otimes})\ar@/^/[dl]^U\\
 & \mathcal{C}^{\otimes} & 
}\]
are equivalent over $\mathcal{C}^{\otimes}$ by Proposition~\ref{prop_eh}. This means their underlying vertical and horizontal $\mathsf{E}_1$-algebras 
are equivalent by way of an equivalence that restricts to the identity on the common underlying object. In particular, to give an
$\mathsf{E}_1$-algebra structure on top of an already determined $\mathsf{E}_1$-algebra $A$ in a symmetric monoidal $\infty$-category
$\mathcal{C}^{\otimes}$ is not to give an arbitrary new $\mathsf{E}_1$-algebra structure subject to further coherence laws, but further coherence laws 
directly on the already existing multiplication.
\end{remark}


For later application we record the following technical consequence of Proposition~\ref{prop_eh}. We denote by $L_k(\mathsf{E}_{\bullet})$ 
the $k$-th latching object of the co-semi-simplicial object $\mathsf{E}_{\bullet}$ in (\ref{def_E_bullet}).

\begin{remark}\label{rem_latch}
Let $X\colon\Delta_s^{op}\rightarrow\mathcal{C}$ be a co-semi-simplicial object in a cocomplete category $\mathcal{C}$.
For $0\leq n<\infty$ we may consider the shifted complex $X_{\bullet +n}\colon\Delta^{op}_s\rightarrow\mathcal{C}$ given by
\[\xymatrix{
X_n\ar[r]^{s_0} & X_{n+1}\ar@<.5ex>[r]^{s_0}\ar@<-.5ex>[r]_{s_1} & X_{n+2}\ar@<.5ex>@/^/[r]^{s_0}\ar[r]|{s_1}\ar@<-.5ex>@/_/[r]_{s_2} & X_{n+3}\ar@{-->}[r] & \dots
}\]
It comes together with a natural transformation $s_{\bullet}\colon X_{\bullet +n}\rightarrow X_{\bullet +n+1}$ given at $0\leq k<\infty$ by 
$s_k\colon X_{k+n}\rightarrow X_{k+n+1}$. Then the latching object $L_{k+1}(X)\rightarrow X_{k+1}$ of $X$ can be computed as the induced
co-gap map
\[L_k(X_{\bullet +1})\sqcup_{L_k(X)}X_k\rightarrow (X_{\bullet +1})_k.\]
\end{remark}

\begin{corollary}\label{cor_E_latching}
For every $0\leq k<\infty$ the canonical embedding
\[\xymatrix{
\mathsf{E}_k\ar[d]_{\eta}\ar[dr]^{s_k} & \\
L_{k+1}(\mathsf{E}_{\bullet})\ar[r] & \mathsf{E}_{k+1}
}\]
has a retract over $\mathsf{E}_{k+1}$ in the $\infty$-category of $\infty$-operads.
\end{corollary}
\begin{proof}
We work in the model category $\mathrm{Op}_{\infty}$. We construct simultaneously for all $0\leq k<\infty$ and all $0\leq n<\infty$ 
a homotopy-commutative triangle
\begin{align}\label{diag_E_latching_1}
\begin{gathered}
\xymatrix{
L_{k+1}(\mathsf{E}_{\bullet +n})\ar[rr]^{r_k^n}\ar[dr]_{}  & \ar@{}[d]|{h_k^n} & (\mathsf{E}_{\bullet +n})_k\ar[dl]^{s_k} \\
 & (\mathsf{E}_{\bullet +n})_{k+1} &
}
\end{gathered}
\end{align}
such that $h_k^n\ast\eta=1_{s_k}$ over $(\mathsf{E}_{\bullet +n})_{k+1}$ and such that there is a 3-cell
$\psi_k^n\colon s_{k+1}\ast h_k^n\simeq h_k^{n+1}\ast L_{k+1}(s_{\bullet})$ which restricts along
$\eta\colon (\mathsf{E}_{\bullet +n})_k\rightarrow L_{k+1}(\mathsf{E}_{\bullet +n})$ to the identity on the square $s_{k+1}s_k=s_ks_k$. This in particular 
proves the statement that is to show for $n=0$. We proceed by induction on $k$. If $k=0$ then
$\eta\colon (\mathsf{E}_{\bullet +n})_0\rightarrow L_1(\mathsf{E}_{\bullet +n})$ is the identity (over $(\mathsf{E}_{\bullet +n})_1$) and so we may pick $r_0^n$ the 
identity and $h_0^n$ the identity as well. The diagram
\[\xymatrix{
L_1(\mathsf{E}_{\bullet +n})\ar[rr]^{r_0^n}\ar[dr]_{s_0}\ar[dd]_{s_0} &\ar@{}[d]|{h_0^n} & (\mathsf{E}_{\bullet +n})_0\ar[dl]^{s_0}\ar[dd]^{s_0} \\
 & (\mathsf{E}_{\bullet +n})_1\ar[dd]^(.3){s_1} & \\
L_1(\mathsf{E}_{\bullet +n+1})\ar[rr]|(.51)\hole^(.35){r_0^{n+1}}\ar[dr]_{} & \ar@{}[d]_{h_0^{n+1}} & (\mathsf{E}_{\bullet +n+1})_0\ar[dl]^{s_0} \\
 & (\mathsf{E}_{\bullet +n+1})_1 & \\
}\]
then commutes trivially. Suppose we have shown the statement for a given $k$ and all $n\geq 0$. By way of Remark~\ref{rem_latch} and the 
inductive assumption we obtain a diagram
\[\xymatrix{
 & (\mathsf{E}_{\bullet +n})_k\ar[rr]^{s_k}\ar[dd]|\hole_(.3){s_k}\ar[dl]^{\eta}\ar@{}[dr]|{h_k^{n}} & & (\mathsf{E}_{\bullet +n})_{k+1}\ar[dd]^{s_{k+1}}\ar@{=}[dl] \\
L_{k+1}(\mathsf{E}_{\bullet +n})\ar[dd]_{L_{k+1}(s_{\bullet})}\ar[rr]\ar@/^1pc/[ur]^{r_k^n} & & (\mathsf{E}_{\bullet +n})_{k+1}\ar[dd]^(.3){s_{k+1}} & \\
 & (\mathsf{E}_{\bullet +n+1})_k\ar[rr]|\hole^(.3){s_k}\ar[dl]^{\eta}\ar@{}[dr]|{h_k^{n+1}} & & \mathsf{E}_{k+n+2} \\
L_{k+1}(\mathsf{E}_{\bullet +n+1})\ar[rr]\ar@/^1pc/[ur]^{r_k^{n+1}} & & \mathsf{E}_{k+n+2}.\ar@{=}[ur] &
}\]
that commutes up to homotopy by way of the 3-cell $\psi_k^n$. This gives rise to a morphism together with a homotopy
\begin{align}\label{diag_E_latching_2}
\begin{gathered}
\xymatrix{
L_{k+2}(\mathsf{E}_{\bullet +n})\ar[rr]^{\overline{r}_{k+1}^n}\ar[dr] &\ar@{}[d]|{\overline{h}_{k+1}^n:=(h_k^{n+1},1)} & \mathsf{E}_{k+n+1}\sqcup_{\mathsf{E}_{k+n}}\mathsf{E}_{k+n+1}\ar[dl]^{(s_k,s_{k+1})}\\
 & \mathsf{E}_{k+n+2} & 
}
\end{gathered}
\end{align}
between the two corresponding (homotopy)-pushouts such that the pasting
\[\xymatrix{
\mathsf{E}_{k+n+1}\sqcup_{\mathsf{E}_{k+n}}\mathsf{E}_{k+n+1}\ar[dr]_{(s_k,s_{k+1})}\ar[r]^(.6){(\eta,1)} & L_{k+2}(\mathsf{E}_{\bullet +n})\ar[r]^(.4){\overline{r}_{k+1}^n}\ar[d] \ar@{}[dr]|(.3){\overline{h}_{k+1}^n} & \mathsf{E}_{k+n+1}\sqcup_{\mathsf{E}_{k+n}}\mathsf{E}_{k+n+1}\ar[dl]^{(s_k,s_{k+1})} \\
 & \mathsf{E}_{k+n+2} & 
}\]
is the identity over $\mathsf{E}_{k+n+2}$. By Proposition~\ref{prop_eh} there is a homotopy 
\begin{align}\label{diag_E_latching_3}
\begin{gathered}
\xymatrix{
\mathsf{E}_{k+n+1}\sqcup_{\mathsf{E}_{k+n}}\mathsf{E}_{k+n+1}\ar[rr]^{\nabla}\ar[dr]_{(s_k,s_{k+1})} & \ar@{}[d]|{(\phi_{k+n+1}^k,1)} & \mathsf{E}_{k+n+1}\ar[dl]^{s_{k+1}}\\
 & \mathsf{E}_{k+n+2} & 
}
\end{gathered}
\end{align}
such that the pasting $(\phi_{k+n+1}^k,1)\ast\iota_R\colon s_{k+1}\rightarrow s_{k+1}$ is the identity over $\mathsf{E}_{k+n+2}$. Concatenating the two triangles (\ref{diag_E_latching_2}) and (\ref{diag_E_latching_3}) yields a morphism
\[r_{k+1}^n:=\nabla\circ\overline{r}_{k+1}^n\] and a homotopy
\[h_{k+1}^n:=(\phi_{k+n+1}^k,1)\ast\overline{h}_{k+1}^n\]
as required in (\ref{diag_E_latching_1}). To end the proof we are left to show that
$s_{k+2}\ast h_{k+1}^n\simeq h_{k+1}^{n+1}\ast L_{k+2}(s_{\bullet})$ relative to $(\mathsf{E}_{\bullet +n})_{k+1})$ for all $0\leq n$. This follows
from the definition of $h_{k+1}^n$ together with Proposition~\ref{prop_eh}.2 and the inductive assumption on the existence of $\psi_k^n$.

%
\end{proof}

\subsection{Fragments of associativity and unitality}\label{sec_sub_assoc}
In Section~\ref{sec_sub_comm} we recalled that
\[\mathrm{Alg}_{\mathsf{E}_{\infty}}(\mathcal{C}^{\otimes})\simeq\mathrm{lim}_{n<\infty}\mathrm{Alg}_{\mathsf{E}_{n}}(\mathcal{C}^{\otimes}).\]
Similarly, the $\infty$-category of $\mathsf{E}_1$-algebras in a symmetric monoidal 
$\infty$-category $\mathcal{C}^{\otimes}$ itself can be expressed as the limit of a tower 
\[\mathrm{Alg}_{\mathsf{E}_1}(\mathcal{C}^{\otimes})\rightarrow\dots\rightarrow\mathrm{Alg}_{A_{m+1}}(\mathcal{C}^{\otimes})\rightarrow\mathrm{Alg}_{A_{m}}(\mathcal{C}^{\otimes})\rightarrow\dots\rightarrow\mathrm{Alg}_{A_1}(\mathcal{C}^{\otimes})\]
of forgetful functors. This is \cite[Proposition 4.1.1.9]{lurieha} for the purely associative, non-unital case; the associative and unital 
case is in fact more straightforward. In a vein similar to that of the $\mathsf{E}_k$-hierarchy, $\mathsf{A}_m$-algebras are objects 
together with a multiplication map that satisfies a degree of homotopy-coherent associativity and unitality. In this sense,
$\mathsf{E}_1$-algebras are at times referred to as $\mathsf{A}_{\infty}$-algebras
\cite[Definition 4.1.3.16, Proposition 4.1.3.19]{lurieha}.

\begin{remark}
More precisely, $\mathsf{A}_{\infty}$ is a planar $\infty$-operad, and hence lives in the non-symmetric context; $\mathsf{E}_1$ is the 
symmetrisation of $\mathsf{A}_{\infty}$. It hence does not make formal sense to talk about $\mathsf{E}_1$-algebras in non-symmetric 
monoidal $\infty$-categories.
\end{remark}

To isolate the study of associativity from that of unitality, Lurie introduces the notion of \emph{non-unital} $\mathsf{A}_m$-algebras. We 
therefore recall the $\infty$-category $\mathrm{Alg}_{\mathsf{A}_m}^{\text{nu}}(\mathcal{C}^{\otimes})$ of non-unital
$\mathsf{A}_m$-algebras in a monoidal $\infty$-category $\mathcal{C}^{\otimes}$ \cite[Section 4.1.4]{lurieha}. This is useful because it 
allows to construct $\mathsf{A}_{\infty}$-algebra structures on an object $A\in\mathcal{C}^{\otimes}$ in two consecutive steps --- rather 
than having to do both at each level at the same time. First, one may construct non-unital $\mathsf{A}_m$-structures for all $m\geq 0$ on
$A\in\mathcal{C}^{\otimes}$ that recursively extend one another.\footnote{In fact, we will do so by using unitality in low degrees without 
extending unitality directly to higher degrees itself.}
This gives rise to a non-unital $\mathsf{A}_{\infty}$-algebra structure on $A$ in $\mathcal{C}^{\otimes}$. Second, to show that this
non-unital $\mathsf{A}_{\infty}$-algebra can be extended to an actual $\mathsf{A}_{\infty}$-algebra, it suffices to show that it is
quasi-unital \cite[Theorem 5.4.3.5]{lurieha}. We recall that a non-unital $\mathsf{A}_{\infty}$-algebra $A$ in $\mathcal{C}^{\otimes}$ is 
\emph{quasi-unital} if its derived algebra $A\in h(\mathcal{C}^{\otimes})$ in the monoidal homotopy category of $\mathcal{C}^{\otimes}$ is 
unital \cite[Definition 5.4.3.1]{lurieha}. 

\begin{definition}\label{def_qu_hom}
Let $\mathcal{C}^{\otimes}$ be a monoidal $\infty$-category, and let $m\colon A\otimes A\rightarrow A$
be a non-unital $\mathsf{A}_2$-algebra in $\mathcal{C}^{\otimes}$ with quasi-unit $u\colon I\rightarrow A$. A morphism
$f\colon A\rightarrow B$ of non-unital $\mathsf{A}_2$-algebras in $\mathcal{C}^{\otimes}$ is quasi-unital if $fu\colon I\rightarrow B$ is a 
quasi-unit for $B$. We define	
$\mathrm{Alg}_{\mathsf{A}_2}^{\mathrm{qu}}(\mathcal{C}^{\otimes})\subset\mathrm{Alg}_{\mathsf{A}_2}^{\mathrm{nu}}(\mathcal{C}^{\otimes})$ 
the subcategory of quasi-unital $\mathsf{A}_2$-algebras in $\mathcal{C}^{\otimes}$ spanned by the quasi-unital $\mathsf{A}_2$-algebras and 
quasi-unital $\mathsf{A}_2$-algebra morphisms between them.
\end{definition}

\begin{remark}\label{rem_qu}
By definition, quasi-unitality of a non-unital $\mathsf{A}_{\infty}$-algebra reduces to quasi-unitality of its underlying non-unital
$\mathsf{A}_2$-algebra. The same applies to morphisms between such. This defines a subcategory
\[\xymatrix{
\mathrm{Alg}_{\mathsf{A}_m}^{\mathrm{qu}}(\mathcal{C}^{\otimes})\ar@{^(->}[r]\ar[d]\ar@{}[dr]|(.3){\pbs} & \mathrm{Alg}_{\mathsf{A}_m}^{\mathrm{nu}}(\mathcal{C}^{\otimes})\ar[d]^U \\
\mathrm{Alg}_{\mathsf{A}_2}^{\mathrm{qu}}(\mathcal{C}^{\otimes})\ar@{^(->}[r] & \mathrm{Alg}_{\mathsf{A}_2}^{\mathrm{nu}}(\mathcal{C}^{\otimes})
}\]
for all $2\leq m\leq\infty$. It is easy to see that the inclusion
$\mathrm{Alg}_{\mathsf{A}_2}^{\mathrm{qu}}(\mathcal{C}^{\otimes})\hookrightarrow\mathrm{Alg}_{\mathsf{A}_2}^{\mathrm{nu}}(\mathcal{C}^{\otimes})$ is replete 
and hence a $(-1)$-truncated fibration. Therefore, the inclusion
$\mathrm{Alg}_{\mathsf{A}_m}^{\mathrm{qu}}(\mathcal{C}^{\otimes})\hookrightarrow\mathrm{Alg}_{\mathsf{A}_m}^{\mathrm{nu}}(\mathcal{C}^{\otimes})$ is a $(-1)$-truncated fibration as well.
\end{remark}

We further recall that the $\infty$-category $\mathrm{Alg}_{\mathsf{A}_m}^{\mathrm{nu}}(\mathcal{C}^{\otimes})$ also is the
$\infty$-category of algebras in $\mathcal{C}^{\otimes}$ for an associated (non-unital) $\infty$-operad $\mathsf{A}_m^{\mathrm{nu}}$ 
\cite[Remark 4.1.4.8]{lurieha}. In particular, for any $m\geq 1$ the $\infty$-category
$\mathrm{Alg}_{\mathsf{A}_m}^{\mathrm{nu}}(\mathcal{C}^{\otimes})$ comes equipped with a symmetric monoidal structure such that the 
forgetful functor $U\colon\mathrm{Alg}_{\mathsf{A}_m}^{\mathrm{nu}}(\mathcal{C}^{\otimes})\rightarrow\mathcal{C}^{\otimes}$ that maps an 
algebra $A$ to its underlying object $U(A):=\mathsf{A}_1$ is symmetric monoidal \cite[Example 3.2.4.4]{lurieha}.

\begin{notation}
For convenience, we let $\mathsf{E}_0^{\mathrm{nu}}:=\mathrm{Triv}^{\otimes}$ be the trivial $\infty$-operad 
\cite[Example 2.1.1.20]{lurieha}, and $\mathsf{E}_k^{\mathrm{nu}}:=(\mathsf{A}_{\infty}^{\mathrm{nu}})^{\otimes_{\mathrm{BV}}^k}$ for
$1\leq k<\infty$. Here, $\mathsf{A}_{\infty}^{\mathrm{nu}}$ denotes the $\infty$-operad of non-unital $\mathsf{A}_{\infty}$-algebras of
\cite[Remark 4.1.4.8]{lurieha}. Whenever $\mathcal{C}^{\otimes}$ is a symmetric monoidal $\infty$-category, we set
$\mathrm{Alg}_{\mathsf{E}_{k}}^{\mathrm{nu}}(\mathcal{C}^{\otimes}):=\mathrm{Alg}_{\mathsf{E}_{k}^{\mathrm{nu}}}(\mathcal{C}^{\otimes})$.
We denote by $U_i\colon\mathrm{Alg}_{\mathsf{E}_{k}}^{\mathrm{nu}}(\mathcal{C}^{\otimes})\rightarrow\mathrm{Alg}_{\mathsf{E}_{i}}^{\mathrm{nu}}(\mathcal{C}^{\otimes})$ the obvious forgetful functor that forgets the first $(k-i)$-many components.
\end{notation}

\begin{remark}
Whenever $\mathcal{C}^{\otimes}$ is a symmetric monoidal $\infty$-category, the symmetric monoidal $\infty$-category
$\mathrm{Alg}_{\mathsf{E}_k^{\mathrm{nu}}}(\mathcal{C}^{\otimes})$ is probably not equivalent to the $\infty$-category
$\mathrm{Alg}_{\mathsf{E}_k}^{\mathrm{nu}}(\mathcal{C}^{\otimes})$ of non-unital $\mathsf{E}_k$-algebras in $\mathcal{C}^{\otimes}$ in the 
sense of \cite[Section 4.5.5]{lurieha}. The author however has not verified this.
\end{remark}

\begin{remark}
The ability to reduce (higher dimensional) unitality to (1-dimensional) quasi-unitality appears to be a 
specific property of fully homotopy-coherent associativity (in terms of $\mathsf{A}_{\infty}$ opposed to the finite $\mathsf{A}_m$'s). 
Whether a similar suitable reduction of $\mathsf{A}_m$-structures to quasi-unital $\mathsf{A}_m$-structures holds is not clear to the 
author. That means, although we may know from the start that a given $\mathsf{A}_m$-algebra has a quasi-unit, we still have to work with it 
essentially non-unitally until we know that we may extend this structure to a full (non-unital) $\mathsf{A}_{\infty}$-algebra structure.
\end{remark}

The following lemma will be very useful to move between the unital and non-unital context.

\begin{lemma}\label{lemma_unitff}
Let $\mathcal{C}^{\otimes}$ be a symmetric monoidal $\infty$-category. For all finite integers $k\geq 1$, the following hold.
\begin{enumerate}
\item The forgetful functor
\[\mathrm{Alg}_{\mathsf{E}_k}(\mathcal{C}^{\otimes})\rightarrow\mathrm{Alg}_{\mathsf{E}_k}^{\mathrm{nu}}(\mathcal{C}^{\otimes})\]
is monic. 
\item The forgetful functor
\begin{align}\label{equ_unitff}
\mathrm{Alg}_{\mathsf{E}_1}(\mathrm{Alg}_{\mathsf{E}_k}(\mathcal{C}^{\otimes}))\rightarrow\mathrm{Alg}_{\mathsf{A}_{\infty}}^{\mathrm{nu}}(\mathrm{Alg}_{\mathsf{E}_k}(\mathcal{C}^{\otimes}))
\end{align}
is fully faithful.
\item Let $A$, $B$ be $\mathsf{E}_k$-algebras in $\mathcal{C}^{\otimes}$, and $f\colon A\rightarrow B$ be a 
morphism of underlying non-unital $\mathsf{E}_k$-algebras. Then $f$ lifts to a morphism of $\mathsf{E}_k$-algebras whenever its underlying 
morphism $f\colon U(A)\rightarrow U(B)$ of $\mathsf{E}_1$-algebras in $\mathcal{C}^{\otimes}$ is
quasi-unital.
\item Let $A$ be a non-unital $\mathsf{E}_k$-algebra in $\mathcal{C}^{\otimes}$. Then $A$ lifts to an
$\mathsf{E}_k$-algebra whenever $U_1(A)\in\mathrm{Alg}_{\mathsf{E}_{1}}^{\mathrm{nu}}(\mathcal{C}^{\otimes})$ is quasi-unital, and for 
every $i\leq k$ its $i$-th multiplication $m_i\colon U_i(A)\otimes U_i(A)\rightarrow U_i(A)$ is quasi-unital in
$\mathrm{Alg}_{E_i}^{\mathrm{nu}}(\mathcal{C}^{\otimes})$ and subsequently exhibits a quasi-unit in
$\mathrm{Alg}_{E_i}(\mathcal{C}^{\otimes})$itself.
\end{enumerate}
\end{lemma}
\begin{proof}
Part (1) again from a repeated application of \cite[Corollary 5.4.3.6]{lurieha} and the fact that
$\mathrm{Alg}_{\mathcal{O}}(-)$ preserves limits for any $\infty$-operad $\mathcal{O}$.

For Part (2), by the Additivity Theorem, we may without loss of generality assume that $k=1$. The 
functor (\ref{equ_unitff}) is monic by Part (1), and so we are only left to show that it is full. Thus, let $A$, $B$ be
$\mathsf{E}_1$-algebras in $\mathrm{Alg}_{\mathsf{E}_1}(\mathcal{C}^{\otimes})$ and let $f\colon A\rightarrow B$ be a morphism of their 
underlying non-unital $\mathsf{A}_{\infty}$-algebras. Let $u_A\colon 1\rightarrow A$ and $u_B\colon 1\rightarrow B$ be quasi-units of $A$ 
and $B$, respectively. To show that $f$ lifts to a morphism in
$\mathrm{Alg}_{\mathsf{E}_1}(\mathrm{Alg}_{\mathsf{E}_1}(\mathcal{C}^{\otimes}))$, by \cite[Theorem 5.4.3.5]{lurieha} we are to provide an 
equivalence $fu_A\simeq u_B$ in $\mathrm{Alg}_{\mathsf{E}_1}(\mathcal{C}^{\otimes})$. But the unit
$1\in\mathrm{Alg}_{\mathsf{E}_1}(\mathcal{C}^{\otimes})$ is an initial object \cite[Proposition 3.2.1.8]{lurieha}, which provides such an 
equivalence trivially.

Part (3) follows from Part (2). Part (4) is also a straightforward proof by induction using the former parts as well as the fact that 
if both the underlying object $A$ and the multiplication $m\colon A\otimes A\rightarrow A$ of a non-unital $\mathsf{E}_1$-algebra in
$\mathrm{Alg}_{E_1}^{\mathrm{nu}}(\mathcal{C}^{\otimes})$ are quasi-unital, then so is the entire $\mathsf{E}_1$-algebra.
\end{proof}

\subsection{Diagonal reflections of $\mathsf{E}_k$-structures}\label{sec_sub_revmon_infty}

Given a triple $\mathcal{O}$, $\mathcal{V}$ and $\mathcal{W}$ of $\infty$-operads, a 
bifunctor of $\infty$-operads from $(\mathcal{O},\mathcal{V})$ to $\mathcal{W}$ is a morphism
$f\colon\mathcal{O}\times\mathcal{V}\rightarrow\mathcal{W}$ of simplicial sets such that the square
\[\xymatrix{
\mathcal{O}\times\mathcal{V}\ar[r]^f\ar[d] & \mathcal{W}\ar@{->>}[d] \\
\mathrm{Fin}_{\ast}\times\mathrm{Fin}_{\ast}\ar[r]^(.6){\wedge} & \mathrm{Fin}_{\ast}
}\]
commutes, and such that $f$ maps pairs of inert morphisms in $\mathcal{O}\times\mathcal{V}$ to an inert morphism in $\mathcal{W}$ 
\cite[Section 2.2.5]{lurieha}. Here the bottom horizontal functor again denotes the smash product on the category of finite pointed sets. 
The Boardman--Vogt tensor product $\mathcal{O}\otimes_{\mathrm{BV}}\mathcal{V}$ from Example~\ref{exple_bv} is defined as a 
fibrant replacement 
\[\xymatrix{
\mathcal{O}\times\mathcal{V}\ar@{=}[r]\ar[d] & \mathcal{O}\times\mathcal{V}\ar[d]\ar@{^(->}[r]^{\sim} & \mathcal{O}\otimes_{\mathrm{BV}}
\mathcal{V}\ar@{->>}@/^1pc/[dl] \\
\mathrm{Fin}_{\ast}\times\mathrm{Fin}_{\ast}\ar[r]^{\wedge} & \mathrm{Fin}_{\ast}
}\]
of the composite morphism $\mathcal{O}\times\mathcal{V}\rightarrow\mathrm{Fin}_{\ast}\times\mathrm{Fin}_{\ast}\xrightarrow{\wedge}\mathrm{Fin}_{\ast}$ in the model category $\mathrm{Op}_{\infty}$. As such, it corepresents bifunctors of
$\infty$-operads out of the pair $(\mathcal{O},\mathcal{V})$ in the $\infty$-category of $\infty$-operads.
We recall that the smash product on the category of finite pointed sets is symmetric as well, albeit 
not strictly so. Adopting the notation from \cite[Notation 2.0.0.2]{lurieha}, the natural isomorphism
\[\xymatrix{
& \mathrm{Fin}_{\ast}\times\mathrm{Fin}_{\ast}\ar[dr]^{\wedge}\ar@{}[d]|{\Uparrow \sigma} & \\
\mathrm{Fin}_{\ast}\times\mathrm{Fin}_{\ast}\ar[ur]^{\mathrm{swap}}_{\cong}\ar[rr]_{\wedge} & & \mathrm{Fin}_{\ast}
}\]
is pointwise given for a pair $(\langle n\rangle,\langle m\rangle)$ by the composition
\begin{align}\label{def_sigma}
\langle nm\rangle\xleftarrow{\simeq}(\langle n\rangle^{\circ}\times \langle m\rangle^{\circ})_{\ast}\xrightarrow{\mathrm{swap}_{\ast}}(\langle m\rangle^{\circ}\times \langle n\rangle^{\circ})_{\ast}\xrightarrow{\simeq}\langle mn\rangle.
\end{align}

It follows that the Boardman--Vogt tensor product is symmetric as well: Given $\infty$-operads $\mathcal{O}$ and $\mathcal{V}$, the 
idempotent swap isomorphism $\mathrm{swap}\colon\mathcal{O}\times\mathcal{V}\xrightarrow{\simeq}\mathcal{V}\times\mathcal{O}$ induces a 
natural idempotent equivalence
\begin{align*}
(-)^{\rho}\colon\mathcal{O}\otimes_{\mathrm{BV}}\mathcal{V}\xrightarrow{\simeq}\mathcal{V}\otimes_{\mathrm{BV}}\mathcal{O}
\end{align*}
of $\infty$-operads by lifting the natural isomorphism $\sigma$ along
$\mathcal{V}\otimes_{\mathrm{BV}}\mathcal{O}\twoheadrightarrow\mathrm{Fin}_{\ast}$ \cite[Proposition 2.2.5.13]{lurieha}.

In particular, for $\mathcal{V}=\mathcal{O}$ this yields an automorphism
$(-)^{\rho}\colon\mathcal{O}\otimes_{\mathrm{BV}}\mathcal{O}\xrightarrow{\simeq}\mathcal{O}\otimes_{\mathrm{BV}}\mathcal{O}$. For instance, 
in the special case $\mathcal{O}=\mathsf{E}_1$ we obtain a natural idempotent but non-trivial equivalence
\[(-)^{\rho}\colon\mathrm{Alg}_{\mathsf{E}_1}(\mathrm{Alg}_{\mathsf{E}_1}(\mathcal{C}^{\otimes}))\rightarrow\mathrm{Alg}_{\mathsf{E}_1}(\mathrm{Alg}_{\mathsf{E}_1}(\mathcal{C}^{\otimes}))\]
for every symmetric monoidal $\infty$-category $\mathcal{C}^{\otimes}$. By way of the Additivity Theorem, this equivalently
gives rise to a natural idempotent but non-trivial equivalence
\[(-)^{\rho}\colon\mathrm{Alg}_{\mathsf{E}_2}(\mathcal{C}^{\otimes})\rightarrow\mathrm{Alg}_{\mathsf{E}_2}(\mathcal{C}^{\otimes}).\]

\begin{definition}
We will refer to the $\mathsf{E}_2$-algebra $A^{\rho}$ associated to an $\mathsf{E}_2$-algebra $A$ in $\mathcal{C}^{\otimes}$ as the 
\emph{diagonal reflection} of $A$. 

More generally, for every $k\geq 2$ and every element $s\in\Sigma_k$ in the symmetric group of order $k$, the component permutation
$s^{\ast}\colon \mathsf{E}_1^{k}\rightarrow \mathsf{E}_1^{k}$ induces an idempotent non-trivial automorphism
\[(-)^{\rho(s)}\colon\mathrm{Alg}_{\mathsf{E}_k}(\mathcal{C}^{\otimes})\rightarrow\mathrm{Alg}_{\mathsf{E}_k}(\mathcal{C}^{\otimes}).\]
\end{definition}

By definition, the $s$-diagonal reflection $A^{\rho(s)}$ of an $\mathsf{E}_k$-algebra $A$ in a symmetric monoidal $\infty$-category
$\mathcal{C}^{\otimes}$ is the essentially unique $\mathsf{E}_k$-algebra whose underlying $k$-fold bifunctor
$\mathsf{E}_1^k\rightarrow\mathcal{C}^{\otimes}$ is given by the composition
\[\xymatrix{
\mathsf{E}_1^k\ar[r]^{s^{\ast}} & \mathsf{E}_1^k\ar@{^(->}[r]^{\sim} & \mathsf{E}_k\ar[r]^A & \mathcal{C}^{\otimes}. 
}\]

\begin{remark}\label{rem_diagreflgeo}
In geometric terms, the equivalence $(-)^{\rho}\colon \mathsf{E}_2\xrightarrow{\simeq} \mathsf{E}_2$ locally flips the $x$ and the $y$ axis 
of the 2-dimensional unit cube (twice, inside and outside). To see this, let
$\mu\colon (-1,1)^2\rightarrow (-1,1)^2$, $(x,y)\mapsto (y,x)$ be the reflection at the $(x=y)$-axis. Define
$\mu\colon \colon \mathsf{E}_2\xrightarrow{\simeq} \mathsf{E}_2$ --- as a functor of topological 
categories --- to be the identity-on-objects functor that on $\mathsf{E}_2(\langle m\rangle,\langle n\rangle)$ acts by conjugation with
$\mu$. That means,
\[\mu\colon \coprod_{f\colon\langle m\rangle\rightarrow\langle n\rangle}\prod_{1\leq j\leq n}\mathrm{Rect}(\Box^2\times f^{-1}\{j\},\Box^2)\rightarrow \coprod_{f\colon\langle m\rangle\rightarrow\langle n\rangle}\prod_{1\leq j\leq n}\mathrm{Rect}(\Box^2\times f^{-1}\{j\},\Box^2)
\]
is induced in the obvious way by the automorphism
$\mathrm{Rect}(\Box^2\times S,\Box^2)\rightarrow\mathrm{Rect}(\Box^2\times S,\Box^2)$ for any set $S$ that takes a map $g$ to
$\mu g (\mu\times 1)$.
We obtain an outer square
\[\xymatrix{
 & \mathsf{E}_1\times \mathsf{E}_1\ar@{^(->}[r]^{\sim} & \mathsf{E}_1\otimes_{\mathrm{BV}} \mathsf{E}_1\ar@{^(->}[r]^(.7){\sim} & \mathsf{E}_2 \\
\mathsf{E}_1\times \mathsf{E}_1\ar@{^(->}[r]^{\sim}\ar[ur]^{\mathrm{swap}}_(.6){\cong} & \mathsf{E}_1\otimes_{\mathrm{BV}} \mathsf{E}_1\ar@{^(->}[r]^{\sim} & \mathsf{E}_2\ar[ur]_{\mu} & 
}\]
that commutes up to the obvious natural isomorphism of topological categories lifting $\sigma$. That is, the natural isomorphism given 
by the pair $\sigma$ --- rearranging the indices --- and the identity on $(-1,1)^2$. This means that the induced square
\[\xymatrix{
 & \mathsf{E}_1\otimes_{\mathrm{BV}} \mathsf{E}_1\ar@{^(->}[r]^(.6){\sim} & \mathsf{E}_2 \\
\mathsf{E}_1\otimes_{\mathrm{BV}} \mathsf{E}_1\ar@{^(->}[r]^(.6){\sim}\ar[ur]^{(-)^{\rho}}  & \mathsf{E}_2\ar[ur]_{\mu} & 
}\]
commutes up to equivalence by the universal property of the Boardman--Vogt tensor product. 	
\end{remark}

In Section~\ref{sec_sub_graymonalg} we will define the reverse of a braiding on a monoidal bicategory. In 
Section~\ref{sec_sub_translation2} we will see that the $\mathsf{E}_2$-structure induced by the reverse braiding is precisely the diagonal 
reflection of the $\mathsf{E}_2$-structure induced by the braiding itself. In Section~\ref{sec_sub_translation3} we will see something 
similar for syllepses. In this sense, diagonal reflection is a generalisation of reversing low dimensional symmetries (without reversing 
the monoidal structure itself). 

\begin{warning}
We note that the diagonal reflection of an $\mathsf{E}_k$-algebra is different from the reverse of an $\mathsf{E}_k$-algebra defined in
\cite [Construction 5.2.5.18]{lurieha}. The latter assigns to an $\mathsf{E}_1$-algebra in
$\mathrm{Alg}_{\mathsf{E}_{k-1}}(\mathcal{C}^{\otimes})$ the reverse $\mathsf{E}_1$-structure in
$\mathrm{Alg}_{\mathsf{E}_{k-1}}(\mathcal{C}^{\otimes})$ as defined in \cite[Remark 4.1.1.7]{lurieha}. In particular, it assigns to an
$\mathsf{E}_1$-algebra $A$ in $\mathcal{C}^{\otimes}$ with multiplication $a\otimes b$ the reverse
$\mathsf{E}_1$-algebra $A^{\mathrm{rev}}$ with multiplication $b\otimes a$. Diagonal reflection however is not even defined for
$\mathsf{E}_1$-algebras in general symmetric monoidal $\infty$-categories. Nevertheless, $A^{\mathrm{rev}}$ also has a geometric 
presentation: The automorphism $(-)^{\mathrm{rev}}:\mathsf{E}_k\rightarrow\mathsf{E}_k$ is induced by the reflection of the $k$-cube at the 
$x_k$-axis.
\end{warning}

It will be useful to have a concrete model of diagonal reflection in low dimensional cases where such concrete models can be constructed by 
hand. Therefore, we note that diagonal reflection then comes together with a lift
\[\xymatrix{
 & \mathsf{E}_1\times \mathsf{E}_1\ar@{^(->}[rr]^{\sim}\ar[dd]|\hole & & \mathsf{E}_1\otimes_{\mathrm{BV}} \mathsf{E}_1\ar@/^1pc/@{->>}[dddl]\\
\mathsf{E}_1\times \mathsf{E}_1\ar@{^(->}[rr]^(.4){\sim}\ar[ur]^{\mathrm{swap}}_{\cong}\ar[dd]\ar@{}[urrr]|{\bar{\sigma}} & & \mathsf{E}_1\otimes_{\mathrm{BV}} \mathsf{E}_1\ar[ur]_{(-)^{\rho}}\ar@{->>}[dd] & \\
 & \mathrm{Fin}_{\ast}\times\mathrm{Fin}_{\ast}\ar[dr]^{\wedge} & & \\
 \mathrm{Fin}_{\ast}\times\mathrm{Fin}_{\ast}\ar[ur]^{\mathrm{swap}}_{\cong}\ar[rr]_{\wedge}\ar@{}[urr]|{\Uparrow\sigma} & & \mathrm{Fin}_{\ast} &
}\]
of the natural isomorphism $\sigma$. In particular, for every symmetric monoidal $\infty$-category $\mathcal{C}^{\otimes}$ and every
$\mathsf{E}_2$-algebra $A$ in $\mathcal{C}^{\otimes}$, we obtain a natural equivalence
\begin{align}\label{diag_orthogonal_alg}
\begin{gathered}
\xymatrix{
 & \mathsf{E}_1\times \mathsf{E}_1\ar[dd]|\hole\ar[dr]^A & & \\
\mathsf{E}_1\times \mathsf{E}_1\ar[rr]_{A^{\rho}}\ar[ur]^{\mathrm{swap}}_{\cong}\ar[dd]\ar@{}[urr]|{A^{\sigma}} & & \mathcal{C}^{\otimes}\ar@{->>}[dd] & \\
 & \mathrm{Fin}_{\ast}\times\mathrm{Fin}_{\ast}\ar[dr]^{\wedge} & & \\
 \mathrm{Fin}_{\ast}\times\mathrm{Fin}_{\ast}\ar[ur]^{\mathrm{swap}}_{\cong}\ar[rr]_{\wedge}\ar@{}[urr]|{\Uparrow\sigma} & & \mathrm{Fin}_{\ast} &
}
\end{gathered}
\end{align}

The equivalences $A^{\sigma}(\langle n\rangle,\langle m\rangle)\colon (A^{\otimes n})^{\otimes m}\rightarrow (A^{\otimes m})^{\otimes n}$ 
are (essentially) the permutation of the components of the tensor power $A^{\otimes nm}$ according to the formula (\ref{def_sigma}). 

\begin{remark}
The natural isomorphism $\sigma$ is symmetric as  $\sigma(\langle n\rangle,\langle m\rangle)=\sigma(\langle m\rangle,\langle n\rangle)^{-1}$ for all $n,m$. Generally however $\sigma$ is not the identity. Yet, it is whenever either $n=1$ or $m=1$:
\[\sigma(\langle 1\rangle,-)=\sigma(-,\langle 1\rangle)=1.\]
One can use this to construct a lift $\bar{\sigma}$ which is the identity when restricted to the two embeddings $\mathsf{E}_1\hookrightarrow 
\mathsf{E}_1\otimes_{\mathrm{BV}}\mathsf{E}_1$ simply by way of the homotopy lifting property in the model category for $\infty$-operads. 
\end{remark}

\begin{example}[Diagonal reflection in dimension 2]\label{exple_diagrefl2}
Heuristically, diagonal reflection
\[(-)^{\rho}\colon\mathrm{Alg}_{\mathsf{E}_2}(\mathcal{C}^{\otimes})\rightarrow\mathrm{Alg}_{\mathsf{E}_2}(\mathcal{C}^{\otimes})\]
takes an $\mathsf{E}_2$-algebra $A$ in $\mathcal{C}^{\otimes}$, considers its underlying bifunctor, flips it, and 
then turns it into an $\infty$-operad by correcting the introduced error in parametrization over $\mathrm{Fin}_{\ast}$ by swapping the 
inputs along a lift of $\sigma$. Furthermore, diagonal reflection is natural; in particular, to compute a portion of a diagonal reflection, 
it suffices to compute the diagonal reflection of this portion per se. For instance, let $A$ be a non-unital $\mathsf{A}_2$-structure in
$\mathrm{Alg}_{\mathsf{E}_1}(\mathcal{C}^{\otimes})$ of the form
\begin{align}\label{diag_exple_diagrefl2}
\begin{gathered}
\xymatrix{
\ar@{}[d]|{\vdots} & \ar@{}[d]|{\vdots} \\
A^{\otimes 4} \ar@<-.5ex>[d]\ar[d]\ar@<.5ex>[d] & (A^{\otimes 2})^{\otimes 4}\ar[l]\ar@<.5ex>[d]\ar[d]\ar@<-.5ex>[d] & & & \\
A^{\otimes 3} \ar@<-.5ex>[d]_{\rotatebox{-90}{$\scriptstyle m\otimes 1$}}\ar@<.5ex>[d]^{\rotatebox{-90}{$\scriptstyle 1\otimes m$}}\ar@{}[ur]|{\phi} & (A^{\otimes 2})^{\otimes 3}\ar[l]_{(m\sprime, m\sprime, m\sprime)}\ar@<-.5ex>[d]_{\rotatebox{-90}{$\scriptstyle m^{\otimes}\otimes 1$}}\ar@<.5ex>[d]^{\rotatebox{-90}{$\scriptstyle 1\otimes m^{\otimes}$}} \\
A^{\otimes 2}\ar[d]_{m} \ar[d]\ar@{}[ur]|{\omega} & (A^{\otimes 2})^{\otimes 2}\ar[d]\ar[l]_{(m\sprime, m\sprime)}\ar[d]^{m^{\otimes}}\\
 A\ar@{}[ur]|{\chi} & A^{\otimes 2}\ar[l]|{m\sprime},
}
\end{gathered}
\end{align}
where $\chi$, $\omega$, $\phi$ denote the cells filling the due Stasheff polygons. Post-composition with the swap isomorphism yields a 
bifunctor $\mathsf{E}_1\times \mathsf{A}_2^{\mathrm{nu}}\rightarrow\mathcal{C}^{\otimes}$ of the form
\begin{align}\label{diag_exple_diagrefl_flip}
\begin{gathered}
\xymatrix{
A^{\otimes 2}\ar[d]|{m\sprime} & (A^{\otimes 2})^{\otimes 2}\ar[l]_{m^{\otimes}}\ar[d]|{(m\sprime,m\sprime)} & (A^{\otimes 2} )^{\otimes 3}\ar@<.5ex>[l]^{m^{\otimes}\otimes 1}\ar@<-.5ex>[l]_{1\otimes m^{\otimes}}\ar[d]|{(m\sprime,m\sprime,m\sprime)} & (A^{\otimes 2})^{\otimes 4}\ar@<.5ex>[l]\ar[l]\ar@<-.5ex>[l]\ar[d] & \ar@{}[l]|{\dots}\\
A\ar@{}[ur]|{\chi^{-1}} & A^{\otimes 2}\ar[l]|{m}\ar@{}[ur]|{\omega^{-1}} & A^{\otimes 3}\ar@<.5ex>[l]^{m\otimes 1}\ar@<-.5ex>[l]_{1\otimes m}\ar@{}[ur]|{\phi^{-1}} & A^{\otimes 4}\ar@<-.5ex>[l]_{m\otimes 1\otimes 1}\ar[l]|{1\otimes m\otimes 1}\ar@<.5ex>[l]^{1\otimes 1\otimes m} & \ar@{}[l]|{\dots}
}\end{gathered}
\end{align}

Indeed, note that the swap automorphism inverts all cells of dimension $n\geq 2$ if we are to preserve the direction of the cells. More 
precisely, two things are happening. First, in (\ref{diag_exple_diagrefl2}) we consider the 2-cell $\chi$ as a morphism from $m^{\otimes}$ 
to $m$, and in particular as a morphism from the (essentially uniquely determined) composition $m(m\sprime,m\sprime)$ to
$m^{\otimes}m\sprime$ in the corresponding mapping space. In Diagram~(\ref{diag_exple_diagrefl_flip}) we consider the 2-cell $\chi^{-1}$ as 
a morphism from $(m\sprime,m\sprime)$ to $m\sprime$, and in particular as a morphism from the (essentially uniquely determined) composition 
$m^{\otimes}m\sprime$ to $m(m\sprime,m\sprime)$ in the corresponding mapping space.

Diagram (\ref{diag_orthogonal_alg}) now corresponds to an equivalence as follows, where the back face determines the diagonally 
reflected bifunctor $A^{\rho}\colon \mathsf{E}_1\times \mathsf{A}_2\rightarrow\mathcal{C}^{\otimes}$ by definition.
\begin{align}\label{diag_orthogonal_alg_1}
\begin{gathered}
\xymatrix{
 & A^{\otimes 2}\ar@/^1pc/[ddl]|(.2){m\sprime}\ar@{=}[dl]_{A^{\sigma(\langle 1\rangle,\langle 2\rangle)}} & (A^{\otimes 2})^{\otimes 2}\ar[l]_{(m,m)}\ar@/^1pc/[ddl]|(.2){(m\sprime)^{\otimes}}\ar[dl]|{A^{\sigma(\langle 2\rangle,\langle 2\rangle)}} & (A^{\otimes 3} )^{\otimes 2}\ar@<.5ex>[l]^{(m\otimes 1)^2}\ar@<-.5ex>[l]_{(1\otimes m)^2}\ar@/^1pc/[ddl]|(.2){(m\sprime)^{\otimes}}\ar[dl]|{A^{\sigma(\langle 3\rangle,\langle 2\rangle)}} & (A^{\otimes 4})^{\otimes 2}\ar@<.5ex>[l]\ar[l]_{(1\otimes m)^3}\ar@<-.5ex>[l]\ar@/^1pc/[ddl]|(.2){(m\sprime)^{\otimes}}\ar[dl]|{A^{\sigma(\langle 4\rangle,\langle 2\rangle)}}  & \ar@{}[l]|{\dots}\\
A^{\otimes 2}\ar[d]|{m\sprime} \ar[d] & (A^{\otimes 2})^{\otimes 2}\ar[l]_{m^{\otimes}}\ar[d]|{(m\sprime,m\sprime)} & (A^{\otimes 2} )^{\otimes 3}\ar@<.5ex>[l]^{m^{\otimes}\otimes 1}\ar@<-.5ex>[l]_{1\otimes m^{\otimes}}\ar[d]|{(m\sprime,m\sprime,m\sprime)} & (A^{\otimes 2})^{\otimes 4}\ar@<.5ex>[l]\ar[l]\ar@<-.5ex>[l]\ar[d] & \ar@{}[l]|{\dots} & \\
A& A^{\otimes 2}\ar[l]|{m}\ar@{}[ul]|(.4){\chi^{-1}}  & A^{\otimes 3}\ar@<.5ex>[l]^{m\otimes 1}\ar@<-.5ex>[l]_{1\otimes m}\ar@{}[ul]|(.4){\omega^{-1}} & A^{\otimes 4}\ar@<-.5ex>[l]_{m\otimes 1\otimes 1}\ar[l]|{1\otimes m\otimes 1}\ar@<.5ex>[l]^{1\otimes 1\otimes m}\ar@{}[ul]|(.4){\phi^{-1}} & \ar@{}[l]|{\dots} & 
}
\end{gathered}
\end{align}
This gives a formula for the diagonal reflection $A^{\rho}$ as the composition of the top squares given by the $\sigma$-permutations with 
the front squares given by flip of $A$ itself.
Furthermore, the square
\begin{align}\label{diag_orthogonal_alg_nat}
\begin{gathered}
\xymatrix{
\mathrm{Alg}_{\mathsf{E}_1}(\mathrm{Alg}_{\mathsf{E}_1}(\mathcal{C}^{\otimes}))\ar[r]^{(-)^{\rho}}\ar[d]_U & \mathrm{Alg}_{\mathsf{E}_1}(\mathrm{Alg}_{\mathsf{E}_1}(\mathcal{C}^{\otimes}))\ar[d]^{\mathrm{Alg}_{\mathsf{E}_1}(U)} \\
\mathrm{Alg}_{\mathsf{A}_2}(\mathrm{Alg}_{\mathsf{E}_1}(\mathcal{C}^{\otimes}))\ar[r]_{(-)^{\rho}} & \mathrm{Alg}_{\mathsf{E}_1}(\mathrm{Alg}_{\mathsf{A}_2}(\mathcal{C}^{\otimes}))
}
\end{gathered}
\end{align}
commutes by naturality of diagonal reflection. Thus, whenever $A$ is an $\mathsf{E}_2$-algebra in $\mathcal{C}^{\otimes}$, then the 
underlying $\mathsf{E}_1\otimes_{\mathrm{BV}}\mathsf{A}^{\mathrm{nu}}_2$-algebra of its diagonal reflection $A^{\rho}$ is the diagonal 
reflection of the $\mathsf{A}^{\mathrm{nu}}_2\otimes_{\mathrm{BV}}\mathsf{E}_1$-algebra underlying $A$.
\end{example}

\begin{example}[Diagonal reflection in dimension 3]\label{exple_diagrefl3}
Here we have three diagonal reflection operations, each corresponding to the reflection at the hyperplane crossing diagonally through one of 
the three axes in $\Box^3$. Let $\mathcal{C}^{\otimes}$ be a symmetric monoidal $\infty$-category. Then, first, diagonal reflection as 
discussed in Example~\ref{exple_diagrefl2} in particular applies to the symmetric monoidal $\infty$-category
$\mathrm{Alg}_{\mathsf{E}_1}(\mathcal{C}^{\otimes})$. This is induced by the permutation
$(\pi_2,\pi_1,\pi_3)\colon (\mathsf{E}_1)^{3}\rightarrow(\mathsf{E}_1)^{3}$. Geometrically, the associated automorphism
\[\rho_{x_1}\colon \mathsf{E}_3\rightarrow \mathsf{E}_3\]
is induced by reflection $\mu_{x}\colon\Box^3\rightarrow\Box^3$ at the $(x_2=x_3)$-hyperplane following the same recipe as in 
Remark~\ref{rem_diagreflgeo}. Second, diagonal reflection as discussed in Example~\ref{exple_diagrefl2} is symmetric monoidal itself, and hence induces an automorphism
\[(-)^{\rho_{x_3}}:=\mathrm{Alg}_{\mathsf{E}_1}((-)^{\rho})\colon\mathrm{Alg}_{\mathsf{E}_1}(\mathrm{Alg}_{\mathsf{E}_2}(\mathcal{C}^{\otimes}))\rightarrow\mathrm{Alg}_{\mathsf{E}_1}(\mathrm{Alg}_{\mathsf{E}_2}(\mathcal{C}^{\otimes})).\]
This is the diagonal reflection induced by the permutation $(\pi_1,\pi_3,\pi_2)\colon (\mathsf{E}_1)^{3}\rightarrow(\mathsf{E}_1)^{3}$. Geometrically, the 
associated automorphism
\[\rho_z\colon \mathsf{E}_3\rightarrow \mathsf{E}_3\]
is induced by the reflection $\mu_{z}\colon\Box^3\rightarrow\Box^3$ at the $(x_1=x_2)$-hyperplane. The third will not be important for 
later application, but for the sake of completeness one may consider the permutation $(\pi_3,\pi_2,\pi_1)\colon (\mathsf{E}_1)^{3}\rightarrow(\mathsf{E}_1)^{3}$, 
whose associated automorphism
\[\rho_{x_2}\colon \mathsf{E}_3\rightarrow \mathsf{E}_3\]
is geometrically induced by the reflection $\mu_{y}\colon\Box^3\rightarrow\Box^3$ at the $(x_1=x_3)$-hyperplane. 

Suppose we are given an $\mathsf{A}_2^{\mathrm{nu}}\otimes \mathsf{A}_2^{\mathrm{nu}}\otimes \mathsf{A}_2^{\mathrm{nu}}$-algebra in
$\mathcal{C}^{\otimes}$ as follows.
\begin{align}\label{diag_diagrefl3}
\begin{gathered}
\xymatrix{
 & A^{4}\ar[dl]_{m_1^{\otimes}}\ar[dd]|(.7){(m_2,m_2)}|\hole & & A^{8}\ar[dd]^{(m_2^{\otimes},m_2^{\otimes})}\ar[ll]_{m_3^4}\ar[dl]_(.6){(m_1^{\otimes})^{\otimes}} \\
A^2\ar[dd]_{m_2}\ar@{}[dr]|(.6){\Leftarrow\chi_{12}} & & A^4 \ar[dd]_(.3){m_2^{\otimes}}\ar[ll]_(.3){(m_3,m_3)} \\
& A^2\ar[dl]_(.3){m_1}\ar@{}[dr]|{\Downarrow\chi_{13}} & & A^4\ar[dl]^(.3){m_1^{\otimes}}\ar[ll]^(.7){(m_3,m_3)}\\
A & & A^2\ar[ll]^{m_3} & 
}
\end{gathered}
\end{align}
Here, the diagonal bottom multiplication $m_1\colon A^2\rightarrow A$ is the underlying $\mathsf{A}_2^{\mathrm{nu}}$-algebra in
$\mathcal{C}^{\otimes}$, and the left face $\chi_{12}$ is the underlying
$\mathsf{A}_2^{\mathrm{nu}}\otimes \mathsf{A}_2^{\mathrm{nu}}$-algebra in $\mathcal{C}^{\otimes}$. The front face is given by some 2-cell
$\chi_{23}$. The other three faces are given by corresponding pairs of their opposite face. The cube (\ref{diag_diagrefl3}) gives a morphism 
from the right face $(\chi_{12},\chi_{12})$ to the left face $\chi_{12}$. 
%
Whenever the cube (\ref{diag_diagrefl3}) underlies an $\mathsf{E}_3$-algebra, the three faces $\chi_{12}$, $\chi_{23}$ and $\chi_{13}$ are 
mutually equivalent, and in fact can be taken to coincide by Proposition~\ref{prop_eh} and Remark~\ref{rem_algfib}.
It may be helpful to observe that the intuition of diagonal reflection as a folding 
operation translates to this algebraic picture. The faces $\chi_{12}$ and $\chi_{13}$ are equivalent by folding them on top of each other 
along the edge $m_1\colon A^2\rightarrow A$, and so are $\chi_{13}$ and $\chi_{23}$ with respect to their common edge
$m_3\colon A^2\rightarrow A$. Folding the face $\chi_{12}$ directly along $m_2\colon A^2\rightarrow A$ however does yield
$\chi_{23}=\chi_{12}$ only up to a twist of components. This corresponds to the observation that the folding of Remark~\ref{rem_eh_twist} 
introduces a twist of components.
\end{example}

\subsection{Iterating fragments of associative structures in $n$-categories}\label{sec_sub_iniseg}
This section is concerned with a set of statements that allows us to reduce the construction of $\mathsf{E}_m$-algebras in a symmetric 
monoidal $n$-category $\mathcal{C}^{\otimes}$ to a proportionally small amount of data (Propositions~\ref{prop_inisegloc} and 
\ref{prop_inisegfib} as well as Corollaries~\ref{cor_truncmor} and \ref{cor_truncobj}). Crucially, in Section~\ref{sec_translation} we will 
see that this small amount of data recovers precisely the axioms provided by \cite{daystreet} for the various notions of symmetry in
$\mathcal{C}^{\otimes}$ the cartesian monoidal $3$-category of bicategories. To do so, we introduce initial segments of non-unital
$\mathsf{E}_k$-algebra structures. Their combinatorial complexity remedies the fact 
that iterated non-unital $\mathsf{A}_n$-algebra structures generally do not satisfy an Eckmann-Hilton type argument. Although the specific
monoidal $3$-category we want to apply these statements to is cartesian, it doesn't appear that cartesianness renders the definitions 
or proofs any shorter. The results in this section will hence concern general symmetric monoidal $n$-categories. The use of these initial 
segments will be illustrated when applied to cartesian monoidal $n$-categories in Section~\ref{sec_sub_iniseg_app}.

For the rest of this section we fix a finite integer $n\geq -2$ together with a symmetric monoidal $n$-category $\mathcal{C}^{\otimes}$. 
That is, a symmetric monoidal $\infty$-category
$\mathcal{C}^{\otimes}$ whose underlying $\infty$-category $\mathcal{C}$ is an $n$-category in the sense of \cite[Section 2.3.4]{luriehtt}.
It is easy to see that the $\infty$-category $\mathrm{Alg}_{\mathsf{E}_k}(\mathcal{C}^{\otimes})$ of $\mathsf{E}_k$-algebras in
$\mathcal{C}^{\otimes}$ for any $0\leq k\leq \infty$ is again an $n$-category. The same applies to the non-unital versions. Furthermore, by 
the Baez--Dolan--Lurie Stabilization Theorem \cite[Corollary 5.1.1.7]{lurieha} the forgetful functor
\[\mathrm{Alg}_{\mathsf{E}_{\infty}}(\mathcal{C}^{\otimes})\rightarrow\mathrm{Alg}_{\mathsf{E}_{k}}(\mathcal{C}^{\otimes})\]
is an equivalence for all $k>n$.  Similarly, the forgetful functor
\[\mathrm{Alg}_{\mathsf{A}_{\infty}}^{\mathrm{nu}}(\mathcal{C}^{\otimes})\rightarrow\mathrm{Alg}_{\mathsf{A}_{m}}^{\mathrm{nu}}(\mathcal{C}^{\otimes})\]
is an equivalence of $n$-categories for all $m>n+1$ \cite[Corollary 4.1.6.16]{lurieha}. In tandem with the Additivity Theorem, this means 
that to define an $\mathsf{E}_m$-monoidal structure on an object $A$ in $\mathcal{C}^{\otimes}$, it is sufficient to construct $m$-many (at 
most $(n+1)$-many) compatible $\mathsf{A}_{n+2}$-algebra structures on $A$. For $n>1$ this is still a fair amount of data to handle, and in 
fact one can do much better. In a sense, the complexity of associativity laws decreases with increasing fragments of commutativity
(Corollary~\ref{cor_truncobj}). 

\begin{notation}
In the following, $K_m$ denotes the $m$-th Stasheff associahedron as defined in \cite[Definition 4.1.6.1]{lurieha}, and
$\partial K_m$ denotes its boundary \cite[Definition 4.1.6.4]{lurieha}. The push-out $K_m\sqcup_{\partial K_m} K_m$ of the 
boundary inclusion along itself will be denoted by $\Sigma\partial K_m$ following \cite[Construction 4.1.6.12]{lurieha}.
\end{notation}

\begin{lemma}\label{lemma_truncext0}
Let $\mathcal{C}^{\otimes}$ be a symmetric monoidal $n$-category and let $2\leq m<\infty$ be an integer. Then the 
forgetful functor
\[\mathrm{Alg}_{\mathsf{A}_m}^{\mathrm{nu}}(\mathcal{C}^{\otimes})\rightarrow\mathrm{Alg}_{\mathsf{A}_{m-1}}^{\mathrm{nu}}(\mathcal{C}^{\otimes})\]
is locally $(n-m)$-truncated. In particular so is the forgetful functor
$\mathrm{Alg}_{\mathsf{A}_m}^{\mathrm{qu}}(\mathcal{C}^{\otimes})\rightarrow\mathrm{Alg}_{\mathsf{A}_{m-1}}^{\mathrm{qu}}(\mathcal{C}^{\otimes})$.
\end{lemma}
\begin{proof}
Let $2\leq m<\infty$ be an integer. We are to show that for any two non-unital $\mathsf{A}_m$-algebras $A,B$ in $\mathcal{C}^{\otimes}$ the 
forgetful functor
\[\mathrm{Alg}_{\mathsf{A}_m}^{\mathrm{nu}}(\mathcal{C}^{\otimes})(A,B)\rightarrow\mathrm{Alg}_{\mathsf{A}_{m-1}}^{\mathrm{nu}}(\mathcal{C}^{\otimes})(A,B)\]
of hom-spaces is $(n-m)$-truncated. The argument is basically that of \cite[Corollary 4.1.6.16]{lurieha}. Let $A,B$ be non-unital 
$\mathsf{A}_m$-algebras in $\mathcal{C}^{\otimes}$. By assumption, $\mathcal{C}$ is a $n$-category. Thus, the hom-space $\mathcal{C}(A,B)$ 
of their underlying objects in $\mathcal{C}$ is $(n-1)$-truncated. Furthermore, the space $\Sigma\partial K_m$ is
$(m-3)$-connected. It follows from \cite[Proposition 8.13]{rezkhtytps} that the diagonal
\begin{align}\label{equlemma_truncext11}
\mathcal{C}(\mathsf{A}_1^{\otimes m},B_1)\rightarrow\mathcal{C}(\mathsf{A}_1^{\otimes m},B_1)^{\Sigma\partial K_m}
\end{align}
is $((n-1)-(m-3)-2)$-truncated. Now, given any morphism
$f\in\mathrm{Alg}_{\mathsf{A}_{m-1}}^{\text{nu}}(\mathcal{C}^{\otimes})(A,B)$, the fiber of the forgetful functor
\begin{align}\label{equlemma_truncext12}
\mathrm{Alg}_{\mathsf{A}_m}^{\text{nu}}(\mathcal{C}^{\otimes})(A,B)\rightarrow\mathrm{Alg}_{\mathsf{A}_{m-1}}^{\text{nu}}(\mathcal{C}^{\otimes})(A,B)
\end{align}
at $f$ is equivalent to a fiber of the map (\ref{equlemma_truncext11}) \cite[Theorem 4.1.6.13]{lurieha}. Hence, all fibers of 
(\ref{equlemma_truncext12}) are $(n-m)$-truncated, which means it is $(n-m)$-truncated itself.
\end{proof}

We now define initial segments of iterated non-unital $\mathsf{A}_m$-monoidal structures in a symmetric monoidal $\infty$-category
$\mathcal{C}^{\otimes}$. Here, we use extensively that the forgetful functor $\mathrm{SMonCat}_{\infty}\rightarrow\mathrm{Cat}_{\infty}$
reflects all limits.

\begin{definition}\label{def_iniseg}
Let $\mathcal{C}^{\otimes}$ be a symmetric monoidal $\infty$-category. Let $J^{(0)}_{\emptyset}(\mathcal{C}^{\otimes}):=\ast$ and 
$V^{(0)}\colon\mathcal{C}^{\otimes}\rightarrow J^{(0)}_{\emptyset}(\mathcal{C}^{\otimes})$ be the essentially uniquely forgetful functor. 
Given $1\leq k<\infty$ and integers $1\leq m_0,\dots,m_{k-1}<\infty$, we recursively define
\begin{align}\label{diag_iniseg_0}
\begin{gathered}
\xymatrix{
\mathrm{Alg}_{\mathsf{E}_k}^{\mathrm{nu}}(\mathcal{C}^{\otimes})\ar@/_1pc/[ddr]_{\mathrm{Alg}_{\mathsf{E}_{1}}^{\mathrm{nu}}(V^{(k-1)})}\ar@/^1pc/[drr]^U\ar[dr]|{V^{(k)}} & & \\
& \mathrm{J}^{(k)}_{\vec{m}}\ar[r]\ar[d]\ar@{}[dr]|(.3){\pbs} & \mathrm{Alg}_{\mathsf{A}_{m_0}}^{\mathrm{nu}}(\mathrm{Alg}_{\mathsf{E}_{k-1}}^{\mathrm{nu}}(\mathcal{C}^{\otimes}))\ar[d]^{\mathrm{Alg}_{\mathsf{A}_{m_0}}^{\mathrm{nu}}(V^{(k-1)})}
\\
& \mathrm{Alg}_{\mathsf{E}_1}^{\mathrm{nu}}(J^{(k-1)}_{(\vec{m}\setminus \{m_0\})}(\mathcal{C}^{\otimes}))\ar[r]_U & \mathrm{Alg}_{\mathsf{A}_{m_0}}^{\mathrm{nu}}(J^{(k-1)}_{\vec{m}\setminus\{m_0\}}(\mathcal{C}^{\otimes})).
}
\end{gathered}
\end{align}
Furthermore, for $0\leq k<\infty$ and integers $1\leq m_0,\dots,m_{k}<\infty$, we define
\begin{align}\label{diag_iniseg_I}
\begin{gathered}
\xymatrix{
\mathrm{Alg}_{\mathsf{A}_{m_{k}+1}}^{\mathrm{nu}}(\mathrm{Alg}_{\mathsf{E}_k}^{\mathrm{nu}}(\mathcal{C}^{\otimes}))\ar@/_1pc/[ddr]_{\mathrm{Alg}_{\mathsf{A}_{m_{k}+1}}^{\mathrm{nu}}(V^{(k)})}\ar@/^1pc/[drr]^U\ar[dr]|{U^{(k)}} & & \\
 & \mathrm{I}^{(k)}_{\vec{m}}\ar[r]\ar[d]\ar@{}[dr]|(.3){\pbs} & \mathrm{Alg}_{\mathsf{A}_{m_{k}}}^{\mathrm{nu}}(\mathrm{Alg}_{\mathsf{E}_k}^{\mathrm{nu}}(\mathcal{C}^{\otimes}))\ar[d]^{\mathrm{Alg}_{\mathsf{A}_{m_{k}}}^{\mathrm{nu}}(V^{(k)})}\\
 & \mathrm{Alg}_{\mathsf{A}_{m_{k}+1}}^{\mathrm{nu}}(J^{(k)}_{\vec{m}\setminus\{m_{k}\}}(\mathcal{C}^{\otimes}))\ar[r]_U & \mathrm{Alg}_{\mathsf{A}_{m_{k}}}^{\mathrm{nu}}(J^{(k)}_{\vec{m}\setminus\{m_{k}\}}(\mathcal{C}^{\otimes})).
}
\end{gathered}
\end{align}
We will refer to $I^{(k)}_{\vec{m}}(\mathcal{C}^{\otimes})$ as the symmetric monoidal $\infty$-category of
\emph{$k$-fold associative initial segments} with parameters $\vec{m}=(m_0,\dots,m_{k})$. 
\end{definition}

To illustrate, the $\infty$-category $I^{(0)}_{m_0}(\mathcal{C}^{\otimes})$ is just
$\mathrm{Alg}_{m_0}^{\mathrm{nu}}(\mathcal{C}^{\otimes})$. The forgetful functor $U^{(0)}$ is the canonical forgetful functor
$U\colon\mathrm{Alg}_{m_0+1}^{\mathrm{nu}}(\mathcal{C}^{\otimes})\rightarrow\mathrm{Alg}_{m_0}^{\mathrm{nu}}(\mathcal{C}^{\otimes})$. 
An object in the $\infty$-category 
\[\xymatrix{
I^{(1)}_{(m_0,m_1)}\ar[r]\ar[d]\ar@{}[dr]|(.3){\pbs} & \mathrm{Alg}_{\mathsf{A}_{m_1}}^{\mathrm{nu}}(\mathrm{Alg}_{\mathsf{E}_1}^{\mathrm{nu}}(\mathcal{C}^{\otimes}))\ar[d]^{\mathrm{Alg}_{\mathsf{A}_{m_1}}^{\mathrm{nu}}(U)} \\
\mathrm{Alg}_{\mathsf{A}_{m_1+1}}^{\mathrm{nu}}(\mathrm{Alg}_{\mathsf{A}_{m_0}}^{\mathrm{nu}}(\mathcal{C}^{\otimes}))\ar[r]_U & \mathrm{Alg}_{\mathsf{A}_{m_1}}^{\mathrm{nu}}(\mathrm{Alg}_{\mathsf{A}_{m_0}}^{\mathrm{nu}}(\mathcal{C}^{\otimes}))
}\]
consists informally of an $\mathsf{A}_{m_1}^{\mathrm{nu}}\otimes_{\mathrm{BV}}\mathsf{A}_{m_0}^{\mathrm{nu}}$-algebra $A$ in $\mathcal{C}^{\otimes}$ together 
with an extension of $A$ to an $\mathsf{A}_{m_1+1}^{\mathrm{nu}}\otimes_{\mathrm{BV}}\mathsf{A}_{m_0}^{\mathrm{nu}}$-algebra in $\mathcal{C}^{\otimes}$, as 
well as an extension of $A$ to an $\mathsf{A}_{m_1}^{\mathrm{nu}}\otimes_{\mathrm{BV}}\mathsf{E}_1^{\mathrm{nu}}$-algebra in $\mathcal{C}^{\otimes}$. The fiber 
of the forgetful functor $U^{(1)}\colon\mathrm{Alg}_{\mathsf{A}_{m_{1}+1}}^{\mathrm{nu}}(\mathrm{Alg}_{\mathsf{E}_1}^{\mathrm{nu}}(\mathcal{C}^{\otimes}))\rightarrow I^{(1)}_{(m_0,m_1)}$ in (\ref{diag_iniseg_I}) over this object consists of an
$\mathsf{A}_{m_1+1}^{\mathrm{nu}}\otimes_{\mathrm{BV}} \mathsf{E}_1^{\mathrm{nu}}$-algebra that extends both of these components at the same time. 

\begin{remark}
The recursive Definition~\ref{def_iniseg} can also be given combinatorially directly for each integer $k$. However, the 
given recursive description is more conducive to the proofs below which proceed by induction. In the special case that all parameters
$m_i$ are minimal, the object $J^{(k)}_{\bar{1}}(\mathcal{C}^{\otimes})$ is a non-unital version of the $\infty$-category of
$L_k(\mathsf{E}_{\bullet})$-algebras in $\mathcal{C}^{\otimes}$, where $L_k(\mathsf{E}_{\bullet})$ is the latching object introduced in 
Section~\ref{sec_sub_comm}. This will be used in Corollary~\ref{cor_truncmor}.
\end{remark}

The following two propositions are a refined version of Schlank--Yanovski's $\infty$-categorical Eckmann--Hilton Argument 
\cite{schlankyanovski_eh} in the case of non-unital partially associative structures.

\begin{proposition}\label{prop_inisegloc}
Let $\mathcal{C}^{\otimes}$ be a symmetric monoidal $n$-category. Let $0\leq k<\infty$ be an integer and $1\leq m_0,\dots,m_{k}<\infty$ be 
a sequence of further integers. Then the forgetful functor
\[U^{(k)}\colon\mathrm{Alg}_{\mathsf{A}_{m_{k}+1}}^{\mathrm{nu}}(\mathrm{Alg}_{\mathsf{E}_k}^{\mathrm{nu}}(\mathcal{C}^{\otimes}))\rightarrow I^{(k)}_{\vec{m}}\]
is locally $\left(n-1-\sum_{i=0}^{k}m_i\right)$-truncated.
\end{proposition}
\begin{proof}
The proof proceeds by induction on the integer $k$. We show simultaneously for every integer $0\leq k<\infty$ and every sequence
$1\leq m_0,\dots,m_{k}<\infty$ of integers that both forgetful functors 
\[U^{(k)}\colon\mathrm{Alg}_{\mathsf{A}_{m_{k}+1}}^{\mathrm{nu}}(\mathrm{Alg}_{\mathsf{E}_k}^{\mathrm{nu}}(\mathcal{C}^{\otimes}))\rightarrow I^{(k)}_{\vec{m}}\]
and 
\[V^{(k+1)}\colon\mathrm{Alg}_{\mathsf{E}_{k+1}}^{\mathrm{nu}}(\mathcal{C}^{\otimes}))\rightarrow J^{(k+1)}_{\vec{m}}\]
are locally $\left(n-1-\sum_{i=0}^{k}m_i\right)$-truncated.

For $k=0$, we recall that the forgetful functor
\[U^{(0)}\colon\mathrm{Alg}_{\mathsf{A}_{m_0+1}}^{\mathrm{nu}}(\mathcal{C}^{\otimes})\rightarrow\mathrm{I}^{(0)}_{m_0}\]
is locally $(n-m_0)$-truncated by Lemma~\ref{lemma_truncext0}. It follows that the forgetful functor
\[V^{(1)}\colon\mathrm{Alg}_{\mathsf{E}_{1}}^{\mathrm{nu}}(\mathcal{C}^{\otimes})\rightarrow\mathrm{J}^{(1)}_{m_0}\]
is the transfinite composition of $(n-m_0)$-truncated morphisms and hence is $(n-m_0)$-truncated itself. 

Let $k\geq 1$, and let us assume that we have shown the forgetful functor
\[V^{(k)}\colon\mathrm{Alg}_{\mathsf{E}_{k}}^{\mathrm{nu}}(\mathcal{C}^{\otimes})\rightarrow\mathrm{J}^{(k)}_{\vec{m}\setminus\{m_k\}}\]
is locally $\left(n-1-\sum_{i=0}^{k-1}m_i\right)$-truncated.
Let $A,B$ be non-unital $\mathsf{A}_{m_{k}+1}$-algebras in $\mathrm{Alg}_{\mathsf{E}_{k}}^{\mathrm{nu}}(\mathcal{C}^{\otimes})$. We are to show that the 
forgetful functor
\[U^{(k)}\colon\mathrm{Alg}_{\mathsf{A}_{m_{k}+1}}^{\mathrm{nu}}(\mathrm{Alg}_{\mathsf{E}_{k}}^{\mathrm{nu}}(\mathcal{C}^{\otimes}))(A,B)\rightarrow I^{(k)}_{\vec{m}}(A,B)\]
of hom-spaces is $\left(n-1-\sum_{i=0}^{k}m_i\right)$-truncated. Let $(f,g)\in I^{(k)}_{\vec{m}}(A,B)$ be an element of the codomain
hom-space. We are to compute the fiber 
\[\xymatrix{
 & \mathrm{Alg}_{\mathsf{A}_{m_{k}+1}}^{\mathrm{nu}}(\mathrm{Alg}_{\mathsf{E}_k}^{\mathrm{nu}}(\mathcal{C}^{\otimes}))(A,B)\ar[rr]^{U^{(k)}}\ar[dd]|\hole & & I^{(k)}_{\vec{m}}(A,B)\ar@/^3pc/[ddll] \\
F(A,B)\ar[ur]\ar[rr]\ar@{}[urr]|(.3){\rotatebox[origin=c]{90}{$\pbs$}}\ar@/_1pc/[dr] & & \ast\ar[ur]^(.4){\{(f,g)\}}\ar@/^1pc/[dl]^(.2){\{f\}} &  \\
 & \mathrm{Alg}_{\mathsf{A}_{m_{k}+1}}^{\mathrm{nu}}(J^{(k)}_{\vec{m}\setminus\{m_{k}\}}(\mathcal{C}^{\otimes})))(A,B). & &
}\]
By way of \cite[Theorem 4.1.6.13]{lurieha}, this fiber is equivalent to the fiber $F(A,B)$ in the following diagram.
\[\xymatrix{
 & & \mathrm{Alg}_{\mathsf{E}_k}^{\mathrm{nu}}(\mathcal{C}^{\otimes})(A^{\otimes m_{k}+1},B)\ar[d]\ar@/^2pc/[drr]^{\Delta} & & \\
 & L(A,B)\ar[d]\ar[ur]\ar@{}[r]|(.2){\pbs} & \mathrm{Alg}_{\mathsf{E}_k}^{\mathrm{nu}}(\mathcal{C}^{\otimes})(A^{\otimes m_{k}+1},B)^{\overline{\Sigma\partial K_{m_{k}+1}}}\ar[dd]|\hole\ar[rr]\ar@{}[ddrr]|(.2){\pbs} & & \mathrm{Alg}_{\mathsf{E}_k}^{\mathrm{nu}}(\mathcal{C}^{\otimes})(A^{\otimes m_{k}+1},B)^{\Sigma\partial K_{m_{k}+1}}\ar[dd]\\
 F(A,B)\ar[d]\ar[ur]\ar@{}[r]|(.3){\pbs}  & \mathrm{Alg}_{\mathsf{A}_{m_{k}+1}}^{\mathrm{nu}}[f]\ar[rr]\ar@/_1pc/[dr]\ar@{}[dr]|(.3){\pbs}\ar[ur] & & \ast\ar[ur]^(.4){\{\beta_g\}}\ar@/_/[dr]^{\{\beta_g\}\simeq\{\beta_f\}} & \\
 \ast\ar[ur]_{\{f\}} & & J^{(k)}_{\vec{m}\setminus\{m_{k}\}}(\mathcal{C}^{\otimes})(A^{\otimes m_{k}+1},B)\ar[rr]_{\Delta} & & J^{(k)}_{\vec{m}\setminus\{m_{k}\}}(\mathcal{C}^{\otimes})(A^{\otimes m_{k}+1},B)^{\Sigma\partial K_{m_{k}+1}}
}\]
Here, the functors $\Delta$ denote the diagonal given by cotensoring with $\Sigma\partial K_{m_{k}+1}\rightarrow\ast$. Thus, the fiber 
product $\mathrm{Alg}_{\mathsf{E}_k}^{\mathrm{nu}}(\mathcal{C}^{\otimes})(A^{\otimes m_{k}+1},B)^{\overline{\Sigma\partial K_{m_{k}+1}}}$ in 
the middle of the diagram is the $\Sigma\partial K_{m_{k}+1}$-cotensor of the forgetful functor
\[V^{(k)}\colon\mathrm{Alg}_{\mathsf{E}_k}^{\mathrm{nu}}(\mathcal{C}^{\otimes})(A^{\otimes m_{k}+1},B)\rightarrow J^{(k)}_{\vec{m}\setminus\{m_{k}\}}(\mathcal{C}^{\otimes})(A^{\otimes m_{k}+1},B)\]
over its base. The space $\mathrm{Alg}_{\mathsf{A}_{m_{k}+1}}^{\mathrm{nu}}[f]$ consists of all possible 
extensions of the non-unital $\mathsf{A}_{m_{k}}$-algebra morphism $f$ in $J^{(k)}_{\vec{m}\setminus\{m_{k}\}}(\mathcal{C}^{\otimes})$ to a
non-unital $\mathsf{A}_{m_{k}+1}$-algebra morphism by \cite[Theorem 4.1.6.13]{lurieha}. For the same reason it follows that the intermediate space 
$L(A,B)$ in the diagram consists of all possible extensions of the non-unital $\mathsf{A}_{m_k}$-algebra morphism $g$ in
$\mathrm{Alg}_{\mathsf{E}_k}^{\mathrm{nu}}(\mathcal{C}^{\otimes})$ to a non-unital $\mathsf{A}_{m_k+1}$-algebra morphism in
$\mathrm{Alg}_{\mathsf{E}_k}^{\mathrm{nu}}(\mathcal{C}^{\otimes})$.
The space $K(A,B)$ consists of those extensions in $L(A,B)$ that return $f$ as the underlying extension as a morphism of underlying
non-unital $\mathsf{A}_{m_k}$-algebras. Now, the space $\Sigma \partial K_{m_k+1}$ is $(m_k-2)$-connected, and the forgetful functors
$\mathrm{Alg}_{\mathsf{E}_k}^{\mathrm{nu}}(\mathcal{C}^{\otimes})(A,B)\rightarrow J^{(k)}_{\vec{m}\setminus\{m_{k}\}}(\mathcal{C}^{\otimes})(A^{\otimes m_1},B)$ are $(n-1-\sum_{i=0}^{k-1}m_i)$-truncated by assumption. By \cite[Proposition 8.13]{rezkhtytps} it follows that the gap 
map
\[\mathrm{Alg}_{\mathsf{E}_1}(\mathcal{C}^{\otimes})(A^{\otimes m_1},B)\rightarrow\mathrm{Alg}_{\mathsf{E}_1}(\mathcal{C}^{\otimes})(A^{\otimes m_1},B)^{\overline{\Sigma\partial K_{m_1}}}\]
is $\left((n-1-\sum_{i=0}^{k-1}m_i)-(m_1-2)-2\right)$-truncated. That means the space $K$ is $\left((n-1-\sum_{i=0}^{k}m_i)\right)$-truncated 
as well.

For the induction to succeed, we are left to show that the forgetful functor
\begin{align}\label{equ_truncext1}
V^{(k+1)}\colon\mathrm{Alg}_{\mathsf{E}_{k+1}}^{\mathrm{nu}}(\mathcal{C}^{\otimes})\rightarrow\mathrm{J}^{(k+1)}_{\vec{m}}
\end{align}
is $\left(n-1-\sum_{i=0}^{k}m_i\right)$-truncated as well.

Therefore, we note that the $\infty$-category $\mathrm{Alg}_{\mathsf{E}_{k+1}}^{\mathrm{nu}}(\mathcal{C})$ is a sequential limit of the tower of forgetful functors
\begin{align}\label{equ_truncext2}
\begin{gathered}
\xymatrix{\mathrm{Alg}_{\mathsf{A}_{p+1}}^{\mathrm{nu}}(\mathrm{Alg}_{\mathsf{E}_{k}}^{\mathrm{nu}}(\mathcal{C}^{\otimes}))\times_{\mathrm{Alg}_{\mathsf{A}_p+1}^{\mathrm{nu}}(J^{(k)}_{\vec{m}\setminus\{m_0\}}(\mathcal{C}^{\otimes}))}\mathrm{Alg}_{\mathsf{E}_1}^{\mathrm{nu}}(J^{(k)}_{\vec{m}\setminus\{m_0\}}(\mathcal{C}^{\otimes}))\ar[d] \\
\mathrm{Alg}_{\mathsf{A}_{p}}^{\mathrm{nu}}(\mathrm{Alg}_{\mathsf{E}_{k}}^{\mathrm{nu}}(\mathcal{C}^{\otimes}))\times_{\mathrm{Alg}_{\mathsf{A}_p}^{\mathrm{nu}}(J^{(k)}_{\vec{m}\setminus\{m_0\}}(\mathcal{C}^{\otimes}))}\mathrm{Alg}_{\mathsf{E}_1}^{\mathrm{nu}}(J^{(k)}_{\vec{m}\setminus\{m_0\}}(\mathcal{C}^{\otimes}))}
\end{gathered}
\end{align}
indexed over integers $m_0\leq p<\infty$. The transfinite composite 
\[\mathrm{Alg}_{\mathsf{E}_{k+1}}^{\mathrm{nu}}(\mathcal{C})\rightarrow\mathrm{Alg}_{\mathsf{A}_{m_0}}^{\mathrm{nu}}(\mathrm{Alg}_{\mathsf{E}_{k}}^{\mathrm{nu}}(\mathcal{C}^{\otimes}))\times_{\mathrm{Alg}_{\mathsf{A}_{m_0}}^{\mathrm{nu}}(J^{(k)}_{\vec{m}\setminus\{m_0\}}(\mathcal{C}^{\otimes}))}\mathrm{Alg}_{\mathsf{E}_1}^{\mathrm{nu}}(J^{(k)}_{\vec{m}\setminus\{m_0\}}(\mathcal{C}^{\otimes}))\]
is precisely the forgetful functor (\ref{equ_truncext1}) by definition (\ref{diag_iniseg_0}). At any given stage $m_0\leq p<\infty$ the forgetful functor (\ref{equ_truncext2}) is equivalent to the functor
\[\xymatrix{\mathrm{Alg}_{\mathsf{A}_{p+1}}^{\mathrm{nu}}(\mathrm{Alg}_{\mathsf{E}_{k}}^{\mathrm{nu}}(\mathcal{C}^{\otimes}))\times_{\mathrm{Alg}_{\mathsf{A}_p+1}^{\mathrm{nu}}(J^{(k)}_{\vec{m}\setminus\{m_0\}}(\mathcal{C}^{\otimes}))}\mathrm{Alg}_{\mathsf{E}_1}^{\mathrm{nu}}(J^{(k)}_{\vec{m}\setminus\{m_0\}}(\mathcal{C}^{\otimes}))\ar[d]^{(U^{(k)},1)} \\
I^{(k)}_{(m_1,\dots,m_k,p)}(\mathcal{C}^{\otimes})\times_{\mathrm{Alg}_{\mathsf{A}_{p+1}}^{\mathrm{nu}}(J^{(k)}_{\vec{m}\setminus\{m_0\}}(\mathcal{C}^{\otimes}))}\mathrm{Alg}_{\mathsf{E}_1}^{\mathrm{nu}}(J^{(k)}_{\vec{m}\setminus\{m_0\}}(\mathcal{C}^{\otimes}))}
\]
which is (the pullback of) a locally $\left(n-1-\sum_{i=1}^{k}m_i-p\right)$-truncated functor as shown above. In particular,
the forgetful functor (\ref{equ_truncext2}) is locally $\left(n-1-\sum_{i=0}^{k}m_i\right)$-truncated for all $m_0\leq p<\infty$. Hence, so 
is the transfinite composition (\ref{equ_truncext1}).
\end{proof}

\begin{proposition}\label{prop_inisegfib}
Let $\mathcal{C}^{\otimes}$ be a symmetric monoidal $n$-category. Let $0\leq k<\infty$ be an integer and $1\leq m_0,\dots,m_k<\infty$ be 
a sequence of further integers. Then the fiber of the forgetful functor
\[U^{(k)}\colon\mathrm{Alg}_{\mathsf{A}_{m_{k}+1}}^{\mathrm{nu}}(\mathrm{Alg}_{\mathsf{E}_k}^{\mathrm{nu}}(\mathcal{C}^{\otimes}))\rightarrow I^{(k)}_{\vec{m}}\]
at any object of $I^{(k)}_{\vec{m}}$ is an $\left(n-\sum_{i=0}^{k}m_i\right)$-truncated space.
\end{proposition}
\begin{proof}
Consider the following diagram of forgetful functors.

\begin{align}\label{diag_prop_inisegfib1}
\begin{gathered}
\xymatrix{
\mathrm{Alg}_{\mathsf{A}_{m_{k}+1}}^{\mathrm{nu}}(\mathrm{Alg}_{\mathsf{E}_k}^{\mathrm{nu}}(\mathcal{C}^{\otimes}))\ar[r]^U\ar[d]_{\mathrm{Alg}_{\mathsf{A}_{m_{k}+1}}^{\mathrm{nu}}(V^{(k)})} & \mathrm{Alg}_{\mathsf{A}_{m_{k}}}^{\mathrm{nu}}(\mathrm{Alg}_{\mathsf{E}_k}^{\mathrm{nu}}(\mathcal{C}^{\otimes}))\ar[r]^(.6)U\ar[d]^{\mathrm{Alg}_{\mathsf{A}_{m_{k}}}^{\mathrm{nu}}(V^{(k)})} & \mathrm{Alg}_{\mathsf{E}_{k}}(\mathcal{C}^{\otimes})\ar[d]^{V^{(k)}}\\
\mathrm{Alg}_{\mathsf{A}_{m_{k}+1}}^{\mathrm{nu}}(J^{(k)}_{\vec{m}\setminus\{m_{k}\}}(\mathcal{C}^{\otimes}))\ar[r]_U & \mathrm{Alg}_{\mathsf{A}_{m_{k}}}^{\mathrm{nu}}(J^{(k)}_{\vec{m}\setminus\{m_{k}\}}(\mathcal{C}^{\otimes}))\ar[r]_(.6)U & J^{(k)}_{\vec{m}\setminus\{m_{k}\}}(\mathcal{C}^{\otimes})
}
\end{gathered}
\end{align}
For any given $\mathsf{E}_{k}$-algebra $A$ in $\mathcal{C}^{\otimes}$, the point
\[\xymatrix{
\ast\ar[rr]^(.4){\{A\}}\ar@{=}[d] & & \mathrm{Alg}_{\mathsf{E}_{k}}(\mathcal{C}^{\otimes})\ar[d]^{V^{(k)}}\\
\ast\ar[rr]_(.4){\{V^{(k)}(A)\}} & & J^{(k)}_{\vec{m}\setminus\{m_{k}\}}(\mathcal{C}^{\otimes})
}\]
induces fibers of (\ref{diag_prop_inisegfib1}) (on the top and bottom) as follows.
\[\xymatrix{
\mathrm{Alg}_{\mathsf{A}_{m_k+1}}^{\mathrm{nu}}[A]\ar[r]^U\ar[d]_{\mathrm{Alg}_{\mathsf{A}_{m_k+1}}^{\mathrm{nu}}(V^{(k)})} & \mathrm{Alg}_{\mathsf{A}_{m_k}}^{\mathrm{nu}}[A]\ar[r]\ar[d]^{\mathrm{Alg}_{\mathsf{A}_{m_k}}^{\mathrm{nu}}(V^{(k)})} & \ast\ar@{=}[d]\\
\mathrm{Alg}_{\mathsf{A}_{m_k+1}}^{\mathrm{nu}}[V^{(k)}(A)]\ar[r]_U & \mathrm{Alg}_{\mathsf{A}_{m_k}}^{\mathrm{nu}}[V^{(k)}(A)]\ar[r] & \ast
}\]
The according pullback of the gap map
\[U^{(k)}\colon\mathrm{Alg}_{\mathsf{A}_{m_{k}+1}}^{\mathrm{nu}}(\mathrm{Alg}_{\mathsf{E}_{k}}^{\mathrm{nu}}(\mathcal{C}^{\otimes}))\rightarrow I_{\vec{m}}^{(k)}\]
is the gap map
\begin{align}\label{diag_cor_reciproc2}
\begin{gathered}
\xymatrix{
\mathrm{Alg}_{\mathsf{A}_{m_k+1}}^{\mathrm{nu}}[A]\ar@/^1pc/[drr]^U\ar@/_1pc/[ddr]_{\mathrm{Alg}_{\mathsf{A}_{m_k+1}}^{\mathrm{nu}}(V^{(k)})}\ar[dr] & & \\
& I^{(k)}_{\vec{m}}[A]\ar[r]\ar[d]\ar@{}[dr]|(.3){\pbs}& \mathrm{Alg}_{m_k}^{\mathrm{nu}}[A]\ar[d]^{\mathrm{Alg}_{\mathsf{A}_{m_k}}^{\mathrm{nu}}(V^{(k)})}\\
& \mathrm{Alg}_{\mathsf{A}_{m_k+1}}^{\mathrm{nu}}[V^{(k)}(A)]\ar[r]_U & \mathrm{Alg}_{\mathsf{A}_{m_k}}^{\mathrm{nu}}[V^{(k)}(A)].
}
\end{gathered}
\end{align}
Since every point of $I^{(k)}_{\vec{m}}$ factors through $I^{(k)}_{\vec{m}}[A]$ for some $\mathsf{E}_{k}$-algebra $A$ in $\mathcal{C}^{\otimes}$, 
it suffices to show that the gap map in (\ref{diag_cor_reciproc2}) is fiberwise an $\left(n-\sum_{i=0}^{k}m_i\right)$-truncated space 
for all $\mathsf{E}_{k}$-algebras $A$ in $\mathcal{C}^{\otimes}$. So let $A$ be an $\mathsf{E}_{k}$-algebra in $\mathcal{C}^{\otimes}$.
By \cite[Theorem 4.1.6.8]{lurieha} the outer square of (\ref{diag_cor_reciproc2}) comes with a natural transformation into the 
diagram
\begin{align}\label{diag_cor_reciproc3}
\begin{gathered}
\xymatrix{
\mathrm{Alg}_{\mathsf{E}_{k}}(\mathcal{C}^{\otimes})(A^{\otimes m_k+1},A)^{K_{m_k+1}}\ar[r]\ar[d] & \mathrm{Alg}_{\mathsf{E}_{k}}(\mathcal{C}^{\otimes})(A^{\otimes m_k+1},A)^{\partial K_{m_k+1}}\ar[d]\\
J^{(k)}_{\vec{m}\setminus\{m_k\}}(\mathcal{C}^{\otimes})(V^{(k)}(A)^{\otimes m_k+1},V^{(k)}(A))^{K_{m_k+1}}\ar[r] & J^{(k)}_{\vec{m}\setminus\{m_k\}}(\mathcal{C}^{\otimes})(V^{(k)}(A)^{\otimes m_k+1},V^{(k)}(A))^{\partial K_{m_k+1}}
}
\end{gathered}
\end{align}
whose top and the bottom square are cartesian. It thus suffices to show that the gap map of hom-spaces associated to the square 
(\ref{diag_cor_reciproc3}) is $\left(n-\sum_{i=0}^{k}m_i\right)$-truncated (as its fibers are spaces indeed). But the 
inclusion $\partial K_{m_k+1}\hookrightarrow K_{m_k+1}$ is $(m_k-3)$-connected; the forgetful functor
\[\mathrm{Alg}_{\mathsf{E}_{k}}(\mathcal{C}^{\otimes})(A^{\otimes m_k+1},A)\rightarrow J^{(k)}_{\vec{m}\setminus\{m_k\}}(\mathcal{C}^{\otimes})(V^{(k)}(A)^{\otimes m_k+1},V^{(k)}(A))\]
is $\left(n-1-\sum_{i=0}^{k-1}m_i\right)$-truncated by the proof of Proposition~\ref{prop_inisegloc}.
This implies that the gap map of (\ref{diag_cor_reciproc3}) is $\left(n-\sum_{i=0}^{k}m_i\right)$-truncated again by
\cite[Proposition 8.13]{rezkhtytps}. 
\end{proof}

We will require two more corollaries to apply the above results to the unital context as well. We will extract them by way of 
Definition~\ref{def_iniseg} and the Eckmann-Hilton argument of Proposition~\ref{prop_eh}. Perhaps one can derive them also directly from 
Schlank and Yanovsky's Eckmann-Hilton Argument \cite{schlankyanovski_eh}, however the author fails to see how. We therefore introduce an 
intermediate unital version of particularly simple initial segments.

\begin{definition}\label{def_iniseg_qu}
Let $\mathcal{C}^{\otimes}$ be a symmetric monoidal $\infty$-category. Let $J^{(0)+}_{\emptyset}(\mathcal{C}^{\otimes}):=\ast$ and 
$V^{(0)}\colon\mathcal{C}^{\otimes}\rightarrow J^{(0)+}_{\emptyset}(\mathcal{C}^{\otimes})$ be the essentially uniquely forgetful functor. 
Given $1\leq k<\infty$, let $\vec{1}$ be the sequence $(1,\dots,1)$ of length $k+1$. We recursively define
\begin{align}\label{diag_iniseg_0}
\begin{gathered}
\xymatrix{
\mathrm{Alg}_{\mathsf{E}_k}(\mathcal{C}^{\otimes})\ar@/_1pc/[ddr]_{\mathrm{Alg}_{\mathsf{E}_{1}}^{\mathrm{nu}}(V^{(k-1)})}\ar@/^1pc/[drr]^U\ar[dr]|{V^{(k)}} & & \\
& \mathrm{J}^{(k)+}_{\vec{1}}\ar[r]\ar[d]\ar@{}[dr]|(.3){\pbs} & \mathrm{Alg}_{\mathsf{E}_{k-1}}(\mathcal{C}^{\otimes})\ar[d]^{V^{(k-1)}}
\\
& \mathrm{Alg}_{\mathsf{E}_1}(J^{(k-1)+}_{(\vec{1}\setminus \{1\})}(\mathcal{C}^{\otimes}))\ar[r]_U & J^{(k-1)+}_{\vec{1}\setminus\{1\}}(\mathcal{C}^{\otimes}).
}
\end{gathered}
\end{align}
\end{definition}

\begin{lemma}\label{lemma_iniseg_qu}
Let $\mathcal{C}^{\otimes}$ be a symmetric monoidal $n$-category and $1\leq k<\infty$.
\begin{enumerate}
\item The forgetful functor 
\[V^{(k+1)}\colon\mathrm{Alg}_{\mathsf{E}_{k+1}}(\mathcal{C}^{\otimes})\rightarrow\mathrm{J}^{(k+1)+}_{\vec{1}}\]
is locally $\left(n-k-2\right)$-truncated.
\item The symmetric monoidal $\infty$-categories $J^{(k+1)+}_{\vec{1}}$ and
$\mathrm{Alg}_{L_{k+1}(\mathsf{E}_{\bullet})}(\mathcal{C}^{\otimes})$ are naturally equivalent, where $L_{k+1}(\mathsf{E}_{\bullet})$ is 
the $(k+1)$-st latching object of the co-semi-simplicial $\infty$-operad $\mathsf{E}_{\bullet}$ introduced in Section~\ref{sec_sub_comm}. 
\end{enumerate}
\end{lemma}
\begin{proof}
The forgetful functors
$U\colon\mathrm{Alg}_{\mathsf{E}_k}(\mathcal{C}^{\otimes})\rightarrow\mathrm{Alg}_{\mathsf{E}_k}^{\mathrm{nu}}(\mathcal{C}^{\otimes})$
induce a diagram
\begin{align}\label{diag_iniseg_qu}
\begin{gathered}
\xymatrix{
\mathrm{Alg}_{\mathsf{E}_{k+1}}(\mathcal{C}^{\otimes})\ar[dd]_{V^{(k+1)}}\ar[rr]^U\ar@{-->}[dr] & & \mathrm{Alg}_{\mathsf{E}_{k+1}}^{\mathrm{nu}}(\mathcal{C}^{\otimes})\ar[dd]^{V^{(k+1)}}\\
 & \bullet\ar[ur]\ar[dl]\ar@{}[dr]|(.3){\pbs} & \\
\mathrm{J}^{(k+1)+}_{\vec{1}}\ar[rr]_U & & \mathrm{J}^{(k+1)}_{\vec{1}}.
}
\end{gathered}
\end{align}
The vertical forgetful functor $V^{(k+1)}$ on the right hand side is locally $\left(n-1-\sum_{i=0}^{k}1\right)$-truncated by 
Proposition~\ref{prop_inisegloc}. To prove Part (1), it thus suffices to show that the diagonal gap map
$\mathrm{Alg}_{\mathsf{E}_k}(\mathcal{C}\otimes)\dashrightarrow\bullet$ in (\ref{diag_iniseg_qu}) is fully faithful. Both horizontal 
functors in the outer square of (\ref{diag_iniseg_qu}) are monic by way of Lemma~\ref{lemma_unitff}.1, and hence so is the gap map 
therein. It thus suffices to show that it is full. Therefore, let $A$ and $B$ be $\mathsf{E}_{k+1}$-algebras in $\mathcal{C}^{\otimes}$, 
and suppose we are given a pair of morphisms $f\colon U(A)\rightarrow U(B)$ in
$\mathrm{Alg}_{\mathsf{E}_{k+1}}^{\mathrm{nu}}(\mathcal{C}^{\otimes})$ together with an extension $f^+$ thereof to
$\mathrm{J}^{(k+1)+}_{\vec{1}}$. To lift $f$ along the top vertical forgetful functor $U$ of (\ref{diag_iniseg_qu}), it suffices to show 
that $f$ is quasi-unital (Lemma~\ref{lemma_unitff}.3). This is assured by the projection of $f^+$ to
$\mathrm{Alg}_{\mathsf{E}_{k}}(\mathcal{C}^{\otimes})$.

Part (2) follows fairly directly by induction from the formula for $L_{k}(\mathsf{E}_{\bullet})$ given in Remark~\ref{rem_latch}.
\end{proof}

\begin{corollary}\label{cor_truncmor}
Let $\mathcal{C}^{\otimes}$ be a symmetric monoidal $n$-category. Let $1\leq k<\infty$ be an integer and
$A,B\in\mathrm{Alg}_{\mathsf{E}_{k+1}}(\mathcal{C}^{\otimes})$ be $\mathsf{E}_{k+1}$-algebras in $\mathcal{C}^{\otimes}$. Let $f\colon A\rightarrow B$ be a 
morphism of underlying $\mathsf{E}_k$-algebras. Then the space 
\begin{align}\label{equ_truncmor}
\mathrm{Alg}_{\mathsf{E}_{k+1}}(\mathcal{C}^{\otimes})(A,B)\times_{\mathrm{Alg}_{\mathsf{E}_{k}}(\mathcal{C}^{\otimes})(A,B)}\{f\}
\end{align}
of $\mathsf{E}_{k+1}$-algebra extensions of $f$ is $(n-k-2)$-truncated.
\end{corollary}
\begin{proof}
By Corollary~\ref{cor_E_latching} and Lemma~\ref{lemma_iniseg_qu}.2 it follows that the diagram
\[\xymatrix{
\mathrm{Alg}_{\mathsf{E}_{k+1}}(\mathcal{C}^{\otimes})\ar[d]_{V^{(k+1)}}\ar[dr]^{U} & \\
J^{(k+1)+}_{\vec{1}}\ar[r] & \mathrm{Alg}_{\mathsf{E}_{k}}(\mathcal{C}^{\otimes})  
}\]
exhibits $U$ as a retract of $V^{(k+1)}$. The latter is locally $\left(n-k-2\right)$-truncated by 
Lemma~\ref{lemma_iniseg_qu}.1.
\end{proof}

\begin{corollary}\label{cor_truncobj}
Let $\mathcal{C}^{\otimes}$ be a symmetric monoidal $n$-category. Let $1\leq k<\infty$ be an integer and
$A\in\mathrm{Alg}_{\mathsf{E}_k}(\mathcal{C}^{\otimes})$ be an $\mathsf{E}_k$-algebra in $\mathcal{C}^{\otimes}$. Then the space 
\begin{align}\label{equ_truncobj}
\mathrm{Alg}_{\mathsf{E}_{k+1}}(\mathcal{C}^{\otimes})\times_{\mathrm{Alg}_{\mathsf{E}_{k}}(\mathcal{C}^{\otimes})}\{A\}
\end{align}
of $\mathsf{E}_{k+1}$-algebra extensions of $A$ is $(n-k-1)$-truncated.
\end{corollary}
\begin{proof}
If $k>n$ this follows from the Baez--Dolan--Lurie Stabilization Theorem. Assume $k\leq n$. The fiber (\ref{equ_truncobj}) is a subcategory 
of the sequential limit of $\infty$-categories
$\mathrm{Alg}_{\mathsf{A}_{m}}^{\mathrm{qu}}(\mathrm{Alg}_{\mathsf{E}_{k}}(\mathcal{C}^{\otimes}))\times_{\mathrm{Alg}_{\mathsf{E}_{k}}(\mathcal{C}^{\otimes})}\{A\}$, each of which is a space; it hence is a space itself. This can be again shown inductively (the total
$\infty$-category of any isofibration into a space whose fibers are spaces is also a space). As $k\leq n$, it suffices to show that the 
space (\ref{equ_truncobj}) is locally $(n-k-2)$-truncated. And indeed, for every pair of objects $A_1,A_2$ in the 
fiber (\ref{equ_truncobj}) their hom-space is the fiber
\[\mathrm{Alg}_{\mathsf{E}_{k+1}}(\mathcal{C}^{\otimes})(A_1,A_2)\times_{\mathrm{Alg}_{\mathsf{E}_{k}}(\mathcal{C}^{\otimes})(A,A)}\{1_A\}.\]
This is $(n-k-2)$-truncated by Corollary~\ref{cor_truncmor}.
\end{proof}
  
\subsection{Higher algebras in low dimensional monoidal categories}\label{sec_sub_iniseg_app}

We end this section with an application of Propositions~\ref{prop_inisegloc} and Propositions~\ref{prop_inisegfib} to low dimensions.

\begin{corollary}
Let $\mathcal{C}^{\otimes}$ be a symmetric monoidal 1-category. Then the category of $\mathsf{E}_2$-algebras in $\mathcal{C}^{\otimes}$ is equivalent to the category of diagrams in $\mathcal{C}$ of the form

\begin{align}\label{diag_eh_1}
\begin{gathered}
\xymatrix{
A\otimes A\otimes A\ar@<-.5ex>[d]_{m\otimes 1}\ar@<.5ex>[d]^{1\otimes m} & & \\
A\otimes A\ar[d]|{m} \ar[d] & (A\otimes A)^{\otimes 2}\ar[d]\ar[l]_{m\sprime\otimes m\sprime}\ar[d]^{m^{\otimes}} & \\
A & A\otimes A\ar[l]|{m\sprime} & A\otimes A\otimes A.\ar@<.5ex>[l]^{m\sprime\otimes 1}\ar@<-.5ex>[l]_{1\otimes m\sprime} & & \\
}
\end{gathered}
\end{align}
such that $m$ and $m\sprime$ have the same (quasi-)unit.
In other words, to give an $\mathsf{E}_{\infty}$-algebra in $\mathcal{C}^{\otimes}$ is the same as to give two monoidal structures on an 
object $A$ in $\mathcal{C}^{\otimes}$ that share the same unit and make the square in (\ref{diag_eh_1}) commute.
\end{corollary}
\begin{proof}
The forgetful functor
\[\mathrm{Alg}_{\mathsf{E}_1}^{\mathrm{nu}}(\mathcal{C}^{\otimes})\twoheadrightarrow \mathrm{Alg}_{\mathsf{A}_{3}}^{\mathrm{nu}}(\mathcal{C}^{\otimes})\]
is an equivalence. Furthermore, the forgetful functors
\[U^{(1)}\colon\mathrm{Alg}_{\mathsf{A}_2}^{\mathrm{nu}}(\mathrm{Alg}_{\mathsf{A}_{3}}^{\mathrm{nu}}(\mathcal{C}^{\otimes}))\rightarrow I^{(1)}_{(2,1)}\]
and 
\[U^{(1)}\colon\mathrm{Alg}_{\mathsf{E}_2}^{\mathrm{nu}}(\mathcal{C}^{\otimes})\xrightarrow{\simeq}\mathrm{Alg}_{\mathsf{A}_{3}}^{\mathrm{nu}}(\mathrm{Alg}_{\mathsf{A}_{3}}^{\mathrm{nu}}(\mathcal{C}^{\otimes}))\rightarrow I^{(1)}_{(1,2)}\]
are equivalences by Proposition~\ref{prop_inisegloc} as well. The space $I^{(1)}_{(1,2)}$ consists of diagrams of shape (\ref{diag_eh_1}) 
by replacing $I^{(1)}_{(2,1)}$ for
$\mathrm{Alg}_{\mathsf{A}_2}^{\mathrm{nu}}(\mathrm{Alg}_{\mathsf{A}_{3}}^{\mathrm{nu}}(\mathcal{C}^{\otimes}))$ in
Definition~\ref{def_iniseg}.
Thus, every diagram of the form (\ref{diag_eh_1}) gives rise to an essentially unique non-unital $\mathsf{E}_2$-algebra $A$ in
$\mathcal{C}^{\otimes}$. By Lemma~\ref{lemma_unitff}.4, it extends to an $\mathsf{E}_2$-algebra if the the multiplication $m$ has a 
quasi-unit $u$, and the multiplication $m\sprime$ is quasi-unital and exhibits a quasi-unit $u\sprime$ in
$\mathrm{Alg}_{\mathsf{E}_1}(\mathcal{C}^{\otimes})$ itself. These three conditions are satisfied whenever $m$ and $m\sprime$ have the same 
quasi-unit. Vice versa, whenever $A$ extends to an $\mathsf{E}_2$-algebra, these two quasi-units are necessarily the same.
We thus obtain a $(-1)$-truncated fibration
\[U^{(1)}\colon\mathrm{Alg}_{\mathsf{E}_2}(\mathcal{C}^{\otimes})\rightarrow I^{(1)}_{(1,2)}\]
whose image consists of those diagrams (\ref{diag_eh_1}) such that $m$ and $m\sprime$ have the same quasi-unit.
Lastly, the projection
\[\mathrm{Alg}_{\mathsf{E}_{\infty}}(\mathcal{C}^{\otimes})\twoheadrightarrow\mathrm{Alg}_{\mathsf{E}_2}(\mathcal{C}^{\otimes})\]
is an equivalence by the Baez--Dolan--Lurie Stabilization Theorem.
\end{proof}

\begin{remark}\label{rem_eh_1}
The classical Eckmann Hilton argument is essentially the observation that every diagram of the form (\ref{diag_eh_1}) allows to pre-compose 
the square therein with the map $(1,u,u,1)\colon A\otimes A\rightarrow (A\otimes A)^{\otimes 2}$. This gives $m=m\sprime$ and subsequently 
commutativity of $m$ by the square itself.
\end{remark}

In dimensions 2 and 3 we obtain the following reductions.

\begin{corollary}\label{cor_eh_2}
Let $\mathcal{C}^{\otimes}$ be a symmetric monoidal 2-category. Then the 2-category of $\mathsf{E}_2$-algebras in $\mathcal{C}^{\otimes}$ is equivalent to the 2-category of diagrams of the form
\begin{align}\label{diag_eh_2}
\begin{gathered}
\xymatrix{
A\otimes A\otimes A\otimes A\ar@<-.5ex>@/_1pc/[d]_{m\otimes 1\otimes 1}\ar[d]|{1\otimes m\otimes 1}\ar@<.5ex>@/^1pc/[d]^{1\otimes 1\otimes m} & & & \\
A\otimes A\otimes A\ar@<-.5ex>[d]_{m\otimes 1}\ar@<.5ex>[d]^{1\otimes m} & (A\otimes A)^{\otimes 3}\ar[l]_{(m\sprime, m\sprime, m\sprime)}\ar@<-.5ex>[d]_{m^{\otimes}\otimes 1}\ar@<.5ex>[d]^{1\otimes m^{\otimes}} & & \\
A\otimes A\ar[d]|{m} \ar[d]\ar@{}[dr]|{\underset{\chi}{\Rightarrow}} & (A\otimes A)^{\otimes 2}\ar[d]\ar[l]_{(m\sprime, m\sprime)}\ar[d]^{m^{\otimes}} & (A\otimes A\otimes A)^{\otimes 2}\ar@<.5ex>[l]^{(m\sprime\otimes 1)^{\otimes 2}}\ar@<-.5ex>[l]_{(1\otimes m\sprime)^{\otimes 2}}\ar[d]^{m^{\otimes}}  & \\
A & A\otimes A\ar[l]|{m\sprime} & A\otimes A\otimes A\ar@<.5ex>[l]^{m\sprime\otimes 1}\ar@<-.5ex>[l]_{1\otimes m\sprime} & A\otimes A\otimes A\otimes A.\ar@<-.5ex>@/_1pc/[l]_{m\sprime\otimes 1\otimes 1}\ar[l]|{1\otimes m\sprime\otimes 1}\ar@<.5ex>@/^1pc/[l]^{1\otimes 1\otimes m\sprime} \\
}
\end{gathered}
\end{align}
such $m$ has a quasi-unit $u$, $m\sprime$ is quasi-unital, and such that $u$ induces a quasi-unit for $m\sprime$ itself. Here, $\chi$ is an 
invertible 2-cell.
\end{corollary}
\begin{proof}
Same proof.
\end{proof}

\begin{example}\label{exple_eh_2}
Consider the cartesian monoidal $2$-category $\mathrm{Cat}^{\times}$ of categories. The associated $\infty$-category of monoidal 
categories is equivalent to the $\infty$-category of $\mathsf{E}_1$-monoids in $\mathrm{Cat}^{\times}$ (e.g.\ by
\cite[Example 5.1.2.4]{lurieha}, which is essentially an argument analogous to Lemma~\ref{lemma_grmon=E1} below). By the Additivity 
Theorem, it follows that an $\mathsf{E}_2$-algebra in $\mathrm{Cat}^{\times}$ is equivalently given by an $\mathsf{E}_1$-algebra in
$\mathrm{Alg}_{\mathsf{E}_1}(\mathrm{Cat}^{\times})$. By Corollary~\ref{cor_eh_2} such can be reduced to diagrams of the form 
(\ref{diag_eh_2}). By the Eckmann--Hilton argument it follows that $m\sprime\simeq m$. 
Furthermore, fixing the given (vertical) $\mathsf{A}_2^{\mathrm{nu}}\otimes_{\mathrm{BV}}\mathsf{E_1}^{\mathrm{nu}}$-algebra structure, the 
bottom 2-dimensional horizontal $\mathsf{E}_1^{\mathrm{nu}}\otimes \mathsf{A}_2^{\mathrm{nu}}$-algebra is determined by the existence of a 
3-cell in $\mathrm{Cat}^{\times}$, and hence is a mere property. This property is automatic by a diagonal reflection argument that we 
will expound on in detail in Section~\ref{sec_translation}. It follows that the $\infty$-category of $\mathsf{E}_2$-monoids in
$\mathrm{Cat}^{\times}$ is equivalent to the $\infty$-category of quasi-unital $\mathsf{A}_2$-monoids in
$\mathrm{Alg}_{\mathsf{E}_1}(\mathrm{Cat}^{\times})$. This recovers Joyal--Street's equivalence between braided monoidal categories and 
monoidal categories with multiplications \cite[Remark 5.1]{joyalstreet}.
\end{example}

\begin{corollary}\label{cor_eh_3}
Let $\mathcal{C}^{\otimes}$ be a symmetric monoidal 3-category. Then the 3-category of $\mathsf{E}_2$-algebras in $\mathcal{C}^{\otimes}$ is equivalent to the 3-category of diagrams of the form
\begin{align}\label{diag_eh_3}
\begin{gathered}
\xymatrix{
A^{\otimes 5}\ar@<1ex>[d]\ar@<.5ex>[d]\ar[d]\ar@<-.5ex>[d]\ar@<-1ex>[d] & & & & \\
A^{\otimes 4} \ar@<-.5ex>@/_1pc/[d]_{\rotatebox{90}{$\scriptstyle m\otimes 1\otimes 1$}}\ar[d]|{1\otimes m\otimes 1}\ar@<.5ex>@/^1pc/[d]^{\rotatebox{-90}{$\scriptstyle 1\otimes 1\otimes m$}} & (A^{\otimes 2})^{\otimes 4}\ar[l]\ar@<.5ex>[d]\ar[d]\ar@<-.5ex>[d] & & & \\
A^{\otimes 3} \ar@<-.5ex>[d]_{\rotatebox{90}{$\scriptstyle m\otimes 1$}}\ar@<.5ex>[d]^{\rotatebox{-90}{$\scriptstyle 1\otimes m$}}\ar@{}[dr]|{\underset{\omega}{\Rrightarrow}} & (A^{\otimes 2})^{\otimes 3}\ar[l]^{(m\sprime)^{\otimes 3}}\ar@<-.5ex>[d]_{\rotatebox{90}{$\scriptstyle m^{\otimes}\otimes 1$}}\ar@<.5ex>[d]^{\rotatebox{-90}{$\scriptstyle 1\otimes m^{\otimes}$}} & (A^{\otimes 3})^{\otimes 3}\ar@<.5ex>[l]\ar@<-.5ex>[l]\ar@<.5ex>[d]\ar@<-.5ex>[d] & & \\
A^{\otimes 2}\ar[d]|{m}\ar[d]\ar@{}[dr]|{\underset{\chi}{\Rightarrow}} & (A^{\otimes 2})^{\otimes 2}\ar[d]\ar[l]_{(m\sprime)^{\otimes 2}}\ar[d]^{m^{\otimes}}\ar@{}[dr]|{\underset{\psi}{\Rrightarrow}} & (A^{\otimes 3} )^{\otimes 2}\ar@<.5ex>[l]^{(m\sprime\otimes 1)^{\otimes 2}}\ar@<-.5ex>[l]_{(1\otimes m\sprime)^{\otimes 2}}\ar[d]^{m^{\otimes}} & (A^{\otimes 4})^{\otimes 2}\ar@<.5ex>[l]\ar[l]\ar@<-.5ex>[l]\ar[d] & \\
A & A^{\otimes 2}\ar[l]|{m\sprime} & A^{\otimes 3} \ar@<.5ex>[l]^{m\sprime\otimes 1}\ar@<-.5ex>[l]_{1\otimes m\sprime} & A^{\otimes 4} \ar@<-.5ex>@/_1pc/[l]_{m\sprime\otimes 1\otimes 1}\ar[l]|{1\otimes m\sprime\otimes 1}\ar@<.5ex>@/^1pc/[l]^{1\otimes 1\otimes m\sprime} & A^{\otimes 5}.\ar@<1ex>[l]\ar@<.5ex>[l]\ar[l]\ar@<-.5ex>[l]\ar@<-1ex>[l]
}
\end{gathered}
\end{align}
such $m$ has a quasi-unit $u$, $m\sprime$ is quasi-unital, and such that $u$ induces a quasi-unit for $m\sprime$ itself.
Here, $\chi$ is an invertible 2-cell, $\omega$ and $\theta$ are invertible 3-cells, and the unlabelled squares denote equations.
\end{corollary}
\begin{proof}
Same proof.
\end{proof}

Corollary~\ref{cor_eh_3} is the most relevant of the three corollaries, given that it is what we will apply to the cartesian monoidal
3-category $\mathrm{BiCat}^{\times}$ in Section~\ref{sec_translation}. While diagrams of the form (\ref{diag_eh_3}) are still a considerable 
amount of data to manage by hand, we will see that some of its data can be further reduced by way of various semi-strictification results 
available in this case.

\section{The cartesian monoidal $\infty$-category of bicategories}\label{sec_modbicat}

The canonical embedding of category theory in $\infty$-category theory translates many categorified algebraic constructions to
their $\infty$-categorical counterparts. This is the starting point for much of \cite{lurieha}. Thus, the former can be understood as a 
sub-theory of the latter.
In contrast to a 1-category, a general bicategory $\mathcal{C}$ has non-invertible 2-cells and hence does not reduce to an
$(\infty,1)$-categorical structure without further ado. In particular, to equip a general bicategory $\mathcal{C}$ with a monoidal structure 
requires more than to equip its underlying locally groupoidal bicategory with one. However, the homotopy theory of bicategories in its 
entirety does define a cartesian monoidal $\infty$-category $\mathrm{BiCat}^{\times}$ abstractly either way. The goal of this section is thus 
to compare some basic notions of the $\infty$-categorical algebra in $\mathrm{BiCat}^{\times}$ as discussed generally in 
Section~\ref{sec_infty} to the corresponding notions of the bicategorical algebra of \cite{daystreet}. We will do so by way of the 
homotopical algebra provided by Lack's model structure on the category of bicategories. For simplicity, we will work with Gray monoids in 
place of monoidal bicategories, which will be justified along the way.
	
In Section~\ref{sec_sub_graymonalg} we recall the most basic semi-strict algebraic definitions of \cite{daystreet}, and construct reverses 
of braidings and of syllepses on (braided) monoidal 2-categories. In Section~\ref{sec_sub_graymonhtythy} we 
define the homotopy theory $\mathrm{GrMon}$ of Gray monoids by folklore means and discuss some of the notions from 
Section~\ref{sec_sub_graymonalg} in this context.


\subsection{The semi-strict models of 2-dimensional algebra}\label{sec_sub_graymonalg}
We recall that a Gray monoid is a (strict) monoid in the monoidal category
$\mathrm{Gray}:=(2\text{-Cat},\otimes_{\mathrm{Gr}})$ of 2-categories and (strict) 2-functors equipped with the Gray tensor product. As 
such, every Gray monoid is in particular a monoidal bicategory. Vice versa, there is the following fundamental theorem.

\begin{theorem}\label{thm_bito2cat}
Every monoidal bicategory $\mathcal{C}$ is monoidally biequivalent to a Gray monoid $\bar{\mathcal{C}}$. 
\end{theorem}
\begin{proof}
This is a popular special case of the coherence theorem for tricategories \cite{gps_coherence}. 
\end{proof}

Theorem~\ref{thm_bito2cat} essentially states that a pseudomonoid in the tricategory of bicategories equipped with the strict cartesian 
monoidal structure is essentially the same as a strict monoid in the tricategory of 2-categories equipped with the pseudo-cartesian 
monoidal structure $\otimes_{\mathrm{Gr}}$. It thus allows us to reduce the proof of Theorem~\ref{thm_main} in the Introduction to a 
theorem about $2$-categories and $2$-functors between them. This is not strictly necessary for the proof, but it simplifies some of the 
involved calculations considerably. Indeed, the basic notions of bicategorical algebra are equivalence-invariant, and so we obtain the 
following corollary.

\begin{corollary}
Let $\mathcal{C}$ be a monoidal bicategory and let $\bar{\mathcal{C}}$ be the associated Gray monoid from 
Theorem~\ref{thm_bito2cat}. Then there is a 1-1 correspondence between
\begin{enumerate}
\item braidings on $\mathcal{C}$ and braidings on $\bar{\mathcal{C}}$;
\item sylleptic braidings on $\mathcal{C}$ and sylleptic braidings on $\bar{\mathcal{C}}$;
\item symmetric braidings on $\mathcal{C}$ and symmetric braidings on $\bar{\mathcal{C}}$.
\end{enumerate}
\end{corollary}\qed


\begin{remark}\label{rem_ssvsfw}
We make the following remarks about semi-strict and fully weak 2-dimensional algebraic structures.
\begin{itemize}
\item Gray monoids are semi-strict monoidal bicategories. They are however essentially equivalent to monoidal bicategories by the theorem 
above.
\item The monoidal homomorphisms between Gray monoids of \cite[Definition 2]{daystreet} (i.e.\ the strong monoidal pseudo-functors) are 
precisely the monoidal bifunctors between Gray monoids, and hence fully weak already.
\item The (semi-strict) braidings on a Gray monoid defined in \cite[Definition 12]{daystreet} differ from the fully weak bicategorical 
braidings on a Gray monoid \cite[Section 2.4]{gurski_braid} by the strict extra unital equations they satisfy. The former however are 
equivalent to the latter on any given Gray monoid by \cite[Theorem 27]{gurski_braid}. We will refer to the former simply as braidings, and to 
the latter as fully weak braidings. We will further discuss this choice in Remark~\ref{rem_sstricttoweakbraids}.
\item Syllepses on a braided Gray monoid in \cite[Definition 15]{daystreet} are precisely bicategorical syllepses
\cite[Definition 1.1]{gurskiosorno_sym} thereon (assuming the braiding is semi-strict). 
\item Symmetric syllepses on braided Gray monoids in \cite[Definition 17]{daystreet} are precisely bicategorical symmetric syllepses thereon
\cite[Definition 1.1]{gurskiosorno_sym}.
\item Braided/sylleptic homomorphisms between braided/sylleptic/symmetric Gray monoids in \cite[Definition 14, Definition 16]{daystreet} are 
precisely the fully weak bicategorical notions.
\end{itemize}
The study of the higher category of semi-strict objects with fully weak morphisms between them is a principle ubiquitous in 
3-dimensional monad theory, see e.g.\ \cite{garner_cofmndthy}. This automatically induces a canonical fully faithful embedding from the 
tricategory of Gray monoids and monoidal homomorphisms to the tricategory of monoidal bicategories. 
\end{remark}

\begin{notation}
A monoidal homomorphism between Gray monoids whose underlying pseudo-functor is a 2-functor will be referred to as a monoidal 2-functor.
Given two Gray monoids $\mathcal{C}$ and $\mathcal{D}$, we will often denote a monoidal 2-functor by a capital letter
$F\colon\mathcal{C}\rightarrow\mathcal{D}$, that is, a tuple $F=(f,\chi,\omega,\iota,\zeta,\kappa)$ where
$f\colon\mathcal{C}\rightarrow\mathcal{D}$ is a 2-functor and $(\chi,\omega,\iota,\zeta,\kappa)$ is a monoidal structure in the sense of 
Day--Street thereon. Monoidal pseudo-natural transformations between monoidal 2-functors, and monoidal modifications between monoidal 
pseudo-natural transformations are subsequently defined in \cite[Definition 3]{daystreet}.
\end{notation}

\begin{definition}\label{def_equiv_braids}
Let $\mathcal{C}$ be a Gray monoid. Say two (fully weak) braidings $\beta_1$, $\beta_2$ on $\mathcal{C}$ are equivalent if the identity
$1\colon\mathcal{C}\rightarrow\mathcal{C}$, considered as a monoidal 2-functor, can be enhanced to a braided monoidal 2-functor
$1\colon(\mathcal{C},\beta_1)\rightarrow(\mathcal{C},\beta_2)$. We define the set $\beta(\mathcal{C})$ of equivalence classes of braidings 
on $\mathcal{C}$. Given a braiding $\beta$ on $\mathcal{C}$, we further define the set $\zeta(\mathcal{C},\beta)$ of syllepses on the 
braided Gray monoid $(\mathcal{C},\beta)$.
\end{definition}

In the following we will freely use the definitions of \cite{daystreet} and stick to their notation.

\begin{notation}\label{notation_hinv}
Let $\mathcal{C}$ be a 2-category. Given two invertible morphisms $f,g\colon A\rightarrow B$ and an invertible 2-cell
$\alpha\colon f\rightarrow g$ in $\mathcal{C}$, we denote by $\alpha^{-1_h}\colon f^{-1}\rightarrow  g^{-1}$ the horizontal inverse of
$\alpha$ given by the whiskering $f^{-1}\ast\alpha^{-1}\ast g^{-1}$. Given a 2-cell $\alpha\colon if\rightarrow gh$ between two compositions, 
depicted as a square	
\[\xymatrix{
A\ar[r]^f\ar[d]_h\ar@{}[dr]|{\Downarrow\alpha} & B\ar[d]^i \\
C\ar[r]_g & D,
}\]
we will refer to the analogously induced 2-cell
\[\xymatrix{
B\ar[rr]^{f^{-1}}\ar[d]_i\ar@{}[drr]|{\Downarrow f^{-1}\ast\alpha^{-1}\ast g^{-1}} & & A\ar[d]^h \\
D\ar[rr]_{g^{-1}} & & C,
}\]
as $\alpha^{-1_h}$ as well. Its vertical inverse 
\[\xymatrix{
C\ar[rr]^{g}\ar[d]_{h^{-1}}\ar@{}[drr]|{\Downarrow h^{-1}\ast\alpha^{-1}\ast \iota^{-1}} & & D\ar[d]^{\iota^{-1}} \\
A\ar[rr]_{f} & & B,
}\]
will be denoted by $\alpha^{-1_v}$. Concatenation of these two assignments yields a 2-cell
\[\xymatrix{
D\ar[r]^{g^{-1}}\ar[d]_{\iota^{-1}}\ar@{}[dr]|{\Downarrow \alpha^{-1_{vh}}} & C\ar[d]^{h^{-1}} \\
B\ar[r]_{f^{-1}}\ar[r] & A.
}\]
given by $(if)^{-1}\ast\alpha\ast(gh)^{-1}$. Here, given a horizontally composable pair of 2-cells
\[\xymatrix{
A\ar[r]^f\ar[d]_h\ar@{}[dr]|{\Downarrow\alpha} & B\ar[d]^i\ar[r]^k\ar@{}[dr]|{\Downarrow\beta} & E\ar[d]^m \\
C\ar[r]_g & D\ar[r]_l & F,
}\]
note that $(\beta\ast\alpha)^{-1_{vh}}=\alpha^{-1_{vh}}\ast\beta^{-1_{vh}}$.
\end{notation}

\subsubsection{Braidings on Gray monoids and their reverses}
It is well-known that braidings are dualisable notions. Indeed, every braiding $\beta$ on a monoidal category $\mathcal{C}^{\otimes}$ induces 
a reverse braiding $\beta^{\mathrm{rev}}$ on $\mathcal{C}^{\otimes}$ given pointwise by $\beta^{\mathrm{rev}}_{A,B}:=\beta_{B,A}^{-1}$ 
\cite[Definition 2.1]{joyalstreet}. Analogously, there is the following proposition.

\begin{proposition}\label{prop_defrevbraid}
Let $\mathcal{C}$ be a Gray monoid and $\beta=(\rho,\bar{\omega})$ be a braiding on $\mathcal{C}$. Then the tuple
$\beta^{\mathrm{rev}}:=(\rho^{\mathrm{rev}},\bar{\omega}^{\mathrm{rev}})$ given by
\[\rho^{\mathrm{rev}}_{A,B}:=\rho_{B,A}^{-1}\colon A\otimes B\xrightarrow{\simeq} B\otimes A\]
and
\[\xymatrix{
A\otimes B\otimes C\otimes D\ar[d]_{\rho^{\mathrm{rev}}_{A,B}\otimes 1}\ar[r]^{1\otimes \rho^{\mathrm{rev}}_{C,D}}\ar@{}[dr]|{\Rightarrow \bar{\omega}^{\mathrm{rev}}_{A,B,C,D}} & A\otimes B\otimes D\otimes C\ar[d]^{\rho^{\mathrm{rev}}_{A,B\otimes D}\otimes 1} \\
B\otimes A\otimes C\otimes D\ar[r]_{1\otimes \rho^{\mathrm{rev}}_{A\otimes C,D}} & B\otimes D\otimes A\otimes C
}
\]
defined as the pasting
\begin{align}\label{diag_defrevbraid}
\left(
\begin{gathered}
\xymatrix{
A\otimes B\otimes C\otimes D & & A\otimes B\otimes D\otimes C\ar[ll]_{1\otimes\rho_{D,C}}  \\
 & B\otimes A\otimes D\otimes C\ar[dl]|{1\otimes\rho_{D,C}}\ar[ur]|{\rho_{B,A}\otimes 1}\ar@{}[ul]|{1}\ar@{}[d]|{\Downarrow 
1\otimes\bar{\omega}_{D,A,I,C}}\ar@{}[r]|(.7){\underset{\bar{\omega}_{B,I,D,A}\otimes 1}{\Leftarrow}} & \\
B\otimes A\otimes C\otimes D\ar[uu]^{\rho_{B,A}\otimes 1} & & B\otimes D\otimes A\otimes C\ar[ll]^{1\otimes\rho_{D,A\otimes C}}\ar[uu]_{\rho_{B\otimes D,A}\otimes 1}\ar[ul]|{1\otimes\rho_{D,A}\otimes 1}
}\end{gathered}
\right)^{-1_{vh}}
\end{align}
is again a braiding on $\mathcal{C}$. We refer to $\beta^{\mathrm{rev}}$ as the reverse braiding on $\mathcal{C}$ 
associated to $\beta$.
\end{proposition}
\begin{proof}
From $\bar{\omega}_{-.I,I,-}=1$ and the associativity law for $\bar{\omega}$ it follows that $\bar{\omega}_{I,I,-,-}=1$ (using 
associativity for $1$ substituted in the left most four variables) and $\bar{\omega}_{-,-,I,I}=1$ (using associativity for $1$ substituted 
in the right most four variables) as well. It thus follows from the definition that
$\bar{\omega}^{\mathrm{rev}}_{A,I,C,D}=\bar{\omega}_{D,A,I,C}$ and $\bar{\omega}^{\mathrm{rev}}_{A,B,I,D}=\bar{\omega}_{B,I,D,A}$. Hence, 
the equation relating the two Young--Baxter cells associated to
$r^{\mathrm{rev}}_{A,B|C}:=\bar{\omega}^{\mathrm{rev}}_{A,I,B,C}$ and $r^{\mathrm{rev}}_{A|B,C}:=\bar{\omega}^{\mathrm{rev}}_{A,B,I,C}$
for $\beta^{\mathrm{rev}}$ follows directly from the same equation relating the two
Young--Baxter cells for $\beta$. Direct verification of associativity for $\bar{\omega}^{\mathrm{rev}}$ is somewhat tricky. 
However, according to \cite[p.109]{daystreet}, one may equivalently verify Baez--Neuchl's corresponding axioms for the definition of a 
braiding on a monoidal 2-category --- as more legibly presented by Crans \cite{crans_centers} --- instead.
These axioms are phrased in terms of the cells $r^{\mathrm{rev}}_{A,B|C}=r_{C|A,B}^{{-1_{hv}}}$ and
$r^{\mathrm{rev}}_{A|B,C}=r_{B,C|A}^{{-1_{hv}}}$\footnote{We stick to
Day--Street's notation for consistency; to use the axioms in \cite{crans_centers}, one has to define $\tilde{R}_{A,B|C}:=r_{A,B|C}$ and
$\tilde{R}_{A|B,C}=r_{A|B,C}^{-1}$.}. These follow fairly directly; Equation~(2.4) in \cite[Definition 2.2]{crans_centers} for
$\beta^{\mathrm{rev}}$ follows from Equation~(2.5) in \cite[Definition 2.2]{crans_centers} for $\beta$, and vice versa. Equation~(2.6) in
\cite[Definition 2.2]{crans_centers} is self-dual in this regard.
\end{proof}

%
\begin{remark}
Given a braided Gray monoid $(\mathcal{C},\beta)$, one can show that $\bar{\omega}(I,-,-,-)=1$ and $\bar{\omega}(-,-,-,I)=1$ follow, too, 
by associativity for the tuple $(A,B,C,D,E,F)$ with $A=B=C=1$ and $D=E=F=1$, respectively.
\end{remark}

\begin{remark}
We recall from \cite[Definition 2.1]{joyalstreet} that in the 1-categorical context the reverse braiding of a braiding is again a braiding 
because ``(B2) is just obtained from (B1) by replacing $\rho$ with $\rho^{\mathrm{rev}}$''. Here, (B1) and (B2) are the two
cells $r_{-,-|-}$ and $r_{-|-,-}$. And indeed, computation of the cells $r^{\mathrm{rev}}_{A,B|C}:=r_{C|A,B}^{{-1_{hv}}}$ and 
$r^{\mathrm{rev}}_{A|B,C}:=r_{B,C|A}^{{-1_{hv}}}$ is more intuitive than that of $\bar{\omega}^{\mathrm{rev}}$ in the 2-categorical context 
as well.
\end{remark}

\begin{notation}\label{notation_revbraid}
We recall the natural isomorphism $\sigma$ part of the symmetric monoidal category $\mathrm{Fin}_{\ast}$ from 
Section~\ref{sec_sub_revmon_infty}. Let $\mathcal{C}$ be a Gray monoid and $\beta=(\rho,\bar{\omega})$ be a braiding on $\mathcal{C}$. Its 
associated reverse braiding $\beta^{\mathrm{rev}}$ induces a modification $(1\otimes(\bar{\omega}^{\mathrm{rev}})^{-1}\otimes 1)(\sigma_{(\langle 3\rangle,\langle 2\rangle)}^{\ast}(-))$, which according to 
(\ref{diag_defrevbraid}) is defined by way of the natural 3-cell
\[\xymatrix{
A\otimes B\otimes C\otimes D \otimes E\otimes F\ar[rr]^{1\otimes\rho_{B\otimes C,D}\otimes 1\otimes 1}\ar[dd]_{1\otimes 1\otimes\rho_{C,D\otimes E}\otimes 1}\ar[dr]|{1\otimes 1\otimes\rho_{C,D}\otimes 1\otimes 1}  & & A\otimes D\otimes B\otimes C\otimes E\otimes F\ar[dd]^{1\otimes 1\otimes 1\otimes \rho_{C,E}\otimes 1}  \\
 & A\otimes B\otimes D\otimes C\otimes E\otimes F\ar[ur]|{1\otimes\rho_{B,D}\otimes 1\otimes 1\otimes 1}\ar[dl]|{1\otimes 1\otimes 1\otimes\rho_{C,E}\otimes 1}\ar@{}[dr]|{1}\ar@{}[u]|{\Uparrow 1\otimes\bar{\omega}^{-1}_{B,I,C,D}\otimes 1\otimes 1}\ar@{}[l]|(.7){\underset{1\otimes 1\otimes \bar{\omega}^{-1}_{C,D,I,E}\otimes 1}{\Rightarrow}} & \\
A\otimes B\otimes D\otimes E\otimes C\otimes F\ar[rr]_{1\otimes\rho_{B,D}\otimes 1\otimes 1\otimes 1} & & A\otimes D\otimes B\otimes E\otimes C\otimes F.
}\]
In suggestive analogy to Section~\ref{sec_sub_revmon_infty}, we refer to this cell as the diagonal reflection of
$1\otimes \bar{\omega}^{\mathrm{rev}}\otimes 1$ and denote it by $(1\otimes \bar{\omega}^{\mathrm{rev}}\otimes 1)^{\rho}$.
For later reference, we note that this cell $(1\otimes \bar{\omega}^{\mathrm{rev}}\otimes 1)^{\rho}$ can be defined for any pair of 
invertible cells $(\rho,\bar{\omega})$ underlying a braiding whenever just $\rho_{I,-}=\rho_{-,I}=1$. Neither axiom on $\bar{\omega}$ (associativity or the equation relating the two associated Yang--Baxter cells) is required.
\end{notation}

\begin{remark}\label{rem_braidtomono}
In Example~\ref{exple_eh_2} we discussed the 1-1 correspondence between braidings on a monoidal category $\mathcal{C}$ and monoidal 
structures on its multiplication \cite[Proposition 5.3]{joyalstreet}. Analogously, let $\beta=(\rho,\bar{\omega})$ be a braiding
on a Gray monoid $\mathcal{C}$. Day and Street \cite[Definition 12]{daystreet} observe that each such braiding induces a monoidal structure 
on the multiplication 2-functor $m\colon\mathcal{C}\otimes_{\mathrm{Gr}}\mathcal{C}\rightarrow\mathcal{C}$ of $\mathcal{C}$ with
\begin{itemize}
\item $\chi_{\beta}:=1\otimes \rho\otimes 1$,
\item $\omega_{\beta}:=1\otimes \bar{\omega}\otimes 1$,
\item $\iota\colon I\rightarrow m(I,I)$ is the identity on $I$,
\item $\zeta\colon m(-,I)\simeq 1$, $\kappa\colon m(I,-)\simeq 1$ are both the identity on $\mathcal{C}$,
\end{itemize}
and such that the identity
\begin{align}\label{diag_def_braiding}
\begin{gathered}
\xymatrix{
\mathcal{C}\otimes_{\mathrm{Gr}}\mathcal{C}\otimes_{\mathrm{Gr}}\mathcal{C}\ar[r]^(.6){1\otimes_{\mathrm{Gr}}m}\ar[d]_{m\otimes_{\mathrm{Gr}}1}\ar@{}[dr]|1 & \mathcal{C}\otimes_{\mathrm{Gr}}\mathcal{C}\ar[d]^m \\
\mathcal{C}\otimes_{\mathrm{Gr}}\mathcal{C}\ar[r]_m & \mathcal{C}
}
\end{gathered}
\end{align}
considered as a pseudo-natural transformation of (monoidal) 2-functors becomes itself monoidal specifically when equipped with the 
modifications $1_0=1$ and $1_2:=(1\otimes\bar{\omega}^{\mathrm{rev}}\otimes 1)^{\rho}$. Semi-strictness of the braiding further implies
that $\chi_{I,-},\chi_{-,I}\colon m(-,-)\simeq m(-,-)$ are both the identity on $\mathcal{C}$. It further follows that 
$\omega(-,-,1,1,-,-)= 1$ from the unitality axiom that determines the interaction of $\omega$ with $\zeta$ and $\kappa$.
\end{remark}

%

Analogously, we can reverse braidings on monoidal 2-functors, and subsequently, syllepses on braided Gray monoids.
Therefore, we slightly reorganise the diagrams in the definition of a braiding on a monoidal 2-functor \cite[Definition 14]{daystreet} to 
render them more symmetric. This makes the proofs much easier.

\begin{lemma}
Let $(\mathcal{C,\beta})$ and $(\mathcal{D},\beta)$ be braided Gray monoids and $F\colon\mathcal{C}\rightarrow\mathcal{D}$ be a monoidal
2-functor (or a weak monoidal homomorphism for that matter). Then a braiding $u$ on $F$ is an invertible modification
\[\xymatrix{
f(A)f(B)\ar[r]^{\rho_{f(A),f(B)}}\ar[d]_{\chi_{A,B}}\ar@{}[dr]|{\underset{u_{A,B}}{\Rightarrow}} & f(B)f(A)\ar[d]^{\chi_{B,A}} \\
f(AB)\ar[r]_{f(\rho_{A,B})} & f(BA)
}\]
such that the following two equations hold.
\begin{align}\label{diag_def_braidmor_I}
\begin{gathered}
\adjustbox{scale=0.7}{%
\xymatrix{
 & f(A)f(B)f(C)\ar[rr]^{\rho_{f(A),f(B)f(C)}}\ar[d]^{\rotatebox{-90}{$\scriptstyle 1_{f(A)}$}}_{\rotatebox{-90}{$\scriptstyle \otimes\chi_{B,C}$}}\ar@/_1pc/[dl]_{\chi_{A,B}\otimes 1_{f(C)}}\ar@{}[drr]|{\Uparrow\rho_{1,\chi_{B,C}}} & & f(B)f(C)f(A)\ar[d]^{\rotatebox{-90}{$\scriptstyle \chi_{B,C}$}}_{\rotatebox{-90}{$\scriptstyle \otimes 1_{f(A)}$}}\ar@/^1pc/[dr]^{1_{f(B)}\otimes\chi_{C,A}} & \\
f(AB)f(C)\ar@/_1pc/[dr]_{\chi_{AB,C}}\ar@{}[r]|{\underset{\omega_{A,B,C}}{\Rightarrow}} & f(A)f(BC)\ar[rr]^{\rho_{f(A),f(BC)}}\ar[d]^{\rotatebox{-90}{$\scriptstyle \chi_{A,BC}$}}\ar@{}[drr]|{\Uparrow u_{A,BC}} & & f(BC)f(A)\ar[d]^{\rotatebox{-90}{$\scriptstyle \chi_{BC,A}$}}\ar@{}[r]|{\underset{\omega_{B,C,A}}{\Rightarrow}} & f(B)f(CA)\ar@/^1pc/[dl]^{\chi_{B,CA}}\\
 & f(ABC)\ar[rr]^{f(\rho_{A,BC})}\ar@/_1pc/[dr]_{f(\rho_{A,B}\otimes 1_C)} & & f(BCA) & \\
 & & f(BAC)\ar@/_1pc/[ur]_{f(1_B\otimes\rho_{A,C})}\ar@{}[u]|{\Uparrow f(r_{A|B,C})} & & \\
 & & \rotatebox{90}{$\scriptstyle =$} & & \\
 & f(A)f(B)f(C)\ar[rr]^{\rho_{f(A),f(B)f(C)}}\ar[dr]^(.6){\rotatebox[origin=c]{-22}{\makebox[0pt]{$\scriptstyle \rho_{f(A),f(B)}\otimes 1_{f(C)}$}}}\ar[dl]_{\chi_{A,B}\otimes 1_{f(C)}} & & f(B)f(C)f(A)\ar[dr]^{1_{f(B)}\otimes\chi_{C,A}} & \\
f(AB)f(C)\ar[dd]_{\chi_{AB,C}}\ar[dr]_(.45){\rotatebox[origin=c]{-21}{\makebox[0pt]{$\scriptstyle f(\rho_{A,B})\otimes 1_{f(C)}$}}}\ar@{}[rr]|{\Uparrow u_{A,B}\otimes 1_{f(C)}} & & f(B)f(A)f(C)\ar[dl]_(.55){\rotatebox[origin=c]{23}{\makebox[0pt]{$\scriptstyle \chi_{B,A}\otimes 1_{f(C)}$}}}\ar[dr]^(.55){\rotatebox[origin=c]{-19}{\makebox[0pt]{$\scriptstyle 1_{f(B)\otimes\chi_{A,C}}$}}}\ar[ur]^(.4){\rotatebox[origin=c]{24}{\makebox[0pt]{$\scriptstyle 1_{f(B)\otimes\rho_{f(A),f(C)}}$}}}\ar@{}[rr]|{\Uparrow 1_{f(B)\otimes u_{A,C}}}\ar@{}[u]|(.7){\Uparrow r_{f(A)|f(B),f(C)}} & & f(B)f(CA)\ar[dd]^{\chi_{B,CA}} \\
 & f(BA)f(C)\ar[dr]_{\chi_{BA,C}}\ar@{}[rr]|{\Uparrow\omega_{B,A,C}} & & f(B)f(AC)\ar[dl]^{\chi_{B,AC}}\ar[ur]^(.45){\rotatebox[origin=c]{23}{\makebox[0pt]{$\scriptstyle 1_{f(B)}\otimes f(\rho_{C,A})$}}} & \\
f(ABC)\ar[rr]_{f(\rho_{A,B}\otimes 1_C)}\ar@{}[ur]|{\Uparrow\chi^{-1}_{\rho_{A,B},1_C}} & & f(BAC)\ar[rr]_{f(1_B\otimes\rho_{A,C})} & & f(BCA)\ar@{}[ul]|{\Uparrow\chi^{-1}_{1_B,\rho_{A,C}}}\\
}}
\end{gathered}
\end{align}
\begin{align}\label{diag_def_braidmor_II}
\begin{gathered}
\adjustbox{scale=0.7}{%
\xymatrix{
 & f(A)f(B)f(C)\ar[rr]^{\rho_{f(A)f(B),f(C)}}\ar[d]^{\rotatebox{-90}{$\scriptstyle \chi_{A,B}$}}_{\rotatebox{-90}{$\scriptstyle \otimes 1_{f(C)}$}}\ar@/_1pc/[dl]_{1_{f(A)}\otimes\chi_{B,C}}\ar@{}[drr]|{\Uparrow\rho_{\chi_{A,B},1}} & & f(C)f(A)f(B)\ar[d]^{\rotatebox{-90}{$\scriptstyle 1_{f(C)}$}}_{\rotatebox{-90}{$\scriptstyle \otimes \chi_{A,B}$}}\ar@/^1pc/[dr]^{\chi_{C,A}\otimes 1_{f(B)}} & \\ 
f(A)f(BC)\ar@/_1pc/[dr]_{\chi_{A,BC}}\ar@{}[r]|{\underset{\omega^{-1}_{A,B,C}}{\Rightarrow}} & f(AB)f(C)\ar[rr]^{\rho_{f(AB),f(C)}}\ar[d]^{\rotatebox{-90}{$\scriptstyle \chi_{AB,C}$}}\ar@{}[drr]|{\Uparrow u_{AB,C}} & & f(C)f(AB)\ar[d]^{\rotatebox{-90}{$\scriptstyle \chi_{C,AB}$}}\ar@{}[r]|{\underset{\omega^{-1}_{C,A,B}}{\Rightarrow}} & f(CA)f(B)\ar@/^1pc/[dl]^{\chi_{CA,B}}\\
 & f(ABC)\ar[rr]^{f(\rho_{AB,C})}\ar@/_1pc/[dr]_{f(1_{A}\otimes\rho_{B,C})} & & f(CAB) & \\
 & & f(ACB)\ar@/_1pc/[ur]_{f(\rho_{A,C}\otimes 1_{B})}\ar@{}[u]|{\Uparrow f(r^{-1}_{A,B|C})} & & \\
 & & \rotatebox{90}{$\scriptstyle =$} & & \\
& f(A)f(B)f(C)\ar[rr]^{\rho_{f(A)f(B),f(C)}}\ar[dr]^(.6){\rotatebox[origin=c]{-22}{\makebox[0pt]{$\scriptstyle 1_{f(A)}\otimes\rho_{f(B),f(C)}$}}}\ar[dl]_{1_{f(A)}\otimes\chi_{B,C}} & & f(C)f(A)f(B)\ar[dr]^{\chi_{C,A}\otimes 1_{f(B)}} & \\
f(A)f(BC)\ar[dd]_{\chi_{A,BC}}\ar[dr]_(.45){\rotatebox[origin=c]{-21}{\makebox[0pt]{$\scriptstyle 1_{f(A)}\otimes f(\rho_{B,C})$}}}\ar@{}[rr]|{\Uparrow 1_{f(A)}\otimes u_{B,C}} & & f(A)f(C)f(B)\ar[dl]_(.55){\rotatebox[origin=c]{23}{\makebox[0pt]{$\scriptstyle 1_{f(A)}\otimes \chi_{C,B}$}}}\ar[dr]^(.55){\rotatebox[origin=c]{-19}{\makebox[0pt]{$\scriptstyle \chi_{A,C}\otimes 1_{f(B)}$}}}\ar[ur]^(.4){\rotatebox[origin=c]{24}{\makebox[0pt]{$\scriptstyle \rho_{f(A),f(C)}\otimes 1_{f(B)}$}}}\ar@{}[rr]|{\Uparrow u_{A,C}\otimes 1_{f(B)}}\ar@{}[u]|(.7){\Uparrow r^{-1}_{f(A),f(B)|f(C)}} & & f(CA)f(B)\ar[dd]^{\chi_{CA,B}} \\
 & f(A)f(CB)\ar[dr]_{\chi_{A,CB}}\ar@{}[rr]|{\Uparrow\omega^{-1}_{A,C,B}} & & f(AC)f(B)\ar[dl]^{\chi_{AC,B}}\ar[ur]^(.45){\rotatebox[origin=c]{23}{\makebox[0pt]{$\scriptstyle f(\rho_{A,C})\otimes 1_{f(B)}$}}} & \\
f(ABC)\ar[rr]_{f(1_A\otimes\rho_{B,C})}\ar@{}[ur]|{\Uparrow\chi^{-1}_{1_A,\rho_{B,C}}} & & f(ACB)\ar[rr]_{f(\rho_{A,C}\otimes 1_B)} & & f(CAB)\ar@{}[ul]|{\Uparrow\chi^{-1}_{\rho_{A,C},1_B}}\\
}}
\end{gathered}
\end{align}

\end{lemma}
\begin{proof}
The statement is exactly \cite[Definition 14]{daystreet} up to composition with a corresponding instance of the natural invertible 2-cell
$\chi^{-1}$ in both cases.
\end{proof}

The same kind of reorganisation of Day--Street's definitions will help with the same kinds of computations further below 
for syllepses. Here, the corresponding reorganisations are precisely given by the definition of syllepses as presented in 
\cite{crans_centers}. 

\begin{lemma}\label{lemma_braid_rev_morphism}
Let $(F,u)\colon(\mathcal{C},\beta)\rightarrow(\mathcal{D},\beta)$ be a braided monoidal 2-functor between braided monoidal Gray monoids. 
Then $(F	,u^{\mathrm{rev}})\colon(\mathcal{C},\beta^{\mathrm{rev}})\rightarrow(\mathcal{D},\beta^{\mathrm{rev}})$ defined by the same 
monoidal structure and $u^{\mathrm{rev}}_{A,B}:=u_{B,A}^{-1_v}$ is a braided monoidal 2-functor as well.
\end{lemma}
\begin{proof}
We are to show that $u^{\mathrm{rev}}$ satisfies the two equations  (\ref{diag_def_braidmor_I}) and (\ref{diag_def_braidmor_II}). 
Equivalently, we may show the two equations for the vertical inverses of the respective natural 2-cells (with respect to their top 
and bottom horizontal morphisms). The vertical inverse of the two diagrams in (\ref{diag_def_braidmor_I}) computed for $u^{\mathrm{rev}}$ are 
exactly the two diagrams in (\ref{diag_def_braidmor_II}) computed for $u$ with $(B,C,A)$ substituted for $(A,B,C)$, and vice versa. Thus, 
Equation~(\ref{diag_def_braidmor_I}) for $u^{\mathrm{rev}}$ follows from Equation~(\ref{diag_def_braidmor_II}) for $u$, and vice versa. 
\end{proof}

\begin{lemma}\label{lemma_syll_rev}
Let $(\mathcal{C},\beta,v)$ be a sylleptic Gray monoid. Let $v^{\mathrm{rev}}_{A,B}$ be the 2-cell
\[\xymatrix{
A\otimes B\ar[r]^{\rho_{A,B}} & B\otimes A\ar@{=}[rr]\ar@/_/[dr]_{\rho_{B,A}^{\mathrm{rev}}} &\ar@{}[d]|{\Uparrow v_{B,A}^{-1_h}} & B\otimes A\ar[r]^{\rho^{\mathrm{rev}}_{B,A}} & A\otimes B.\\
 & & A\otimes B\ar@/_/[ur]_{\rho_{A,B}^{\mathrm{rev}}} & &
}\]
Then $(\mathcal{C},\beta^{\mathrm{rev}},v^{\mathrm{rev}})$ is again a sylleptic Gray monoid. We refer to $v^{\mathrm{rev}}$ as the 
reverse syllepsis associated to $\zeta$ on $(\mathcal{C},\beta)$.
\end{lemma}
\begin{proof}
The computation is again easier to perform by way of the equivalent definition of syllepses in
\cite[Definition 4.1]{crans_centers}. Here, Equation~(4.1) in \cite[Definition 4.1]{crans_centers} for
$(\mathcal{C},\beta^{\mathrm{rev}},v^{\mathrm{rev}})$ follows from Equation~(4.2) in \cite[Definition 4.1]{crans_centers} for
$(\mathcal{C},\beta,v)$, and vice versa.
\end{proof}


\begin{remark}\label{rem_syll_rev_hinv}
Let $(\mathcal{C},\beta,v)$ be a sylleptic Gray monoid. The modification $v$ may equivalently be considered as a pseudo-natural 
equivalence $z_{A,B}\colon \rho_{A,B}\rightarrow\rho_{A,B}^{\mathrm{rev}}$ by defining $z_{A,B}:=v_{A,B}\ast\rho_{A,B}^{\mathrm{rev}}$,
see e.g.\ \cite[Section 4]{crans_centers}. Under this equivalence, the reverse syllepsis $v^{\mathrm{rev}}_{A,B}$ corresponds to the 
horizontal inverse $z^{-1_h}_{B,A}$.
\end{remark}

Lastly, for later use let us record that the inverse of a pseudo-natural equivalence that comes equipped with a monoidal structure can also 
be equipped with a canonical inverse monoidal structure.

\begin{lemma}\label{lemma_mnt_inv}
Let $\mathcal{C}$ and $\mathcal{D}$ be Gray monoids, and $F,G\colon\mathcal{C}\rightarrow\mathcal{D}$ be monoidal 2-functors. Let
$(\theta,\theta_2,\theta_0)\colon F\rightarrow G$ be a monoidal pseudo-natural equivalence. Then
$(\theta^{-1},\theta_2^{-1_h},\theta_0^{-1_h})\colon F\rightarrow G$ is also a monoidal pseudo-natural equivalence. It defines an inverse
to $(\theta,\theta_2,\theta_0)$ in the 2-category of monoidal 2-functors from $\mathcal{C}$ to $\mathcal{D}$ \cite[p.107]{daystreet}.
\end{lemma}
\begin{proof}
Direct computation.
\end{proof}

\subsection{The canonical homotopy theory of Gray monoids}\label{sec_sub_graymonhtythy}
We consider the category $2$-Cat of 2-categories and 2-functors to come equipped with its canonical model structure introduced by Lack 
\cite{lack2catms}. We recall from loc.\ cit.\ that the canonical inclusion $2\text{-Cat}\hookrightarrow\mathrm{BiCat}$ is a right Quillen 
equivalence with regards to the canonical model structure on $\mathrm{BiCat}$ \cite{lackbicatms}. It hence induces an equivalence
\begin{align}\label{equ_2bicatequiv}
\mathrm{Ho}_{\infty}(2\text{-Cat})\xrightarrow{\simeq}\mathrm{Ho}_{\infty}(\mathrm{BiCat})
\end{align}
of underlying $\infty$-categories. That means we can replace bicategories and bifunctors with 2-categories and 2-functors from a homotopy 
theoretical standpoint as well. More precisely, the objects of $\mathrm{Ho}_{\infty}(2\text{-Cat})$ are cofibrant 2-categories $\mathcal{C}$, 
and the arrows $f\colon\mathcal{C}\rightarrow\mathcal{D}$ are 2-functors between such. Given a 2-category $\mathcal{C}$, its path object
$[E,\mathcal{C}]_{\mathrm{Gr}}$ in $2\text{-Cat}$ is given by its Gray exponential with the walking adjoint equivalence $E\in 2\text{-Cat}$ 
\cite{lackbicatms}. It follows that the 2-cells
\[\xymatrix{
 & \mathcal{D}\ar[dr]^g\ar@{}[d]|{\phi} & \\
\mathcal{C}\ar[rr]_h\ar[ur]^f & & \mathcal{E}
}\]
in $\mathrm{Ho}_{\infty}(2\text{-Cat})$ are pseudo-natural equivalences $\phi\colon gf\rightarrow h$, and the 3-cells are accordingly typed 
invertible modifications. All cells in dimension $>3$ are uniquely determined by their boundary as $\mathrm{Ho}_{\infty}(2\text{-Cat})$ 
is a 3-category \cite[Section 2.3.4]{luriehtt}.

\begin{remark}\label{rem_2bicatequiv}
As observed in \cite{lack_2bitricat}, the tricategory of 2-categories, 2-functors, pseudo-natural transformations and modifications 
is not triequivalent to the tricategory of bicategories, pseudo-functors, pseudo-natural transformations and modifications. This is because 
there are pseudo-functors between 2-categories that are not equivalent to a 2-functor (of same domain and codomain). There are however two 
ways (a left and a right way) to resolve this disparity: One may either enlarge $2\text{-Cat}$ by considering all pseudo-functors between 
2-categories, or restrict $2\text{-Cat}$ to the full subtricategory of cofibrant 2-categories instead. The first resolution follows the 
recipe already alluded to in Remark~\ref{rem_ssvsfw}. The second resolution is owed to the fact that pseudo-functors and 2-functors 
coincide whenever their domain is cofibrant. This latter resolution is what implicitly gives rise to the equivalence (\ref{equ_2bicatequiv}) 
of 3-categories. The same ``right'' principle applies to the tricategory of monoidal bicategories and the tricategory of cofibrant Gray 
monoids, which is what implicitly makes many of our arguments below work conceptually.
\end{remark}

We further recall from \cite{lack2catms} that $\mathrm{Gray}=(2\text{-Cat},\otimes_{\mathrm{Gr}})$ is a monoidal model category. Let
$\mathrm{GrMon}$ denote the category of Gray monoids and (strict) Gray monoidal 2-functors. It arises as the category of algebras for 
the free monoid monad on $\mathrm{Gray}$ and hence comes equipped with a (monadic) adjunction as follows.
\begin{align}\label{equ_propgrmmodstr}
\xymatrix{
\mathrm{GrMon}\ar@<-.5ex>[r]_(.55)U & 2\text{-Cat}\ar@<-.5ex>[l]_(.45)F
}
\end{align}

\begin{proposition}
The adjunction (\ref{equ_propgrmmodstr}) allows to lift the Gray tensor product on $2\text{-Cat}$ to a monoidal structure on
$\mathrm{GrMon}$ so that the forgetful functor $U$ becomes symmetric monoidal.
\end{proposition}
\begin{proof}
This follows for instance directly from \cite[Theorem 3.4.(2).(a)]{rs_monbicat}, as $\mathrm{Gray}$ is a symmetric monoidal category.
\end{proof}

\begin{proposition}\label{propgrmmodstr}
The adjunction (\ref{equ_propgrmmodstr}) allows to transfer the canonical model structure on $2\text{-Cat}$ to a model structure on
$\mathrm{GrMon}$ such that a morphism $f$ in $\mathrm{GrMon}$ is a (trivial) fibration if and only if $U(f)$ is so in $2\text{-Cat}$.
\end{proposition}
\begin{proof}
To construct this model structure it suffices to show that the monoidal model 
category $\mathrm{Gray}$ satisfies the monoid axiom \cite[Definition 3.3, Theorem 4.1]{schwedeshipley_mon}.
This is shown in \cite[Theorem 7.7]{lack2catms}.
\end{proof}

\begin{remark}
The model structure on $\mathrm{GrMon}$ from Proposition~\ref{propgrmmodstr} is the single object version of Lack's model structure on
the category of Gray categories \cite{lackgraycatms}. 
\end{remark}

It will be useful to record the following additional homotopical property of the adjunction (\ref{equ_propgrmmodstr}).

\begin{lemma}\label{lemmacofobjects}
The forgetful functor
\[U\colon\mathrm{GrMon}\rightarrow 2\text{-Cat}\]
preserves cofibrant objects, as well as cofibrations between cofibrant objects.
\end{lemma}
\begin{proof}
This follows directly from \cite[Theorem 4.1.(3)]{schwedeshipley_mon}. 
\end{proof}

In particular, the underlying 2-category of the lifted Gray tensor product
$\mathcal{C}\otimes_{\mathrm{Gr}}\mathcal{D}$ (obtained from Proposition~\ref{propgrmmodstr}.1) of two cofibrant Gray monoids is a cofibrant 
2-category. In particular, the powers $\mathcal{C}^{\otimes_{\mathrm{Gr}}^n}$ of a cofibrant Gray monoid are cofibrant 2-categories.

To understand the $\infty$-category $\mathrm{Ho}_{\infty}(\mathrm{GrMon})$ that arises by formal means, we state the following lemma which 
characterises it precisely as the $\infty$-category of $\mathsf{E}_1$-monoids in $\mathrm{Ho}_{\infty}(2\text{-Cat})$. 

\begin{lemma}\label{lemma_grmon=E1}
The forgetful functor $U\colon\mathrm{GrMon}\rightarrow 2\text{-Cat}$ induces a right adjoint
\[\mathrm{Ho}_{\infty}(U)\colon\mathrm{Ho}_{\infty}(\mathrm{GrMon})\rightarrow\mathrm{Ho}_{\infty}(2\text{-Cat})\]
that lifts to an equivalence
\begin{align}\label{equ_grmon=E1}
\mathrm{Ho}_{\infty}(U)\colon\mathrm{Ho}_{\infty}(\mathrm{GrMon})\rightarrow\mathrm{Alg}_{\mathsf{E}_1}(\mathrm{Ho}_{\infty}(2\text{-Cat})^\times).
\end{align}
\end{lemma}
\begin{proof}
Being the underlying $\infty$-category of a model category, the $\infty$-category $\mathrm{Ho}_{\infty}(2\text{-Cat})$ has all small 
limits. The monoidal structure on $\mathrm{Ho}_{\infty}(2\text{-Cat})$ induced by the monoidal model category $\mathrm{Gray}$ is the 
cartesian product by way of the natural pointwise trivial fibration $\otimes_{\mathrm{Gr}}\rightarrow \times$ in $2\text{-Cat}$ 
\cite{bourkegurski}. Thus, the statement follows directly from the rectification theorem of \cite[Theorem 4.1.8.4]{lurieha} since all 
conditions are satisfied.
\end{proof}

\begin{corollary}
The composition $\mathrm{GrMon}\xrightarrow{U}2\text{-Cat}\xrightarrow{\iota}\mathrm{BiCat}$ induces an equivalence
\[\mathrm{Ho}_{\infty}(\mathrm{GrMon})^{\times}\xrightarrow{\simeq}\mathrm{Alg}_{\mathsf{E}_1}(\mathrm{Ho}_{\infty}(\mathrm{BiCat})^\times).\]
of cartesian monoidal $\infty$-categories.
\end{corollary}\qed

To study the 2-dimensional algebra of Section~\ref{sec_sub_graymonalg} $\infty$-categorically, it is crucial to note that fully weak 
braidings, syllepses, and symmetries on a Gray monoid are equivalence invariant notions. In particular, 
given a Gray monoid $\mathcal{C}$, we may without loss of generality consider cofibrant replacements
$\mathbb{L}\mathcal{C}\xdblrightarrow{\sim}\mathcal{C}$ in $\mathrm{GrMon}$ thereof.

\begin{lemma}\label{lemma_mono_cofrepl}
Let $\mathcal{C}$ be a Gray monoid and $p\colon\mathbb{L}\mathcal{C}\xdblrightarrow{\sim}\mathcal{C}$ be a cofibrant replacement 
in $\mathrm{GrMon}$.
\begin{enumerate}
\item Then every (semi-strict/fully weak) braiding $\beta$ on $\mathcal{C}$ lifts to an essentially unique (semi-strict/fully weak) braiding 
$\mathbb{L}\beta$ on $\mathbb{L}\mathcal{C}$ such that $p(\mathbb{L}\beta)=\beta$.
\item Every syllepsis $\zeta$ on a braiding $\beta$ on $\mathcal{C}$ lifts to a unique syllepsis $\mathbb{L}\zeta$ on
$\mathbb{L}\beta$ on $\mathbb{L}(\mathcal{C})$ such that $p(\mathbb{L}\zeta)=\zeta$.
\item A syllepsis $\zeta$ on a braiding $\beta$ on $\mathcal{C}$ is symmetric if and only if the syllepsis $\mathbb{L}\zeta$ on
$\mathbb{L}\beta$ on $\mathbb{L}(\mathcal{C})$ is symmetric.
\end{enumerate}
\end{lemma}
\begin{proof}
This follows directly from the fact that $p$ is a (strict) Gray functor, and that trivial fibrations in $\twocat$ are surjective as well as 
locally surjective and fully faithful.
\end{proof}

It will be useful in the next section to understand data inside the $\infty$-category
$\mathrm{Alg}_{\mathsf{E}_1}(\mathrm{Ho}_{\infty}(2\text{-Cat})^\times)$ directly in terms of its associated bicategorical monoidal structures on
2-categories. Therefore, from Section~\ref{sec_infty} we recall that an $\mathsf{E}_1$-algebra $A$ in a symmetric monoidal $3$-category is uniquely 
determined by its underlying quasi-unital $\mathsf{A}_5$-algebra. That is, an $\mathsf{E}_1$-algebra is essentially a diagram of the form
\begin{align}\label{diag_equ_grmon=E1_exple}
\xymatrix{
I\ar[r]_u & A\ar@/^1pc/[r]|{u\otimes 1}\ar@/_1pc/[r]|{1\otimes u} & A^{\otimes 2}\ar[l]|{m} & A^{\otimes 3}\ar@<.5ex>[l]^{m\otimes 1}\ar@<-.5ex>[l]_{1\otimes m} & A^{\otimes 4}\ar@<-.5ex>@/_1pc/[l]_{m\otimes 1\otimes 1}\ar[l]|(.42){1\otimes m\otimes 1}\ar@<.5ex>@/^1pc/[l]^{1\otimes 1\otimes m} & A^{\otimes 5}\ar@<.75ex>[l]\ar@<.25ex>[l]\ar@<-.25ex>[l]\ar@<-.75ex>[l]
}.
\end{align}
In particular, a Gray monoid $A$ is a monoid in the monoidal 1-category $\mathrm{Gray}$, and so is essentially of the form
\[\xymatrix{
I\ar[r]_u & A\ar@/^1pc/[r]|{u\otimes 1}\ar@/_1pc/[r]|{1\otimes u} & A\otimes_{\mathrm{Gr}}A\ar[l]|{m} & A\otimes_{\mathrm{Gr}}A\otimes_{\mathrm{Gr}}A\ar@<.5ex>[l]^(.6){m\otimes 1}\ar@<-.5ex>[l]_(.6){1\otimes m}
}\]
in $2\text{-Cat}$. This diagram can be trivially extended to a quasi-unital $\mathsf{A}_5$-structure in $\mathrm{Gray}$. By way of the 
equivalence $\otimes_{\mathrm{Gr}}\xrightarrow{\simeq}\times$ in $\mathrm{Ho}_{\infty}(2\text{-Cat})$, the equivalence 
(\ref{equ_grmon=E1}) maps a cofibrant Gray monoid $A$ to its associated diagram (\ref{diag_equ_grmon=E1_exple}) in $\mathrm{Ho}_{\infty}(2\text{-Cat})^{\times}$.
Given two cofibrant Gray monoids $A$ and $B$, a morphism $f\colon U(A)\rightarrow U(B)$ in
$\mathrm{Alg}_{\mathsf{E}_1}(\mathrm{Ho}_{\infty}(2\text{-Cat})^\times)$ is (up to equivalence) uniquely determined by a diagram of the form
\begin{align}\label{diag_grmoncof_mor}
\begin{gathered}
\xymatrix{
A\otimes_{\mathrm{Gr}}A\otimes_{\mathrm{Gr}}A\ar[r]^{(f,f,f)}\ar@<-.5ex>[d]_{1\otimes m}\ar@<.5ex>[d]^{m\otimes 1}\ar@{}[dr]|{\omega} & B\otimes_{\mathrm{Gr}}B\otimes_{\mathrm{Gr}}B\ar@<-.5ex>[d]_{1\otimes m}\ar@<.5ex>[d]^{m\otimes 1}\\
A\otimes_{\mathrm{Gr}}A\ar[r]^{(f,f)}\ar[d]|{\otimes}\ar@{}[dr]|{\chi} & B\otimes_{\mathrm{Gr}}B\ar[d]|{\otimes}\\
A\ar[r]|f\ar@{}[dr]|{\iota} & B \\
\ast\ar[u]^{u}\ar@{=}[r] & \ast\ar[u]^{u}
}
\end{gathered}
\end{align}
which extends to $(f,f,f,f)\colon A^{\otimes_{\mathrm{Gr}} 4}\rightarrow B^{\otimes_{\mathrm{Gr}} 4}$ in the obvious way. This extension is 
propositional and hence witnessed by a set of corresponding equations. Extension to the fifth power (and hence all others) is trivial by way 
of Lemma~\ref{lemma_truncext0}. 

\begin{remark}
Let $\mathcal{C}$ and $\mathcal{D}$ be cofibrant Gray monoids. By the above it is clear that a monoidal 2-functor
$F\colon\mathcal{C}\rightarrow\mathcal{D}$ is a notion that sits inbetween that of an $\mathsf{E}_1$-monoid morphism and that of a
quasi-unital $\mathsf{A}_4$-monoid morphism, and that is equivalent to both. One can make precise that they are the morphisms of
$\mathsf{A}_4$-monoids with distinguished quasi-units.
\end{remark}

The Gray powers $A^{\otimes_{\mathrm{Gr}}n}$ and $B^{\otimes_{\mathrm{Gr}}n}$ are cofibrant 2-categories whenever $A$ and $B$ are so in
$\mathrm{GrMon}$. It follows that the data of a morphism $F\colon U(A)\rightarrow U(B)$ in 
$\mathrm{Alg}_{\mathsf{E}_1}(\mathrm{Ho}_{\infty}(2\text{-Cat})^\times)$ gives rise to the data of a monoidal 2-functor. Vice versa, every 
monoidal 2-functor $F$ can be lifted to a Gray functor $\bar{F}\colon A\rightarrow B$ along the equivalence (\ref{equ_grmon=E1}) such that 
$U(\bar{F})\simeq F$. This is a purely homotopical incarnation of the same such representation result typical and well-known in low 
dimensional monad theory by way of corresponding algebraic cofibrant replacements \cite{garner_cofmndthy}.
In fact, by the above reasoning it follows that monoidal 2-functors between \emph{cofibrant} Gray monoids are
over-determined: As quasi-unitality implies unitality, the higher unitality laws of a monoidal 2-functor can be left implicit. More 
precisely, the natural equivalences $\kappa$ and $\zeta$ that witness left- and right-unitality of the 2-functor, as well as compatibility 
of the associator $\omega$ with $\kappa$ and $\zeta$ can be derived. In fact, the unitality laws can be chosen to be trivial as the following 
sequence of lemmas shows.

\begin{lemma}\label{lemma_cofidunit}
Suppose $\mathcal{C}$ is a Gray monoid such that the inclusion $\ast\rightarrow\mathcal{C}$ is a cofibration of 2-categories (e.g.\ $\mathcal{C}$ is a cofibrant Gray monoid). Then every monoidal 2-functor
\[(f,\chi,\omega,\iota,\zeta,\kappa)\colon\mathcal{C}\rightarrow\mathcal{D}\]
is equivalent to a monoidal 2-functor $(f\sprime,\chi\sprime,\omega\sprime,\iota\sprime,\zeta\sprime,\kappa\sprime)$ with
$\iota\sprime=1_I$.
\end{lemma}
\begin{proof}
The morphism $\iota\colon I\rightarrow f(I)$ is a square of the form
\[\xymatrix{
\ast\ar[r]^(.4){\iota}\ar[d]_{\{I\}} & [E,\mathcal{C}]_{\mathrm{Gr}}\ar@{->>}[d]^{\rotatebox{90}{$\scriptstyle \sim$}} \\
\mathcal{C}\ar[r]_{f} & \mathcal{D} 
}\]
in $2\text{-Cat}$. The left vertical map is a cofibration by Lemma~\ref{lemmacofobjects}. We hence obtain a lift
$l\colon\mathcal{C}\rightarrow[E,\mathcal{D}]_{\mathrm{Gr}}$. The 2-functor $f\sprime:= s l\colon\mathcal{C}\rightarrow\mathcal{D}$ maps the unit $I$ of $\mathcal{C}$ to the unit $I$ of $\mathcal{D}$. The monoidal structure on $f$ now can be transported along $l$ to a monoidal structure $(f\sprime,\chi\sprime,\omega\sprime,1,\zeta\sprime,\kappa\sprime)$ as claimed.
\end{proof}


\begin{lemma}\label{lemma_cofbraidunit}
Suppose $\mathcal{C}$ is a Gray monoid such that the inclusion $\ast\rightarrow\mathcal{C}$ is a cofibration of 2-categories. Then for every  
Gray monoid $\mathcal{D}$ and every monoidal 2-functor of the form
\[(f,\chi,\omega,1,\zeta,\kappa)\colon\mathcal{C}\rightarrow\mathcal{D}\]
there is a 2-cell $\chi\sprime$ and a 3-cell $\omega\sprime$ such that the identity on $f$ extends to an equivalence
$(f,\chi,\omega,\iota,\zeta,\kappa)\simeq(f,\chi\sprime,\omega\sprime,1,1,1)$ of monoidal 2-functors. In particular, $\chi\sprime_{A,I}=\chi\sprime_{I,A}=1_A$ for all $A\in\mathcal{C}$.
\end{lemma}
\begin{proof}
By virtue of the fact that $\mathrm{Gray}$ is monoidal, it follows that the 2-latching map 
of the nerve $N(\mathcal{C})$ --- i.e.\ the canonical inclusion
\[(\ast\otimes_{\mathrm{Gr}}\mathcal{C})\sqcup_{\ast}(\mathcal{C}\otimes_{\mathrm{Gr}}\ast)\hookrightarrow\mathcal{C}\otimes_{\mathrm{Gr}}\mathcal{C}\]
that maps $A$ to $(I,A)$ or to $(A,I)$, respectively --- is a cofibration in $2\text{-Cat}$. Given a monoidal 2-functor
$(f,\chi,\omega,\iota,\zeta,\kappa)\colon\mathcal{C}\rightarrow\mathcal{D}$, the pseudo-natural 
transformations $\chi$, $\zeta$, and $\kappa$ give rise to the diagram
\[\xymatrix{
(\ast\otimes_{\mathrm{Gr}}\mathcal{C})\sqcup_{\ast}(\mathcal{C}\otimes_{\mathrm{Gr}}\ast)\ar@{^(->}[d]\ar[r]^(.65){(\Delta f,\Delta f)}\ar@{}[dr]|(.35){\simeq(\zeta,\kappa)} & [E,\mathcal{D}]_{\mathrm{Gr}}\ar@{->>}[d]^{(s,t)} \\
\mathcal{C}\otimes_{\mathrm{Gr}}\mathcal{C}\ar[r]_{(m(f,f),f m)}\ar[ur]_{\ulcorner\chi\urcorner} & \mathcal{D}\times\mathcal{D}
}\]
in $2\text{-Cat}$. Here, both the outer square and the lower triangle commute strictly, while the top triangle commutes precisely up to the 
homotopy given by the pair $(\zeta,\kappa)$. This homotopy is in fact a homotopy in the slice $2\text{-Cat}_{/\mathcal{D}\times\mathcal{D}}$ whenever $\iota=1$, because the composition $(s,t)(\zeta,\kappa)$ is constant. This means, we are given a homotopy-commutative triangle
\[\xymatrix{
(\ast\otimes_{\mathrm{Gr}}\mathcal{C})\sqcup_{\ast}(\mathcal{C}\otimes_{\mathrm{Gr}}\ast)\ar[r]^(.65){(\Delta f,\Delta f)}\ar@{^(->}[d]_{\iota} & [E,\mathcal{D}]_{\mathrm{Gr}}\\
\mathcal{C}\otimes_{\mathrm{Gr}}\mathcal{C}\ar[ur]_{\ulcorner\chi\urcorner} & 
}\]
in the model category $2\text{-Cat}_{/\mathcal{D}\times\mathcal{D}}$, where the vertical map on the left hand side is a cofibration and the 
object on right hand side is fibrant. Thus, by the homotopy extension property, we find a homotopy $\Xi\colon \chi\sprime\simeq\chi$ in
$2\text{-Cat}_{/\mathcal{D}\times\mathcal{D}}$ such that $\Xi\iota=(\zeta,\kappa)$ and such that the triangle 
\[\xymatrix{
(\ast\otimes_{\mathrm{Gr}}\mathcal{C})\sqcup_{\ast}(\mathcal{C}\otimes_{\mathrm{Gr}}\ast)\ar[r]^(.65){(\Delta f,\Delta f)}\ar@{^(->}[d]_{\iota} & [E,\mathcal{D}]_{\mathrm{Gr}}\\
\mathcal{C}\otimes_{\mathrm{Gr}}\mathcal{C}\ar[ur]_{\ulcorner\chi\sprime\urcorner} & 
}\]
commutes strictly. One then defines a 3-cell $\omega\sprime$ by pasting the 3-cell $\omega$ with corresponding instances of $\Xi$ at its 
boundaries. A direct computation shows that $(f,\chi\sprime,\omega\sprime,1,1,1)$ is a monoidal 2-functor that is equivalent to 
the monoidal structure $(f,\chi,\omega,\iota,\zeta,\kappa)$ that we had started with.
\end{proof}

\begin{remark}
The same reduction for monoidal functors between monoidal 1-categories can be found already in various parts of \cite{joyalstreet}, albeit by 
not explicitly specified means.
\end{remark}

\begin{corollary}\label{cor_cofbraidunit}
Suppose $\mathcal{C}$ is a Gray monoid such that the inclusion $\ast\rightarrow\mathcal{C}$ is a cofibration of 2-categories. Then for every monoidal structure $(\chi,\omega,1,\zeta,\kappa)$ on the multiplication
\[m\colon\mathcal{C}\otimes_{\mathrm{Gr}}\mathcal{C}\rightarrow\mathcal{C}\]
there is a 2-cell $\chi\sprime$ and a 3-cell $\omega\sprime$ such that
$(m,\chi,\omega,\iota,\zeta,\kappa)\simeq(m,\chi\sprime,\omega\sprime,1,1,1)$ as monoidal 2-functors, and such that furthermore
the pseudo-natural equivalence $\chi\sprime_{-,I,I,-}$ is the identity.
\end{corollary}
\begin{proof}
By virtue of the fact that $\mathrm{Gray}$ is monoidal, the inclusion $\ast\rightarrow\mathcal{C}\otimes_{\mathrm{Gr}}\mathcal{C}$ of Gray 
monoids is again a cofibration of underlying 2-categories. We thus can apply Lemma~\ref{lemma_cofbraidunit} to equivalently replace
$(\chi,\omega,1,\zeta,\kappa)$ by a monoidal structure of the form $(\chi,\omega,1,1,1)$ on $m$. To prove the second claim, we may transport 
the monoidal structure $(\chi,\omega,1,1,1)$ on $m$ along the pseudo-natural equivalence
$\chi_{-,I,I,-}\colon m\rightarrow m$ as follows. 

Consider the pseudo-natural equivalence $\chi_{A,1,1,B}\colon AB\rightarrow AB$. We define the composite pseudo-natural equivalence
\[\chi\sprime_{A,B,C,D}\colon ABCD\xrightarrow{\chi_{A,1,1,B}\otimes\chi_{C,1,1,D}}ABCD\xrightarrow{\chi_{A,B,C,D}}ACBD\xrightarrow{\chi_{AC,1,1,BD}^{-1}}ACBD,\]
as well as the modification $\omega\sprime_{A,B,C,D,E,F}$ defined as the whiskering
\[(\chi_{A,I,I,B}\otimes\chi_{C,I,I,D}\otimes\chi_{E,I,I,F})\ast\omega_{A,B,C,D,E,F}\ast\chi_{ACE,1,1,BDF}.\]
Then it is straightforward to verify that $(\chi\sprime,\omega\sprime,1,1,1)$ is again a monoidal structure on $m$, and that 
$\chi_{-,1,1,-}\colon m\rightarrow m$ together with the identity square
\[\xymatrix{
ABCD\ar[r]^{\chi\sprime_{A,B,C,D}}\ar[d]_{\chi_{A,1,1,B}\otimes\chi_{C,1,1,D}}\ar@{}[dr]|{1} & ACBD\ar[d]^{\chi_{AC,1,1,BD}} \\
ABCD\ar[r]_{\chi_{A,B,C,D}} & ACBD
}\]
and the identity triangle
\[\xymatrix{
 & I\ar@{=}[dl]\ar@{=}[dr]\ar@{}[d]|{1} & \\
 I\ar@{=}[rr] & & I   
}\]
defines a monoidal equivalence $(m,\chi,\omega,1,1,1)\simeq(m,\chi\sprime,\omega\sprime,1,1,1)$ by construction. 
Furthermore, we have
\[\chi\sprime_{A,1,1,D}=\chi_{A,1,1,D}^{-1}\circ\chi_{A,1,1,D}\circ(\chi_{A,1,1,1}\otimes\chi_{1,1,1,D})=1.\]
Indeed, $\chi_{A,1,1,1}=\chi_{1,1,1,A}=1_A$ because $\zeta_A=\kappa_A=1_{1_{A}}$ again by construction.
\end{proof}

\begin{remark}\label{rem_conj}
The transport argument in the proof of Corollary~\ref{cor_cofbraidunit} that replaces a given monoidal structure on the multiplication
$m$ with pseudo-natural equivalence $\chi$ by one such that $\chi\sprime_{-,I,I,-}=1$ is essentially an algebraic version of the
$\infty$-categorical structure-lifting arguments following Remark~\ref{rem_algfib}. We will apply this sort of argument more often below.
\end{remark}

The same can be shown for monoidal pseudo-natural equivalences; we will not need this however.
We summarise this sections discussion about the basic monoidal structures on 2-categories and basic higher algebraic structures thereon as 
follows. 

\begin{lemma}\label{lemma_halgtrans}
Let $\mathcal{C}$ and $\mathcal{D}$ be cofibrant Gray monoids. Then the forgetful functor
\begin{align}\label{equ_halgtrans}
\mathrm{Alg}_{\mathsf{E}_1}(\mathrm{Ho}_{\infty}(2\text{-Cat})^{\times})(U(\mathcal{C}),U(\mathcal{D}))\twoheadrightarrow\mathrm{Alg}_{\mathsf{A}_4^{\mathrm{qu}}}(\mathrm{Ho}_{\infty}(2\text{-Cat})^{\times})(U(\mathcal{C}),U(\mathcal{D}))
\end{align}
is a trivial fibration. In particular, 
\begin{enumerate}
\item every monoidal 2-functor can essentially uniquely be extended to a morphism of associated
$\mathsf{E}_1$-monoids in $\hotwocat$; vice versa, every morphism of associated $\mathsf{E}_1$-monoids in $\hotwocat$ has an underlying monoidal 2-functor; 
\item every monoidal pseudo-natural equivalence between two monoidal 2-functors $F,G\colon\mathcal{C}\rightarrow\mathcal{D}$ can 
essentially uniquely be extended to a 2-cell of associated $\mathsf{E}_1$-monoid morphisms in $\mathrm{Ho}_{\infty}(2\text{-Cat})$; vice versa, every 2-cell of associated $\mathsf{E}_1$-monoid morphisms in $\mathrm{Ho}_{\infty}(2\text{-Cat})$ has an underlying monoidal
pseudo-natural equivalence;
\item the same goes for invertible monoidal modifications and 3-cells in $\mathrm{Alg}_{\mathsf{E}_1}(\hotwocat)$.
\end{enumerate}
\end{lemma}
\begin{proof}
The fact that the map (\ref{equ_halgtrans}) is a trivial fibration follows directly from \cite[Corollary 4.1.6.17]{lurieha} and
\cite[Theorem 5.4.3.5]{lurieha}. For Part 1 it thus suffices to construct corresponding cells of quasi-unital
$\mathsf{A}_4$-monoid structures.
The fact that such 1-cells are precisely captured by the notion of a monoidal 2-functor is discussed in
Diagram~(\ref{diag_grmoncof_mor}). Parts 2 and 3 follow accordingly.
\end{proof}

\section{Proof of the Main Theorem}\label{sec_translation}

In the last section we characterised the various notions of monoidal structure on Gray monoids and 2-functors between 
them in terms of corresponding cells in the $\infty$-category $\mathrm{Ho}_{\infty}(2\text{-Cat})$. In this section we prove 
Theorem~\ref{thm_main} by using this dictionary together with the results of Section~\ref{sec_infty}. 

\begin{notation}
In the following we denote the cartesian monoidal $\infty$-category $\mathrm{Ho}_{\infty}(2\text{-Cat})^{\times}$ simply by
$2\text{-Cat}^{\times}$. In the context of this paper this is unambiguous as the 1-category $2\text{-Cat}$ is only ever considered as a 
symmetric monoidal category when equipped with the Gray tensor product. Furthermore, we will notationally identify a cofibrant Gray monoid 
$A$ with its associated $\mathsf{E}_1$-monoid $U(A)$ in $2\text{-Cat}^{\times}$.
\end{notation}

\subsection{Braidings and $\mathsf{E}_2$-structures}\label{sec_sub_translation2}

In this section we prove the following characterisation of a braiding on a Gray monoid.

\begin{theorem}\label{thm_braid=E_1}
Let $\mathcal{C}$ be a Gray monoid and $\beta(\mathcal{C})$ be the set of equivalence classes of braidings on $\mathcal{C}$.
Then there is a bijection
\[\beta(\mathcal{C})\cong\pi_0\left(\mathrm{Alg}_{\mathsf{E}_2}(2\text{-Cat}^{\times})\times_{\mathrm{Alg}_{\mathsf{E}_1}(2\text{-Cat}^{\times})}\{\mathcal{C}\}\right).\]
This is to say, every braiding on $\mathcal{C}$ induces an essentially unique $\mathsf{E}_1$-monoid structure 
on $\mathcal{C}$ in $\mathrm{Alg}_{\mathsf{E}_1}(2\text{-Cat}^{\times})$, and vice versa. 
\end{theorem}

We further prove the following relativisation.

\begin{proposition}\label{prop_braid=E_1_functors}
Let $(\mathcal{C},\beta)$ and $(\mathcal{D},\beta)$ be braided Gray monoids, and let $F\colon\mathcal{C}\rightarrow\mathcal{D}$ be a monoidal 
2-functor of underlying Gray monoids. Let $\beta(F)$ be the set of braidings on $F$.
Then
\[\beta(F)\simeq\mathrm{Alg}_{\mathsf{E}_2}(2\text{-Cat}^{\times})(\mathcal{C},\mathcal{D})\times_{\mathrm{Alg}_{\mathsf{E}_1}(2\text{-Cat}^{\times})(\mathcal{C},\mathcal{D})}\{F\}.\]
This is to say, every braiding on $F$ induces an essentially unique $\mathsf{E}_1$-monoid structure 
on $F$ in $\mathrm{Alg}_{\mathsf{E}_1}(2\text{-Cat})^{\times}$, and vice versa. In particular, the latter is an $h$-set.
\end{proposition}

We break up the two proofs into multiple parts.

\begin{lemma}\label{lemma_braid=E_1_forth}
Let $\mathcal{C}$ be a Gray monoid. Then every braiding $\beta$ on $\mathcal{C}$ induces an $\mathsf{E}_1$-monoid structure
$\mathcal{C}^{\beta}$ on $\mathcal{C}$ in $\mathrm{Alg}_{\mathsf{E}_1}(2\text{-Cat})^{\times}$.
\end{lemma}

\begin{proof}
Without loss of generality we may assume that $\mathcal{C}$ is cofibrant by Lemma~\ref{lemma_mono_cofrepl}. 
Suppose we are given a braiding $\beta=(\rho,\bar{\omega})$ on $\mathcal{C}$. By Remark~\ref{rem_braidtomono}, this induces a 
monoidal structure on the 2-functor $m\colon\mathcal{C}\otimes_{\mathrm{Gr}}\mathcal{C}\rightarrow\mathcal{C}$ such that the identity
\[\xymatrix{
\mathcal{C}\otimes_{\mathrm{Gr}}\mathcal{C}\otimes_{\mathrm{Gr}}\mathcal{C}\ar[r]^(.6){1\otimes_{\mathrm{Gr}}m}\ar[d]_{m\otimes_{\mathrm{Gr}}1}\ar@{}[dr]|1 & \mathcal{C}\otimes_{\mathrm{Gr}}\mathcal{C}\ar[d]^m \\
\mathcal{C}\otimes_{\mathrm{Gr}}\mathcal{C}\ar[r]_m & \mathcal{C}
}
\]
becomes a monoidal 2-functor via $1_0=\iota=1$ and $1_2=(1\otimes\bar{\omega}^{\mathrm{rev}}\otimes 1)^{\rho}$.
By Lemma~\ref{lemma_halgtrans}, this induces a diagram of solid arrows in $2\text{-Cat}^{\times}$ and corresponding cells between them as 
follows.
\begin{align}\label{diag_braid=E_1_functors_1}
\begin{gathered}
\xymatrix{
\mathcal{C}^{\otimes_{\mathrm{Gr}} 5}\ar@<1ex>[d]\ar@<.5ex>[d]\ar[d]\ar@<-.5ex>[d]\ar@<-1ex>[d] & & & & \\
\mathcal{C}^{\otimes_{\mathrm{Gr}}4}\ar@<-.5ex>@/_1pc/[d]_{m\otimes 1\otimes 1}\ar[d]|{1\otimes m\otimes 1}\ar@<.5ex>@/^1pc/[d]^{1\otimes 1\otimes m} & (\mathcal{C}^{\otimes_{\mathrm{Gr}} 2})^{\otimes_{\mathrm{Gr}} 4}\ar[l]_{(m,m,m,m)}\ar@<.5ex>[d]\ar[d]\ar@<-.5ex>[d] & & & \\
\mathcal{C}^{\otimes_{\mathrm{Gr}}3}\ar@<-.5ex>[d]_{m\otimes 1}\ar@<.5ex>[d]^{1\otimes m}\ar@{}[dr]|{\omega_{\beta}} & (\mathcal{C}^{\otimes_{\mathrm{Gr}}2})^{\otimes_{\mathrm{Gr}} 3}\ar[l]_{(m, m, m)}\ar@<-.5ex>[d]_{m^{\otimes}\otimes 1}\ar@<.5ex>[d]^{1\otimes m^{\otimes}} & (\mathcal{C}^{\otimes_{\mathrm{Gr}} 3})^{\otimes_{\mathrm{Gr}} 3}\ar@<.5ex>[l]^{(m\otimes 1)^{\otimes 3}}\ar@<-.5ex>[l]_{(1\otimes m)^{\otimes 3}}\ar@<.5ex>[d]^{1\otimes  m^{\otimes}}\ar@<-.5ex>[d]_{m^{\otimes}\otimes 1} & & & \\
\mathcal{C}^{\otimes_{\mathrm{Gr}}2}\ar[d]|{m} \ar[d]\ar@{}[dr]|{\chi_{\beta}} & (\mathcal{C}^{\otimes_{\mathrm{Gr}}2})^{\otimes_{\mathrm{Gr}} 2}\ar[d]\ar[l]_{(m, m)}\ar[d]^{m^{\otimes}}\ar@{}[dr]|{1_2} & (\mathcal{C}^{\otimes_{\mathrm{Gr}}3})^{\otimes_{\mathrm{Gr}} 2}\ar@<.5ex>[l]^{(m\otimes 1)^{\otimes 2}}\ar@<-.5ex>[l]_{(1\otimes m)^{\otimes 2}}\ar[d]^{m^{\otimes}}  &  (\mathcal{C}^{\otimes_{\mathrm{Gr}} 4})^{\otimes_{\mathrm{Gr}} 2}\ar@<.5ex>@{-->}[l]\ar@{-->}[l]\ar@<-.5ex>@{-->}[l]\ar@{-->}[d]^{m^{\otimes}} & \\
\mathcal{C} & \mathcal{C}^{\otimes_{\mathrm{Gr}}2}\ar[l]|{m} & \mathcal{C}^{\otimes_{\mathrm{Gr}}3}\ar@<.5ex>[l]^{m\otimes 1}\ar@<-.5ex>[l]_{1\otimes m} & \mathcal{C}^{\otimes_{\mathrm{Gr}}4}\ar@<-.5ex>@/_1pc/[l]_{m\otimes 1\otimes 1}\ar[l]|{1\otimes m\otimes 1}\ar@<.5ex>@/^1pc/[l]^{1\otimes 1\otimes m} & \mathcal{C}^{\otimes_{\mathrm{Gr}} 5}\ar@<1ex>[l]\ar@<.5ex>[l]\ar[l]\ar@<-.5ex>[l]\ar@<-1ex>[l]
}
\end{gathered}
\end{align}
To extend this to a non-unital $\mathsf{E}_1$-structure on $\mathcal{C}$, by Corollary~\ref{cor_eh_3} it suffices to extend the solid part 
of the diagram by a 4-cell in $2\text{-Cat}^{\times}(\mathcal{C}^{\otimes_{\mathrm{Gr}}^8},\mathcal{C})$ which yields the dotted square on 
the bottom right. Therefore, we use that the constructions to this point apply to every braiding $\beta$ on $\mathcal{C}$; in particular, 
they apply to the reverse braiding $\beta^{\mathrm{rev}}$ associated to the fixed braiding $\beta$ in the beginning of the proof. We hence 
obtain a monoidal structure $(\chi_{\beta^{\mathrm{rev}}},\omega_{\beta^{\mathrm{rev}}})$ on the same 2-functor
$m\colon\mathcal{C}\otimes_{\mathrm{Gr}}\mathcal{C}\rightarrow\mathcal{C}$, which corresponds to an
$\mathsf{A}_2^{\mathrm{qu}}\otimes \mathsf{E}_1$-structure 
\[\xymatrix{
 \mathcal{C}^{\otimes_{\mathrm{Gr}} 5}\ar@<1ex>[d]\ar@<.5ex>[d]\ar[d]\ar@<-.5ex>[d]\ar@<-1ex>[d] & \\
 \mathcal{C}^{\otimes_{\mathrm{Gr}}4}\ar@<-.5ex>@/_1pc/[d]_{}\ar[d]|{1\otimes m\otimes 1}\ar@<.5ex>@/^1pc/[d]^{} & (\mathcal{C}^{\otimes_{\mathrm{Gr}} 2})^{\otimes_{\mathrm{Gr}} 4}\ar[l]_{(m,m,m,m)}\ar@<-.5ex>@/_1.5pc/[d]_{}\ar[d]|{1\otimes m^{\otimes}\otimes 1}\ar@<.5ex>@/^1.5pc/[d]^{}  \\
 \mathcal{C}^{\otimes_{\mathrm{Gr}}3}\ar@<-.5ex>[d]_{\rotatebox{-90}{$\scriptstyle m\otimes 1$}}\ar@<.5ex>[d]^{\rotatebox{-90}{$\scriptstyle 1\otimes m$}}\ar@{}[dr]|{\omega_{\beta^{\mathrm{rev}}}} & (\mathcal{C}^{\otimes_{\mathrm{Gr}}2})^{\otimes_{\mathrm{Gr}} 3}\ar[l]_{(m, m, m)}\ar@<-.5ex>[d]_{\rotatebox{90}{$\scriptstyle m^{\otimes}\otimes 1$}}\ar@<.5ex>[d]^{\rotatebox{90}{$\scriptstyle 1\otimes m^{\otimes}$}} \\
 \mathcal{C}^{\otimes_{\mathrm{Gr}}2}\ar[d]_{m} \ar[d]\ar@{}[dr]|{\chi_{\beta^{\mathrm{rev}}}} & (\mathcal{C}^{\otimes_{\mathrm{Gr}}2})^{\otimes_{\mathrm{Gr}} 2}\ar[d]\ar[l]_{(m, m)}\ar[d]^{m^{\otimes}}\\
\mathcal{C} & \mathcal{C}^{\otimes_{\mathrm{Gr}}2}\ar[l]_{m}
}\]
in $2\text{-Cat}^{\times}$. Its diagonal reflection is an
$\mathsf{E}_1\otimes \mathsf{A}_2^{\mathrm{qu}}$-algebra of the form
\begin{align}\label{diag_braid=E_1_functors_2}
\begin{gathered}
\xymatrix{
\mathcal{C}^{\otimes_{\mathrm{Gr}}2}\ar[d]_{m} \ar[d]\ar@{}[dr]|{(\chi_{\beta^{\mathrm{rev}}})^{\rho}} & (\mathcal{C}^{\otimes_{\mathrm{Gr}}2})^{\otimes_{\mathrm{Gr}} 2}\ar[d]\ar[l]_(.55){(m, m)}\ar[d]^{m^{\otimes}}\ar@{}[dr]|{(\omega_{\beta^{\mathrm{rev}}})^{\rho}} & (\mathcal{C}^{\otimes_{\mathrm{Gr}}3})^{\otimes_{\mathrm{Gr}} 2}\ar@<.5ex>[l]^{(m\otimes 1)^{\otimes 2}}\ar@<-.5ex>[l]_{(1\otimes m)^{\otimes 2}}\ar[d]^{m^{\otimes}}  &  (\mathcal{C}^{\otimes_{\mathrm{Gr}} 4})^{\otimes_{\mathrm{Gr}} 2}\ar@<-.5ex>@/_1pc/[l]_{}\ar[l]|{}\ar@<.5ex>@/^1pc/[l]^{}]\ar[d]^{m^{\otimes}} & \\
\mathcal{C} & \mathcal{C}^{\otimes_{\mathrm{Gr}}2}\ar[l]|{m} & \mathcal{C}^{\otimes_{\mathrm{Gr}}3}\ar@<.5ex>[l]^{m\otimes 1}\ar@<-.5ex>[l]_{1\otimes m} & \mathcal{C}^{\otimes_{\mathrm{Gr}}4}\ar@<-.5ex>@/_1pc/[l]_{}\ar[l]|{1\otimes m\otimes 1}\ar@<.5ex>@/^1pc/[l]^{} & \mathcal{C}^{\otimes_{\mathrm{Gr}} 5}\ar@<1ex>[l]\ar@<.5ex>[l]\ar[l]\ar@<-.5ex>[l]\ar@<-1ex>[l]
}
\end{gathered}
\end{align}
which is precisely the underlying corresponding portion of Diagram (\ref{diag_braid=E_1_functors_1}) by way of
Example~\ref{exple_diagrefl2} and the definition of the cell $1_2$. Diagrams (\ref{diag_braid=E_1_functors_1}) and 
(\ref{diag_braid=E_1_functors_2}) together hence yield the desired non-unital $\mathsf{E}_2$-monoid structure $\mathcal{C}^{\beta}$ by 
Corollary~\ref{cor_eh_3}. This structure lifts to an $\mathsf{E}_2$-monoid by Lemma~\ref{lemma_unitff}. Indeed, its underlying non-unital
$\mathsf{E}_1$-monoid --- i.e.\ the vertical $\mathsf{E}_1$-monoid on the very left of (\ref{diag_braid=E_1_functors_1}) --- is given by
$\mathcal{C}$, and hence is unital. Its horizontal multiplication in $\mathrm{Alg}_{\mathsf{E}_1}(\hotwocat)$ is given by the monoidal 2-
functor $(m,\chi_{\beta},\omega_{\beta},1,1,1)$, which also has a unit given by the diagonally reflected unitality laws of the reverse 
braiding.
\end{proof}

\begin{remark}
To elaborate on the function of the reverse braiding in the proof of Lemma~\ref{lemma_braid=E_1_forth} let us highlight that 
primarily two of its properties are needed here. First, we use that the reverse $\beta^{\mathrm{rev}}$ of a braiding $\beta$ is indeed a 
braiding, and second, that the underlying
$\mathsf{A}_2^{\mathrm{nu}}\otimes_{\mathrm{BV}}\mathsf{A}_2^{\mathrm{nu}}$-monoid structure of the diagonal reflection of
$\beta^{\mathrm{rev}}$ is precisely the underlying $\mathsf{A}_2^{\mathrm{nu}}\otimes_{\mathrm{BV}}\mathsf{A}_2^{\mathrm{nu}}$-monoid 
structure of $\beta$ itself. If instead for example we were naively to ask that the underlying
$\mathsf{A}_2^{\mathrm{nu}}\otimes_{\mathrm{BV}}\mathsf{A}_2^{\mathrm{nu}}$-monoid structure of the diagonal reflection of $\beta$ itself 
be equivalent to the underlying $\mathsf{A}_2^{\mathrm{nu}}\otimes_{\mathrm{BV}}\mathsf{A}_2^{\mathrm{nu}}$-monoid structure of $\beta$ 
again, we were to ask for an equivalence
\[\chi_{\beta}\xrightarrow{\simeq}\chi_{\beta}^{-1}(\sigma^{\ast}_{(\langle 2\rangle,\langle 2\rangle})(-)).\]
That is, equivalently, for a natural equivalence $u_{A,B}\colon\rho_{A,B}\rightarrow\rho^{-1}_{B,A}$.
This means we were to ask for the underlying modification of a syllepsis, which generally of course does not exist.
\end{remark}

To extract a braiding from a general $\mathsf{E}_2$-monoid in $\hotwocat$, we perform a series of semi-strictifications.

\begin{lemma}\label{lemma_braid=E_1_back}
Let $\mathcal{C}$ be a cofibrant Gray monoid. Let $\mathcal{C}^+$ be an $\mathsf{E}_1$-monoid structure on $\mathcal{C}$ in
$\mathrm{Alg}_{\mathsf{E}_1}(2\text{-Cat}^{\times})$ whose underlying diagram (\ref{diag_eh_3}) from Corollary~\ref{cor_eh_3} is such that 
the following equations hold:
\begin{enumerate}
\item $U(\mathcal{C}^+)=\mathrm{Alg}_{\mathsf{E}_1}(U)(\mathcal{C}^+)=\mathcal{C}$,
\item $\chi_{I,I,-,-}=\chi_{-,-,I,I}=\chi_{-,I,I,-}=1$,
\item $\omega_{-,-,I,I,-,-}=\omega_{-,I,I,-,-,-}=\omega_{-,-,-,I,I,-}=1$,
\item $\psi_{-,I,-,-,I,-}=1$.
\end{enumerate}
Then the tuple 
\begin{itemize}
\item $\rho_{A,B}:=\chi_{I,A,B,I}$
\item $\bar{\omega}_{A,B,C,D}:=\omega_{I,A,B,C,D,I}$
\end{itemize}
defines a braiding $\beta$ on $\mathcal{C}$ such that the identity on $\mathcal{C}$ extends to an equivalence
$\mathcal{C}^{\beta}\simeq\mathcal{C}^+$ in $\mathrm{Alg}_{\mathsf{E}_1}(\hotwocat)$. Any two equivalent such $\mathsf{E}_1$-monoid 
structures induce equivalent braidings.
\end{lemma}	
\begin{proof}
To prove that $\beta:=(\rho,\bar{\omega})$ is a braiding on $\mathcal{C}$ we are to show associativity of $\bar{\omega}$ as well as the 
equation relating the two Young--Baxter cells $r_{-,-|-}=\bar{\omega}_{-,I,-,-}$ and $r_{-|-,-}==\bar{\omega}_{-,-,I,-}$.
Therefore, by way of Corollary~\ref{cor_eh_3} it suffices to show that the induced diagram
\begin{align}\label{diag_lemma_braid=E_1_back}
\begin{gathered}
\xymatrix{
\mathcal{C}^{\otimes_{\mathrm{Gr}}4}\ar@<-.5ex>@/_1pc/[d]_{m\otimes 1\otimes 1}\ar[d]|{1\otimes m\otimes 1}\ar@<.5ex>@/^1pc/[d]^{1\otimes 1\otimes m} & & & \\
\mathcal{C}^{\otimes_{\mathrm{Gr}}3}\ar@<-.5ex>[d]_{m\otimes 1}\ar@<.5ex>[d]^{1\otimes m}\ar@{}[dr]|{\omega_{\beta}} & (\mathcal{C}^{\otimes_{\mathrm{Gr}}2})^{\otimes_{\mathrm{Gr}} 3}\ar[l]_{(m, m, m)}\ar@<-.5ex>[d]_{m^{\otimes}\otimes 1}\ar@<.5ex>[d]^{1\otimes m^{\otimes}} & & & \\
\mathcal{C}^{\otimes_{\mathrm{Gr}}2}\ar[d]|{m} \ar[d]\ar@{}[dr]|{\chi_{\beta}} & (\mathcal{C}^{\otimes_{\mathrm{Gr}}2})^{\otimes_{\mathrm{Gr}} 2}\ar[d]\ar[l]_{(m, m)}\ar[d]^{m^{\otimes}}\ar@{}[dr]|{(\omega_{\beta^{\mathrm{rev}}})^{\rho}} & (\mathcal{C}^{\otimes_{\mathrm{Gr}}3})^{\otimes_{\mathrm{Gr}} 2}\ar@<.5ex>[l]^{(m\otimes 1)^{\otimes 2}}\ar@<-.5ex>[l]_{(1\otimes m)^{\otimes 2}}\ar[d]^{m^{\otimes}} & \\
\mathcal{C} & \mathcal{C}^{\otimes_{\mathrm{Gr}}2}\ar[l]|{m} & \mathcal{C}^{\otimes_{\mathrm{Gr}}3}\ar@<.5ex>[l]^{m\otimes 1}\ar@<-.5ex>[l]_{1\otimes m} & \mathcal{C}^{\otimes_{\mathrm{Gr}}4}\ar@<-.5ex>@/_1pc/[l]_{m\otimes 1\otimes 1}\ar[l]|{1\otimes m\otimes 1}\ar@<.5ex>@/^1pc/[l]^{1\otimes 1\otimes m} 
}
\end{gathered}
\end{align}
is precisely the according portion of the diagram
\begin{align}\label{diag_braid=E_1_back_1}
\begin{gathered}
\xymatrix{
\mathcal{C}^{\otimes 5}\ar@<1ex>[d]\ar@<.5ex>[d]\ar[d]\ar@<-.5ex>[d]\ar@<-1ex>[d] & & & & \\
\mathcal{C}^{\otimes 4} \ar@<-.5ex>@/_1pc/[d]_{\rotatebox{90}{$\scriptstyle m\otimes 1\otimes 1$}}\ar[d]|{1\otimes m\otimes 1}\ar@<.5ex>@/^1pc/[d]^{\rotatebox{-90}{$\scriptstyle 1\otimes 1\otimes m$}} & (\mathcal{C}^{\otimes 2})^{\otimes 4}\ar[l]\ar@<.5ex>[d]\ar[d]\ar@<-.5ex>[d] & & & \\
\mathcal{C}^{\otimes 3} \ar@<-.5ex>[d]_{\rotatebox{90}{$\scriptstyle m\otimes 1$}}\ar@<.5ex>[d]^{\rotatebox{-90}{$\scriptstyle 1\otimes m$}}\ar@{}[dr]|{\omega} & (\mathcal{C}^{\otimes 2})^{\otimes 3}\ar[l]_{m^3}\ar@<-.5ex>[d]_{\rotatebox{90}{$\scriptstyle m^{\otimes}\otimes 1$}}\ar@<.5ex>[d]^{\rotatebox{-90}{$\scriptstyle 1\otimes m^{\otimes}$}} & (\mathcal{C}^{\otimes 3})^{\otimes 3}\ar@<.5ex>[l]\ar@<-.5ex>[l]\ar@<.5ex>[d]\ar@<-.5ex>[d] & & \\
\mathcal{C}^{\otimes 2}\ar[d]|{m}\ar[d]\ar@{}[dr]|{\chi} & (\mathcal{C}^{\otimes 2})^{\otimes 2}\ar[d]\ar[l]_{m^2}\ar[d]^{m^{\otimes}}\ar@{}[dr]|{\psi} & (\mathcal{C}^{\otimes 3} )^{\otimes 2}\ar@<.5ex>[l]^{(m\otimes 1)^{\otimes 2}}\ar@<-.5ex>[l]_{(1\otimes m)^{\otimes 2}}\ar[d]^{m^{\otimes}} & (\mathcal{C}^{\otimes 4})^{\otimes 2}\ar@<.5ex>[l]\ar[l]\ar@<-.5ex>[l]\ar[d] & \\
\mathcal{C} & \mathcal{C}^{\otimes 2}\ar[l]|{m} & \mathcal{C}^{\otimes 3} \ar@<.5ex>[l]^{m\otimes 1}\ar@<-.5ex>[l]_{1\otimes m} & \mathcal{C}^{\otimes 4} \ar@<-.5ex>@/_1pc/[l]_{m\otimes 1\otimes 1}\ar[l]|{1\otimes m\otimes 1}\ar@<.5ex>@/^1pc/[l]^{1\otimes 1\otimes m} & \mathcal{C}^{\otimes 5}\ar@<1ex>[l]\ar@<.5ex>[l]\ar[l]\ar@<-.5ex>[l]\ar@<-1ex>[l]
}
\end{gathered}
\end{align}
associated to the $\mathsf{E}_1$-monoid $\mathcal{C}^+$ on $\mathcal{C}$ that we start with. Given that, 
it follows that Diagram~(\ref{diag_lemma_braid=E_1_back}) can be extended to a full $\mathsf{E}_1$-monoid structure on $\mathcal{C}$ by 
assumption. In particular, the multiplication
$(m,\chi_{\beta},\omega_{\beta})\colon\mathcal{C}\otimes_{\mathrm{Gr}}\mathcal{C}\rightarrow\mathcal{C}$ is monoidal, and the
pseudo-natural equivalence
$(\omega_{\beta^{\mathrm{rev}}})^{\rho}\colon m(m\otimes_{\mathrm{Gr}} 1)\rightarrow m(1\otimes_{\mathrm{Gr}} m)$ is monoidal as well by 
Lemma~\ref{lemma_halgtrans}. By \cite[p.117-118]{daystreet} it follows that $\beta$ is a braiding on $\mathcal{C}$.
Furthermore, the identity on $\mathcal{C}$ automatically extends to an equivalence $\mathcal{C}^{\beta}\simeq\mathcal{C}^+$ by way of 
Corollary~\ref{cor_eh_3}. The fact that this construction maps equivalences of $\mathsf{E}_1$-monoid structures to equivalences of 
braidings is straightforward.

To show equality of (\ref{diag_lemma_braid=E_1_back}) and (\ref{diag_braid=E_1_back_1}), we are to show that
\begin{enumerate}[label=\roman*.]
\item $\chi_{\beta A,B,C,D}:=1_A\otimes\rho_{B,C}\otimes 1_D:=1_A\otimes\chi_{I,B,C,I}\otimes 1_D=\chi_{A,B,C,D}$,
\item $\omega_{\beta A,B,C,D,E,F}:=1_A\otimes\bar{\omega}_{B,C,D,E}\otimes 1_F:=1_A\otimes\omega_{I,B,C,D,E,I}\otimes 1_F=\omega_{A,B,C,D,E,F}$, and
\item $\psi_{\beta A,B,C,D,E,F}:=(\omega_{\beta^{\mathrm{rev}}})^{\rho}_{A,B,C,D,E,F}:=(1\otimes\bar{\omega}^{\mathrm{rev}}\otimes 1)^{\rho}_{A,B,C,D,E,F}=\omega_{A,B,C,D,E,F}$,
\end{enumerate}
where in the last case we recall Notation~\ref{notation_revbraid}. Let us first show that Equations (1)-(4) in the assumption of the lemma 
imply i.\ and ii.
Indeed, $\omega_{-,I,I,-,-,-}=1$ and $\chi_{-,I,I,-}=1$ together directly imply that $\chi_{A,B,C,D}=1_A\otimes\chi_{I,B,C,D}$. 
Analogously, by $\omega_{-,-,-,I,I,-}=1$ we get $\chi_{A,B,C,D}=1_A\otimes\chi_{I,B,C,I}\otimes 1_D$. Similarly, one shows that
\begin{align}\label{equ_braid=E_1_back_1}
\omega=1\otimes\omega_{I,-,-,-,-,-}
\end{align}
holds, too. More precisely, the cell $\omega$ is part of an $\mathsf{E}_1\otimes\mathsf{E}_1$-monoid structure, and so the unit $I$ is also 
a horizontal unit in $\mathrm{Alg}_{\mathsf{E}_1}(\hotwocat)$. That means
\[\omega_{-,I,-,I,-,I}=\omega_{I,-,I,-,I,-}=1\]
hold. From these we compute
\begin{align*}
\omega_{A,I,C,D,E,F} & = \omega_{A,I,C,I,E,DF}\\
& = (\omega_{A,I,C,I,E,I}\otimes 1_{DF})\ast\omega_{A,I,CE,I,I,DF}\\
& = 1\ast 1
\end{align*}
for all $A,C,D,E,F$ in $\mathcal{C}$. Here, the first equation is associativity of $\omega$ for the tuple $(A,I,C,I,I,D,E,F)$, and 
the second equation is associativity again for $(A,I,C,I,E,I,I,DF)$. One more application of associativity of $\omega$, setting $B=C=1$, 
yields (\ref{equ_braid=E_1_back_1}).
Analogously one shows that
\[\omega_{A,B,C,D,1,F}=1\]
holds for all $A,B,C,D,F$ in $\mathcal{C}$ as well. Subsequently, one computes that
\begin{align}\label{equ_braid=E_1_back_11}
\notag \omega_{A,B,C,D,E,F} &= 1_A\otimes \omega_{I,B,C,D,E,F}\\
& =1_A\otimes\omega_{A,B,C,D,E,I}\otimes 1_F
\end{align}
as before. We have thus shown that $\chi=\chi_{\beta}$ and $\omega=\omega_{\beta}$. We are left to show iii.

Therefore, we consider the 4-cell in Diagram~(\ref{diag_braid=E_1_back_1}) that starts at
$(\mathcal{C}^{\otimes_{\mathrm{Gr}}3})^{\otimes_{\mathrm{Gr}}3}$. It is an equation of the following two natural pasted 2-cells.
\begin{align}\label{diag_braid=E_1_back_EH2}
\begin{gathered}
\adjustbox{scale=0.7}{%
\xymatrix{
ABCDEFXYZ\ar@{=}[rr]\ar@{=}[dd]\ar@{}[ddrr]|{=} & & ABCDEFXYZ\ar[rr]^{\chi_{AB,C,DE,F}}_{\otimes 1_{XYZ}}\ar[dd]^{\rotatebox{-90}{$\scriptstyle \chi_{A,BC,D,EF}$}}_{\rotatebox{-90}{$\scriptstyle \otimes 1_{XYZ}$}}\ar@{}[ddrr]|{\Downarrow \psi_{A,B,C,D,E,F}\otimes 1_{XYZ}} & & ABDECFXYZ\ar[rr]^{\chi_{A,B,D,E}}_{\otimes 1_{CFXYZ}}\ar[dd]^{\rotatebox{-90}{$\scriptstyle \chi_{A,B,D,E}$}}_{\rotatebox{-90}{$\scriptstyle \otimes 1_{CFXYZ}$}}\ar@{}[ddrr]|{=} & & ADBECFXYZ\ar[dd]^{\rotatebox{-90}{$\scriptstyle \chi_{ADBE,CF,XY,Z}$}} \\
 && && && \\
ABCDEFXYZ\ar[rr]^{\chi_{A,BC,D,EF}}_{\otimes 1_{XYZ}}\ar[dd]^{\rotatebox{-90}{$\scriptstyle 1_{ABC}$}}_{\rotatebox{-90}{$\scriptstyle \otimes\chi_{D,EF,X,YZ}$}}\ar@{}[ddrr]|{\Downarrow \omega_{A,BC,D,EF,X,YZ}} & & ADBCEFXYZ\ar[rr]^{1_{AD}\otimes\chi_{B,C,E,F}}_{\otimes 1_{XYZ}}\ar[dd]^{\rotatebox{-90}{$\scriptstyle \chi_{AD,BCEF,X,YZ}$}}\ar@{}[ddrr]|{\Downarrow\chi_{(\chi^{(1,m)}_{A,B,C,D,E,F},1_{(X,Y,Z)})}} & & ABDECFXYZ\ar[rr]^{\chi_{ADBE,CF,XY,Z}}\ar[dd]^{\rotatebox{-90}{$\scriptstyle \chi_{AD,BECF,X,YZ}$}}\ar@{}[ddrr]|{\Downarrow\psi_{AD,BE,CF,X,Y,Z}} & & ADBEXYCFZ\ar[dd]^{\rotatebox{-90}{$\scriptstyle \chi_{AD,BE,X,Y}$}}_{\rotatebox{-90}{$\scriptstyle \otimes 1_{CFZ}$}} \\
 && && && \\
ABCDXEFYZ\ar[rr]^{\chi_{A,BC,DX,E,FYZ}}\ar[dd]^{\rotatebox{-90}{$\scriptstyle 1_{ABCDX}$}}_{\rotatebox{-90}{$\scriptstyle \otimes\chi_{E,F,Y,Z}$}}\ar@{}[ddrr]|{\Downarrow\chi^{-1}_{(1_{(A,B,C)},\chi^{(1,m)}_{D,E,F,X,Y,Z})}} & & ADXBCEFYZ\ar[rr]^{1_{ADX}\otimes\chi_{B,C,E,F}}_{\otimes 1_{YZ}}\ar[dd]^{\rotatebox{-90}{$\scriptstyle 1_{ADXBC}$}}_{\rotatebox{-90}{$\scriptstyle \otimes\chi_{E,F,Y,Z}$}}\ar@{}[ddrr]|{\Downarrow 1_{ADX}\otimes\omega_{B,C,E,F,Y,Z}} & & ADXBECFYZ\ar[rr]^{1_{ADX}}_{\otimes\chi_{BE,CF,Y,Z}}\ar[dd]^{\rotatebox{-90}{$\scriptstyle 1_{ADX}$}}_{\rotatebox{-90}{$\scriptstyle \otimes\chi_{BE,CF,Y,Z}$}}\ar@{}[ddrr]|{=} & & ADXBEYCFZ\ar@{=}[dd] \\
 && && && \\
ABCDXEYFZ\ar[rr]_{\chi_{A,BC,DX,EYFZ}} & & ADXBCEYFZ\ar[rr]^{1_{ADX}}_{\otimes\chi_{B,C,EY,FZ}} & & ADXBEYCFZ\ar@{=}[rr] & & ADXBEYCFZ \\
 && && && \\
 & & \ar@{}[rr]|{\rotatebox{90}{$=$}} & & & & \\
 && && && \\
ABCDEFXYZ\ar@{=}[rr]\ar@{=}[dd]\ar@{}[ddrr]|{=} & & ABCDEFXYZ\ar[rr]^{\chi_{AB,C,DE,F}}_{\otimes 1_{XYZ}}\ar[dd]^{\rotatebox{-90}{$\scriptstyle 1_{ABC}$}}_{\rotatebox{-90}{$\scriptstyle \chi_{DE,F,XY,Z}$}}\ar@{}[ddrr]|{\Downarrow \omega_{AB,C,DE,F,XY,Z}}  & & ABDECFXYZ\ar[rr]^{\chi_{A,B,D,E}}_{\otimes 1_{CFXYZ}}\ar[dd]^{\rotatebox{-90}{$\scriptstyle \chi_{ABDE,CF,XY,Z}$}}\ar@{}[ddrr]|{\Downarrow\chi_{(\chi^{(m,1)}_{A,B,C,D,E,F},1_{(X,Y,Z)})}}  & & ADBECFXYZ\ar[dd]^{\rotatebox{-90}{$\scriptstyle \chi_{ADBE,CF,XY,DZ}$}} \\
 && && && \\
ABCDEFXYZ\ar[rr]^{1_{ABC}}_{\otimes \chi_{DE,F,XY,Z}}\ar[dd]^{\rotatebox{-90}{$\scriptstyle 1_{ABC}$}}_{\rotatebox{-90}{$\scriptstyle \otimes\chi_{D,EF,X,YZ}$}}\ar@{}[ddrr]|{\Downarrow 1_{ABC}\otimes\psi_{D,E,F,X,Y,Z}} & & ABCDEXYFZ\ar[rr]^{\chi_{AB,C,DEXY,FZ}}\ar[dd]^{\rotatebox{-90}{$\scriptstyle 1_{ABC}\otimes\chi_{D,E,X,Y}$}}_{\rotatebox{-90}{$\scriptstyle\otimes 1_{FZ}$}}\ar@{}[ddrr]|{\Downarrow\chi^{-1}_{(1_{(A,B,C)},\chi^{(m,1)}_{D,E,F,X,Y,Z})}} & & ABDEXYCFZ\ar[rr]^{\chi_{A,B,D,E}}_{\otimes 1_{XYCFZ}}\ar[dd]^{\rotatebox{-90}{$\scriptstyle 1_{AB}\otimes\chi_{D,E,X,Y}$}}_{\rotatebox{-90}{$\scriptstyle \otimes 1_{CFZ}$}}\ar@{}[ddrr]|{\Downarrow\omega_{A,B,D,E,X,Y}\otimes 1_{CFZ}} & & ADBEXYCFZ\ar[dd]^{\rotatebox{-90}{$\scriptstyle \chi_{AD,BE,X,Y}$}}_{\rotatebox{-90}{$\scriptstyle \otimes 1_{CFZ}$}} \\
 && && && \\
ABCDXEFYZ\ar[rr]^{1_{ABCDX}}_{\otimes\chi_{E,F,Y,Z}}\ar[dd]^{\rotatebox{-90}{$\scriptstyle 1_{ABCDX}$}}_{\rotatebox{-90}{$\scriptstyle \otimes\chi_{E,F,Y,Z}$}}\ar@{}[ddrr]|{=} & & ABCDXEYFZ\ar[rr]^{\chi_{AB,C,DXEY,FZ}}\ar[dd]^{\rotatebox{-90}{$\scriptstyle \chi_{A,BC,DX,EYFZ}$}}\ar@{}[ddrr]|{\Downarrow\psi_{A,B,C,DX,EY,FZ}} & & ABDXEYCFZ\ar[rr]^{\chi_{A,B,DX,EY}}_{\otimes 1_{CFZ}}\ar[dd]^{\rotatebox{-90}{$\scriptstyle \chi_{A,B,DX,EY}$}}_{\rotatebox{-90}{$\scriptstyle \otimes 1_{CFZ}$}}\ar@{}[ddrr]|{=} & & ADXBEYCFZ\ar@{=}[dd] \\
 && && && \\
ABCDXEYFZ\ar[rr]_{\chi_{A,BC,DX,EYFZ}} & & ADXBCEYFZ\ar[rr]^{1_{ADX}}_{\otimes\chi_{B,C,EY,FZ}} & & ADXBEYCFZ\ar@{=}[rr] & & ADXBEYCFZ\\
}}
\end{gathered}
\end{align}
From this, let us show that $\psi=1\otimes\psi_{I,-,-,-,-,I}\otimes 1$ holds as well. We make use of the following claim.
\begin{claim}\label{claim_braid=E_1_back}
Under the given assumptions, the equation
\begin{align}\label{equ_braid=E_1_back_psiextra}
\psi_{-,I,I,I,-,-}=1
\end{align}
holds, too. 
\end{claim}
We defer a proof of the claim to the end and assume that it holds for now. Then, to recycle some of the computations we already did, we 
may work with the diagonal reflection of $\mathcal{C}^+$ instead, and hence assume that $\psi=1\otimes\psi_{I,-,-,-,-,I}\otimes 1$ 
in order to show that $\omega=1\otimes\omega_{I,-,-,-,-,I}\otimes 1$ instead. The two equations $\psi_{-,I,-,-,I,-}=\psi_{-,I,I,I,-,-}=1$ 
then translate to the assumption that
\begin{align}\label{claim_braid=E_1_back_reflected} 
&\omega(-,-,I,I,-,-)=1\text{, and} \\
\notag &\omega(-,I,I,-,I,-)=1.
\end{align}
From these two equations and Diagram~(\ref{diag_braid=E_1_back_EH2}) --- for the $\mathsf{E}_2$-monoid $(\mathcal{C}^+)^{\rho}$ --- it 
follows that
\begin{itemize}
\item $\omega_{A,I,I,F,Y,Z}=\omega_{A,I,I,F,I,YZ}$ and $\omega_{A,B,D,I,I,Z}=\omega_{AB,I,D,I,I,Z}$ by (\ref{diag_braid=E_1_back_EH2}) with $B,C,D,E,X=1$, and with $C,E,F,X,Y=1$, respectively;
\item $\omega_{A,I,E,I,I,Z}=\omega_{A,I,I,E,I,Z}$ by (\ref{diag_braid=E_1_back_EH2}) with $B,C,D,F,X,Y=1$.
\end{itemize}
In particular, by assumption of (\ref{claim_braid=E_1_back_reflected}), all these cells are degenerate. It follows, by way of the 
computations above concerning (\ref{equ_braid=E_1_back_1}) and (\ref{equ_braid=E_1_back_11}), that
\[\omega=1\otimes\omega_{I,-,-,-,-,I}\otimes 1.\]
That means, in turn, that $\psi=1\otimes\psi_{I,-,-,-,-,I}\otimes 1$ and $\omega=1\otimes\omega_{I,-,-,-,-,I}\otimes 1$ in $\mathcal{C}^+$ 
itself as well. Now, for $\mathcal{C}^+$, it consequently follows that
\begin{itemize}
\item $\psi_{A,I,C,D,Y,Z}=\omega_{A,C,D,I,Y,Z}^{-1_v}$ by (\ref{diag_braid=E_1_back_EH2}) with $B,E,F,X=1$, and
\item $\psi_{A,B,F,X,I,Z}=\omega_{A,B,I,F,X,Z}^{-1_v}$ by (\ref{diag_braid=E_1_back_EH2}) with $C,D,E,Y=1$.
\end{itemize}
Furthermore, associativity (with $D=E=1$) of the cell $\omega$ yields that
\[\omega_{A,B,C,F,G,H}=(1_A\otimes\omega_{I,B,C,I,G,I}\otimes 1_{FH})\ast(1_{AC}\otimes\omega_{I,B,I,F,G,I}\otimes 1_H).\]
In turn, (horizontal) associativity of the cell $\psi$ yields that
\[\psi_{A,B,C,D,E,F}=(1_{AD}\otimes\psi_{I,I,C,B,E,I}\otimes 1_F)\ast(1_A\otimes\psi_{I,B,E,D,I,I}\otimes 1_{CF}).\]
Putting these together eventually yields
\[\psi_{A,B,C,D,E,F}=(1_{AB}\otimes\omega^{-1_v}_{I,C,D,I,E,I}\otimes 1_F)\ast(1_A\otimes\omega^{-1_v}_{I,B,I,C,D,I}\otimes 1_{EF}).\]
Thus, the cell $\psi$ is precisely the cell $(1\otimes\bar{\omega}^{\mathrm{rev}}\otimes 1)^{\rho}$ as defined in 
Notation~\ref{notation_revbraid} indeed. That means, $\mathcal{C}^+\simeq(\mathcal{C},\beta)$ as $\mathsf{E}_2$-monoids by way of extending 
the identity on $\mathcal{C}$, as was to show. 
This finishes the proof subject to Claim~\ref{claim_braid=E_1_back}.

\paragraph{Proof of Claim~\ref{claim_braid=E_1_back}}
By definition, the $\mathsf{E}_1$-monoid $\mathcal{C}^+$ in $\mathrm{Alg}_{\mathsf{E}_1}(\hotwocat)$ is an algebra
\[\mathcal{C}^+\colon\mathsf{E}_1\otimes_{\mathrm{BV}}\mathsf{E}_1\rightarrow\hotwocat.\]
By the Additivity Theorem, the comparison map $\mathsf{E}_1\otimes_{\mathrm{BV}}\mathsf{E}_1\rightarrow\mathsf{E}_2$ is an equivalence of
$\infty$-operads. In particular, for every $n\geq 0$, the induced functor
\[\mathsf{E}_1\otimes_{\mathrm{BV}}\mathsf{E}_1(\langle n\rangle,\langle 1\rangle)\rightarrow\mathsf{E}_2(\langle n\rangle,\langle 1\rangle)\]
of $n$-ary operations is an equivalence of spaces. In turn, the space $\mathsf{E}_2(\langle n\rangle,\langle 1\rangle)$ is 
equivalent to the configuration space $\mathsf{C}_n(\mathbb{R}^2)$ of $n$-many points in $\mathbb{R}^2$ \cite[Lemma 5.1.1.3]{lurieha}. The 
latter is $1$-truncated by \cite[Corollary 2.1]{nf_config}. It follows that the space
$\mathsf{E}_1\otimes_{\mathrm{BV}}\mathsf{E}_1(\langle n\rangle,\langle 1\rangle)$ is $1$-truncated as well.
Furthermore, for simplicity, let us consider the monoidal envelope
$y^{\otimes}\colon\mathsf{E}_1\otimes_{\mathrm{BV}}\mathsf{E}_1\rightarrow\mathrm{Env}(\mathsf{E}_1\otimes_{\mathrm{BV}}\mathsf{E}_1)^{\otimes}$ 
\cite[Section 2.2.4]{lurieha}. The functor $y^{\otimes}$ is fully faithful \cite[Remark 2.2.4.10]{lurieha}, and so the envelope
$\mathrm{Env}(\mathsf{E}_1\otimes_{\mathrm{BV}}\mathsf{E}_1)^{\otimes}$ is still locally $1$-truncated on objects in the image of 
$y^{\otimes}$. The algebra $\mathcal{C}^+$ factors through the embedding
$y^{\otimes}\colon\mathsf{E}_1\otimes_{\mathrm{BV}}\mathsf{E}_1\rightarrow\mathrm{Env}(\mathsf{E}_1\otimes_{\mathrm{BV}}\mathsf{E}_1)^{\otimes}$. It hence is a post-composition of the
$\mathsf{E}_1\otimes_{\mathrm{BV}}\mathsf{E}_1$-algebra $y^{\otimes}$ in the symmetric monoidal $\infty$-category
$\mathrm{Env}(\mathsf{E}_1\otimes_{\mathrm{BV}}\mathsf{E}_1)^{\otimes}$. As the embedding
$\mathsf{E}_1\sqcup_{\mathsf{E}_0}\mathsf{E}_1\rightarrow\mathsf{E}_1\otimes_{\mathrm{BV}}\mathsf{E}_1$ also is a cofibration, we may without loss of 
generality assume that the two $\mathsf{E}_1$-algebras $U(y^{\otimes})$ and $\mathrm{Alg}_{\mathsf{E}_1}(y^{\otimes})$ coincide.
Let us consider the following initial segment of this algebra in $\mathrm{Env}(\mathsf{E}_1\otimes_{\mathrm{BV}}\mathsf{E}_1)^{\otimes}$:
\begin{align*}
\begin{gathered}
\xymatrix{
 & & & A\otimes A\otimes A\ar[dl]_{(-,u,u,-,u,-)} & \\
& A\otimes A\otimes A\ar@<-.5ex>[d]_{m\otimes 1}\ar@<.5ex>[d]^{1\otimes m}\ar@{}[dr]|{\underset{\omega}{\Rrightarrow}} & (A\otimes A)^{\otimes 3}\ar[l]_{(m, m, m)}\ar@<-.5ex>[d]_{m^{\otimes}\otimes 1}\ar@<.5ex>[d]^{1\otimes m^{\otimes}} & & A\otimes A\otimes A\ar[dl]^{(-,u,u,u,-,-)}\\
& A\otimes A\ar[d]|{m} \ar[d]\ar@{}[dr]|{\underset{\chi}{\Rightarrow}} & (A\otimes A)^{\otimes 2}\ar[d]\ar[l]_{(m, m)}\ar[d]^{m^{\otimes}} & (A\otimes A\otimes A)^{\otimes 2}\ar@<.5ex>[l]^{(m\otimes 1)^{\otimes 2}}\ar@<-.5ex>[l]_{(1\otimes m)^{\otimes 2}}\ar[d]^{m^{\otimes}} \ar@{}[dl]|{\underset{\psi}{\Rrightarrow}} & \\
I\ar[r]^u & A & A\otimes A\ar[l]|{m} & A\otimes A\otimes A\ar@<.5ex>[l]^{m\otimes 1}\ar@<-.5ex>[l]_{1\otimes m} & \\
}
\end{gathered}
\end{align*}
By pasting the two composite 3-cells $(-,u,u,-,u,-)\ast\omega$ and $(-,u,u,u,-,-)\ast\psi$ with a suitable set of vertical and 
horizontal unitors, they both give rise to a composite 2-cell in the hom-space
$\mathrm{Env}(\mathsf{E}_1\otimes_{\mathrm{BV}}\mathsf{E}_1)^{\otimes}(A^{\otimes 3},A)$ with the same boundary. Since $A$ is in the image of 
$y^{\otimes}$, it follows that these two composite 2-cells are necessarily equivalent. As all unitors are mapped by $\mathcal{C}^+$
to identities in $\hotwocat$, and $\hotwocat$ is a 3-category, we see that their images given by the natural 2-cells
\begin{align}\label{diag_braid=E_1_back_forth_2}
\begin{gathered}
\xymatrix{
ABC\ar@{}[drr]|{\omega_{A,I,I,B,I,C}}\ar@{=}[rr]^{1_A\otimes\chi_{I,B,I,C}}\ar@{=}[d]_{\chi_{A,I,I,B}\otimes 1_C} & & ABC\ar@{=}[d]^{\chi_{A,I,I,BC}} \\
ABC\ar@{=}[rr]_{\chi_{A,B,I,C}} & & ABC 
}
\end{gathered}
&
\begin{gathered}
\xymatrix{
ABC\ar@{}[drr]|{\psi_{A,I,I,I,B,C}}\ar@{=}[rr]^{\chi_{A,I,I,BC}}\ar@{=}[d]_{\chi_{A,I,B,C}} & & ABC\ar@{=}[d]^{1_A\otimes\chi_{I,I,B,C}} \\
ABC\ar@{=}[rr]_{\chi_{A,I,I,B}\otimes 1_C} & & ABC 
}
\end{gathered}
\end{align}
are necessarily the same. Thus, as the left hand side is an identity, so is the right hand side. This proves the claim.

\end{proof}

\begin{remark}
Lemma~\ref{lemma_braid=E_1_back_forth} below will take as input a general $\mathsf{E}_1$-monoid structure on a Gray monoid $\mathcal{C}$, 
and deliver as output an equivalent $\mathsf{E}_1$-monoid structure on $\mathcal{C}$ that is semi-strictified so to satisfy the assumptions 
of Lemma~\ref{lemma_braid=E_1_back}. In particular, it will yield a semi-strictified monoidal structure
$(1\otimes\chi_{I,-,-,I}\otimes 1,1\otimes\omega_{I,-,-,-,-,I}\otimes 1,1,1,1)$ on the multiplication 2-functor
$m\colon\mathcal{C}\otimes_{\mathrm{Gr}}\mathcal{C}\rightarrow\mathcal{C}$, and assure that its underlying monoidal 
structure $(\psi,1)$ on the identity $1\colon m(m\otimes 1)\rightarrow m(1\otimes m)$ satisfies a single additional semi-strict 
unitality law. It follows from Claim~\ref{claim_braid=E_1_back} that this procedure then automatically semi-strictifies the 
associated horizontal associator $\psi$ as well, in the sense that $\psi=1\otimes\psi_{I,-,-,-,-,I}\otimes 1$. The argument we gave is 
entirely homotopical, essentially using that the cells $\psi_{-,I,I,I,-,-}$ and $\omega_{-,I,I,-,I,-}$ occupy the same space of equations. 
This space is contractible, and the latter cell is known to be the identity by construction.
This strategy is at heart essentially that of \cite[Theorem 15]{gurski_braid}. Clearly, one ought to be able to prove this algebraically 
as well; the author miserably fails to prove it by hand however. It thus may be worthwhile to note that the algebra in dimension 2 is quite 
complicated, while the homotopy theory of its configuration spaces is trivial in the relevant dimension. In contrast, in dimension 3 the 
algebra is manageable by hand (see next section), while the homotopy theory of its associated configuration spaces is complicated instead
\cite[p.4228]{gurski_braid}, \cite[Corollary 2.1]{nf_config}. 
\end{remark}

\begin{lemma}\label{lemma_braid=E_1_forth_back}
Let $\mathcal{C}$ be a Gray monoid together with a braiding $\beta$. Let $\mathcal{C}^{\beta}$ be the $\mathsf{E}_1$-monoid
structure on $\mathcal{C}$ in $\mathrm{Alg}_{\mathsf{E}_1}(2\text{-Cat}^{\times})$ from Lemma~\ref{lemma_braid=E_1_forth}. Then
$\mathcal{C}^{\beta}$ satisfies the conditions of Lemma~\ref{lemma_braid=E_1_back}. The braiding on $\mathcal{C}$ obtained from
$\mathcal{C}^{\beta}$ in Lemma~\ref{lemma_braid=E_1_back} is equal to $\beta$.
\end{lemma}
\begin{proof}
Let $\mathcal{C}$ be a Gray monoid together with a braiding $\beta$, and let $\mathbb{L}\beta$ be the associated braiding on a cofibrant 
replacement $\mathbb{L}\mathcal{C}$ of $\mathcal{C}$ by way of Lemma~\ref{lemma_mono_cofrepl}. The $\mathsf{E}_1$-monoid
$\mathcal{C}^{\beta}$ is constructed from $(\mathbb{L}\mathcal{C},\mathbb{L}\beta)$ by definition; it is given by corresponding instances 
of $\chi_{\mathbb{L}\beta}$ and $\omega_{\mathbb{L}\beta}$ in Diagram~(\ref{diag_braid=E_1_functors_1}).
The assumptions of Lemma~\ref{lemma_braid=E_1_back} are satisfied by definition.
The construction of the braiding on $\mathcal{C}$ in the proof of Lemma~\ref{lemma_braid=E_1_back} recovers $\mathbb{L}\beta$ on the nose. 
This recovers $\beta$ on the nose after push-forward along the Gray functor $\mathbb{L}\mathcal{C}\twoheadrightarrow\mathcal{C}$.
\end{proof}

\begin{remark}\label{rem_sstricttoweakbraids}
We have worked entirely with semi-strict braidings to reduce the computational effort significantly. With some more work one can generalise 
Remark~\ref{rem_braidtomono} to show that every fully weak braiding induces a monoidal structure on the multiplication of a Gray 
monoid such that the identity becomes monoidal when equipped with the 3-cell that given by diagonal reflection of the fully weak reverse 
braiding. One can prove Lemma~\ref{lemma_braid=E_1_forth} for fully weak braidings accordingly.
Since Lemma~\ref{lemma_braid=E_1_back} always returns a semi-strict braiding, one can rephrase Lemma~\ref{lemma_braid=E_1_forth_back} to 
show that every fully weak braiding is equivalent to a semi-strict one. This gives an alternative proof of
\cite[Corollary 28]{gurski_braid}, which we have used to justify the reduction to semi-strict braidings for convenience.
\end{remark}

We are left to show that every $\mathsf{E}_1$-monoid structure on a Gray monoid $\mathcal{C}$ is equivalent to one that satisfies the 
extra assumptions of Lemma~\ref{lemma_braid=E_1_back}. By cofibrantly replacing $\mathcal{C}$, we can assume that $\mathcal{C}$ is 
cofibrant without loss of generality.

\begin{lemma}\label{lemma_braid=E_1_back_forth}
Let $\mathcal{C}$ be a cofibrant Gray monoid. Let $\mathcal{C}^+$ be an $\mathsf{E}_1$-monoid structure on $\mathcal{C}$ in
$\mathrm{Alg}_{\mathsf{E}_1}(2\text{-Cat}^{\times})$. Then the identity on $\mathcal{C}$ can be extended to an equivalence between
$\mathcal{C}^+$ and an $\mathsf{E}_1$-monoid structure $\mathcal{C}^-$ on $\mathcal{C}$ such that the latter satisfies Equations (1)-(4) in 
Lemma~\ref{lemma_braid=E_1_forth}.
\end{lemma}
\begin{proof}
Let $\mathcal{C}^+$ be an $\mathsf{E}_1$-monoid structure on $\mathcal{C}$. Without loss of generality, we may further assume that 
the underlying horizontal $\mathsf{E}_1$-monoid structure $\mathrm{Alg}_{\mathsf{E}_1}(U)(\mathcal{C}^+)$ is given by the underlying 
vertical $\mathsf{E}_1$-monoid structure $U(\mathcal{C}^+)=\mathcal{C}$ itself as well. 
Thus, by way of Corollary~\ref{cor_eh_3}, the $\mathsf{E}_2$-monoid $\mathcal{C}^+$ is determined by its underlying diagram in $\hotwocat$ 
of the form (\ref{diag_braid=E_1_back_1}).
\begin{align}\label{equ_braid=E_1_back}
\xymatrix{
\mathcal{C} & \mathcal{C}^{\otimes_{\mathrm{Gr}}2}\ar[l]|{m} & \mathcal{C}^{\otimes_{\mathrm{Gr}}3}\ar@<.5ex>[l]^{m\otimes 1}\ar@<-.5ex>[l]_{1\otimes m} 
}
\end{align}
for $m$ the multiplication 2-functor of $\mathcal{C}$ together with the identity associator. By Lemma~\ref{lemma_halgtrans} the cells
$(\chi,\omega)$ are part of a monoidal structure $(\chi,\omega,\iota,\zeta,\kappa)$ on the 2-functor
$m\colon\mathcal{C}\otimes_{\mathrm{Gr}}\mathcal{C}\rightarrow\mathcal{C}$ that represents the multiplication $m$ in
(\ref{equ_braid=E_1_back}) in $\mathrm{Alg}_{\mathsf{E}_1}(\hotwocat)$. 

The morphism $\iota\colon I\rightarrow m(I,I)$ is part of --- and 
hence determined by --- the (unital) horizontal $\mathsf{E}_1$-monoid (\ref{equ_braid=E_1_back}). It follows that $\iota=1_I$. Hence, by 
Lemma~\ref{lemma_cofbraidunit} we can assume that $\zeta=\kappa=1$ and that $\chi_{-,I,I,-}=1$ as well. 
The identity
$1\colon m(m\otimes 1)\xrightarrow{\simeq}m(1\otimes m)$ in (\ref{equ_braid=E_1_back}) is equipped with a monoidal structure given by the
associator 3-cell $\psi$ and a unitor 2-cell $1_0$. The latter is again the identity $1_{1_{I}}$ because it is determined by the horizontal 
(unital)$\mathsf{E}_1$-monoid (\ref{equ_braid=E_1_back}) given by the Gray monoid $\mathcal{C}$ itself. The following further equations arise:
\begin{itemize}
\item $\chi_{I,I,-,-}=\chi_{-,-,I,I}=1$ and $\omega_{-,-,I,I,-,-}=1$, because $\zeta=\kappa=1$;
\item $\omega_{I,I,-,-,-,-}=\omega_{-,-,-,-,I,I}=1$ by $\omega_{-,-,I,I,-,-}=1$ and associativity of $\omega$ \cite[p.104]{daystreet} for $A=B=C=D=1$ and $E=F=G=H=1$, respectively;
\end{itemize}

We first enforce Assumption (4) regarding the horizontal associator $\psi$ by strictifying its associated horizontal unitors as well. 
Therefore, we note that the horizontal $\mathsf{E}_1$-monoid structure of $\mathcal{C}^+$ --- represented by a  diagram of the form 
(\ref{diag_braid=E_1_back_1}) --- contains a pair of horizontal unitors
$\zeta^h\colon\chi_{I,-,I,-}\rightarrow 1$ and $\kappa^h\colon\chi_{-,I,-,I}\rightarrow 1$ in $\mathrm{Alg}_{\mathsf{E}_1}(\hotwocat)$. 
By Lemma~\ref{lemma_halgtrans}, these two equivalences correspond to monoidal pseudo-natural equivalences from the composite monoidal
2-endofunctor $m(-,I)$ (respectively, $m(I,-)$) to the monoidal identity $(1,1,1,1,1,1)$ on $\mathcal{C}$. The former is the monoidal
2-functor $(1,\chi_{I,-,I,-},1,1,1,1)$ (respectively, $(1,\chi_{-,I,-,I},1,1,1,1)$) by $\omega_{I,-,I,-,I,-}=\omega_{-,I,-,I,-,I}=1$ and by 
the strictly unital choice of the monoidal structure on the monoidal 2-functor $m$. In particular, it follows 
that \begin{align}\label{equ_braid=E_1_back_unitors}
\zeta^{h}_{-,I}=\zeta^{h}_{I,-}=\kappa^{h}_{I,-}=\kappa^{h}_{-,I}=1.
\end{align}
We want to show that these two horizontal unitors can be replaced entirely by identities as well. Therefore, for simplicity, we may work 
with the diagonal reflection $(\mathcal{C}^+)^{\rho}$ for the moment. Here, 
\begin{enumerate}[label=\roman*.]
\item we still have $\iota^{\rho}=\iota=1_I$, and
\item $\chi^{\rho}_{-,I,I,-}=\chi^{\rho}_{I,-,I,-}=\chi^{\rho}_{-,I,-,I}=1$, but
\item $\psi^{\rho}_{-,I,-,-,I,-}=1$,
\end{enumerate}
with no assumptions on $\omega^{\rho}$ instead. We want to replace the monoidal structure
$(\chi^{\rho},\omega^{\rho},\zeta^{\rho},\kappa^{\rho},1)$ on $m$ by some monoidal structure 
structure of the form $((\chi^{\rho})^u,(\omega^{\rho})^u,1,1,1)$ along a suitable monoidal pseudo-natural equivalence
$(\theta,\theta_2,\theta_0)$ such that i., ii., and iii.\ still hold. Therefore, we consider the pseudo-natural equivalence
\begin{align*}
\begin{gathered}
\xymatrix{
 & \\
ABCD\ar@/_2pc/@{=}[rrr]\ar@/^2pc/[rrr]^{(\chi^{\rho})^{-1}_{A,B,I,I}\otimes(\chi^{\rho})^{-1}_{I,I,C,D}}\ar@{}[rrr]|{\Downarrow(\zeta^{\rho})^{-1_{h}}_{A,B}\otimes(\kappa^{\rho})^{-1_{h}}_{C,D}} & & & ABCD,\\
&
}
\end	{gathered}
\end{align*}
and note that $\zeta^{\rho}_{-,I}=\zeta^{\rho}_{I,-}=\kappa^{\rho}_{I,-}=\kappa^{\rho}_{-,I}=1$ by (\ref{equ_braid=E_1_back_unitors}). We 
consider the induced pseudo-natural equivalence $\theta\colon m\rightarrow m$ given by the identity $1$, together with the modification
\begin{align}\label{diag_braid=E_1_back_theta2}
\begin{gathered}
\xymatrix{
ABCD\ar[rrrr]^{((\chi^{\rho})^{-1}_{A,B,I,I}\otimes(\chi^{\rho})^{-1}_{I,I,C,D})\ast\chi^{\rho}_{A,B,C,D}}\ar[d]_{1_{AB}\otimes 1_{CD}}\ar@{}[drrrr]|{\Downarrow\theta_2} & & & & ACBD\ar[d]^{1_{ACBD}} \\
ABCD\ar[rrrr]_{\chi^{\rho}_{A,B,C,D}} & & & & ACBD
}
\end{gathered}
\end{align}
for
\[\theta_2:=\chi^{\rho}_{A,B,C,D}\ast\left((\zeta^{\rho})^{-1}_{A,B}\otimes(\kappa^{\rho})^{-1}_{C,D}\right),\]
and $\theta_0$ the identity unitor. Conjugation of $(\chi^{\rho},\omega^{\rho},\zeta^{\rho},\kappa^{\rho},1)$ with $(1,\theta_2,1_I)$ 
yields an equivalent monoidal structure $((\chi^{\rho})^u,(\omega^{\rho})^u,1,1,1)$ with
$(\chi^{\rho})^u_{A,B,C,D}=(\chi_{A,B,I,I}^{-1}\otimes\chi_{I,I,C,D}^{-1})\ast\chi^{\rho}_{A,B,C,D}$. In particular,
$(\omega^{\rho})^u_{-,-,I,I,-,-}=1$ and
\[(\chi^{\rho})^u_{I,I,-,-}=(\chi^{\rho})^u_{-,-,I,I}=(\chi^{\rho})^u_{-,I,I,-}=(\chi^{\rho})^u_{-,I,-,I}=(\chi^{\rho})^u_{I,-,I,-}=1.\]
This further defines a new but equivalent $\mathsf{E}_1$-structure on $\mathcal{C}$ in $\mathrm{Alg}_{\mathsf{E}_1}(\hotwocat)$ equivalent to
$(\mathcal{C}^+)^{\rho}$. Its underlying horizontal associator $(\psi^{\rho})^{u}$ --- defining a monoidal structure on the identity
$1\colon m(m\otimes 1)\rightarrow m(1\otimes m)$ --- is given by conjugation as follows.
\[\xymatrix{
ABCDEF\ar[rr]_{\chi^{\rho}_{A,BC,D,EF}}\ar@/^3pc/[rr]^{\chi^{\rho}_{A,BC,D,EF}}\ar@{}[ddrr]|{\underset{\psi^{\rho}{A,B,C,D,E,F}}{\Rightarrow}}\ar@{}@<2ex>[rr]^{\Uparrow \theta_{2 A,BC,D,EF}}\ar@/^1pc/[dd]^{\rotatebox{90}{$\scriptstyle \chi^{\rho}_{AB,C,DE,F}$}}\ar@/_4pc/[dd]_{\rotatebox{90}{$\scriptstyle \chi^{\rho}_{AB,C,DE,F}$}}\ar@{}@<3ex>[dd]_{\underset{\theta^{-1}_{2 AB,C,DE,F}}{\Rightarrow}} & & ADBCEF\ar@/_1pc/[dd]_{\rotatebox{-90}{$\scriptstyle 1_{AD} \otimes \chi^{\rho}_{B,C,E,F}$}}\ar@/^4pc/[dd]^{\rotatebox{-90}{$\scriptstyle 1_{AD} \otimes \chi^{\rho}_{B,C,E,F}$}}\ar@{}@<-3ex>[dd]^{\underset{1_{AD} \otimes \theta_{2 B,C,E,F}}{\Rightarrow}} \\
 & & \\
ABDECF\ar[rr]^{\chi^{\rho}_{A,B,D,E}\otimes 1_{CF}}\ar@/_3pc/[rr]_{\chi^{\rho}_{A,B,D,E}\otimes 1_{CF}}\ar@{}@<-2ex>[rr]_{\Uparrow \theta^{-1}_{2 A,B,D,E}\otimes 1_{CF}} & & ACEBDF.
 }\]
One computes that $\theta_{2 I,A,I,B}=\theta_{2 A,I,B,I}=1$, and so $(\psi^\rho)^u_{-,I,-,-,I,-}=1$ assuming that
$\psi^\rho_{-,I,-,-,I,-}=1$. Thus, in turn, we obtain an $\mathsf{E}_1$-monoid structure $\mathcal{C}^+$ on $\mathcal{C}$ that is 
equivalent to the one that we started with, and such that 
\begin{itemize}
\item $\chi_{I,I,-,-}=\chi_{-,-,I,I}=1$;
\item $\omega_{-,-,I,I,-,-}=\omega_{I,I,-,-,-,-}=\omega_{-,-,-,-,I,I}=1$;
\item $\psi(-,I,-,-,I,-)=\psi(I,-,-,I,-,-)=\psi(-,-,I,-,-,I)=1$, because $\omega^{\rho}_{-,-,I,I,-,-}=1$ by the above;
\item $\chi_{I,-,I,-}^{-1}\ast \psi(-,I,I,-,I,I)\colon\chi_{I,-,I,-}\simeq 1$ and
$\chi_{-,I,-,I}^{-1}\ast \psi(I,I,-,I,I,-)\colon\chi_{-,I,-,I}\simeq 1$; thus $\chi_{I,-,I,-}=\chi_{-,I,-,I}=1$.
\end{itemize}
%
%
By way of these equations, we may define a series of new (stricter but equivalent) monoidal structures on $m$ to enforce Assumption (3) of Lemma~\ref{lemma_braid=E_1_back} as follows. Consider the pseudo-natural 
equivalence $\theta\colon m\rightarrow m$ given by the identity $1$, together with the modification
\begin{align}\label{diag_braid=E_1_back_theta2}
\begin{gathered}
\xymatrix{
ABCD\ar[rr]^{1_A\otimes \chi_{I,B,C,D}}\ar[d]_{1_{AB}\otimes 1_{CD}}\ar@{}[drr]|{\Downarrow\theta_2} & & ACBD\ar[d]^{1_{ACBD}} \\
ABCD\ar[rr]_{\chi_{A,B,C,D}} & & ACBD
}
\end{gathered}
\end{align}
defined as the following pasting:
\[\xymatrix{
ABCD\ar@{=}[dd]\ar[rrr]^(.4){1_A\otimes \chi_{I,B,C,D}} & & & ACBD\ar@/_3pc/[dd]_{\rotatebox{90}{$\scriptstyle \chi_{A,I,C,BD}$}}\ar@/^3pc/@{=}[dd]\ar@{}[dd]|{\underset{\omega_{A,I,C,I,I,BD}}{\Leftarrow}} \\
\ar@{}[rrr]|(.35){\Downarrow\omega^{-1}_{A,I,I,B,C,D}}  & & & \\
ABCD\ar[rrr]_{\chi_{A,B,C,D}} & & & ACBD.
}\]

Let $\theta_0:=1_{1_I}$. Define the natural 2-cell
\begin{align}\label{diag_equ_braid=E_1_back_l}
\begin{gathered}
\xymatrix{
ABCDEF\ar[rr]\ar[dd]\ar@{}[ddrr]|{\underset{\Rightarrow}{\omega^l_{A,B,C,D,E,F}}} & & ACBDEF\ar[dd] \\
 & & \\
ACBDEF\ar[rr] & & ACEBDF
}
\end{gathered}
\end{align}
by conjugation of $\omega$ with $\theta_2$ (Remark~\ref{rem_conj}). That is to say, we define (\ref{diag_equ_braid=E_1_back_l}) as the 
following pasting diagram:
%

\[\xymatrix{
ABCDEF\ar[rr]_{1_{AB}\otimes\chi_{C,D,E,F}}\ar@/^3pc/[rr]^{1_{ABC}\otimes\chi_{1,D,E,F}}\ar@{}[ddrr]|{\underset{\omega_{A,B,C,D,E,F}}{\Rightarrow}}\ar@{}@<2ex>[rr]^{\Uparrow 1_{AB}\otimes\theta_{2 C,D,E,F}^{-1}}\ar@/^1pc/[dd]^{\rotatebox{90}{$\scriptstyle \chi_{A,B,C,D}\otimes 1_{EF}$}}\ar@/_4pc/[dd]_{\rotatebox{90}{$\scriptstyle 1_A\otimes \chi_{1,B,C,D}\otimes 1_{EF}$}}\ar@{}@<3ex>[dd]_{\underset{\theta_{2 A,B,C,D}\otimes 1_{EF}}{\Rightarrow}} & & ABCEDF\ar@/_1pc/[dd]_{\rotatebox{-90}{$\scriptstyle \chi_{A,B,CE,DF}$}}\ar@/^4pc/[dd]^{\rotatebox{-90}{$\scriptstyle 1_A\otimes \chi_{1,B,CE,DF}$}}\ar@{}@<-1.5ex>[dd]^{\underset{\theta^{-1}_{2 A,B,CE,DF}}{\Rightarrow}} \\
 & & \\
ACBDEF\ar[rr]^{\chi_{AC,BD,E,F}}\ar@/_3pc/[rr]_{1_{AC}\otimes\chi_{1,BD,E,F}}\ar@{}@<-2ex>[rr]_{\Uparrow \theta_{2 AC,BD,E,F}} & & ACEBDF.
}\]
It then follows that $(1\otimes \chi_{I,-,-,-},\omega^l,1,1,1)$ is a monoidal structure on $m$, and that
$(1,\theta_2,\theta_0)$ is a monoidal pseudo-natural equivalence from $(m,\chi,\omega,1,1,1)$ to
$(m,1\otimes \chi_{I,-,-,-},\omega^l,1,1,1)$. 

We do the same one more time, with respect to the pseudo-natural equivalence $\theta\colon m\rightarrow m$ given by the identity $1$, 
together with the modification
\begin{align}\label{diag_braid=E_1_back_theta_lr}
\begin{gathered}
\xymatrix{
ABCD\ar[rr]^{1_A\otimes \chi_{I,B,C,I}\otimes 1_D}\ar[d]_{1_{AB}\otimes 1_{CD}}\ar@{}[drr]|{\Downarrow\omega^l_{A,B,C,I,I,D}} & & ACBD\ar[d]^{1_{ACBD}} \\
ABCD\ar[rr]_{1_A\otimes \chi_{I,B,C,D}} & & ACBD
}
\end{gathered}
\end{align}
for $\theta_2$, and again $\theta_0:=1_{1_I}$. We define the natural 2-cell
\begin{align}\label{diag_equ_braid=E_1_back}
\begin{gathered}
\xymatrix{
ABCDEF\ar[rrr]^{1_A\otimes\chi_{I,B,C,I}\otimes 1_{DEF}}\ar[dd]_{1_{ABC}\otimes\chi_{I,D,E,I}\otimes 1_F}\ar@{}[ddrrr]|{\underset{\Rightarrow}{\omega^{lr}_{A,B,C,D,E,F}}} & & & ACBDEF\ar[dd]^{1_{AC}\otimes\chi_{I,BD,E,I}\otimes 1_{F}} \\
& & & \\
ACBDEF\ar[rrr]_{1_{A}\otimes\chi_{I,B,CE,I}\otimes 1_{DF}} & & & ACEBDF
}
\end{gathered}
\end{align}
by conjugation of $\omega^l$ with $\theta_2$ (Remark~\ref{rem_conj}). 
It then follows again that $(1\otimes \chi_{I,-,-,I}\otimes 1,\omega^{lr},1,1,1)$ is a monoidal structure on $m$, and that
$(1,\theta_2,1)$ is a monoidal pseudo-natural equivalence from $(m,1\otimes \chi_{I,-,-,-},\omega^{l},1,1,1)$
to $(m,1\otimes \chi_{I,-,-,I}\otimes 1,\omega^{lr},1,1,1)$. 

This monoidal structure satisfies Assumption (2) of Lemma~\ref{lemma_braid=E_1_back} by construction. The fact that the associator
$\omega^{lr}$ satisfies Assumption (3) is a straightforward computation; indeed, $\omega^l_{-,-,I,I,-,-}=1$ holds because 
$\omega_{-,-,I,I,-,-}=1$ holds, and $\omega^l_{-,I,I,-,-,-}=1$ holds by construction. The replacement $\omega^{lr}$ then takes care of the 
right hand side. To see that Assumption (4) holds, we note that it held in $\mathcal{C}^+$ by construction. To see that it still holds 
after the two replacements, it suffices to show that the two corresponding monoidal pseudo-natural equivalences $(1,\theta_2,1)$ satisfy
$\theta_{2 I,A,I,B}=\theta_{2 A,I,B,I}=1$. This is straightforward.
Thus, we obtain an equivalent $\mathsf{E}_1$-monoid structure on $\mathcal{C}$ that satisfies all assumptions of 
Lemma~\ref{lemma_braid=E_1_back}.
\end{proof}

\begin{remark}
Theorem~\ref{thm_braid=E_1} can be thought of as an analogon to Joyal--Street's characterisation of braided monoidal 
categories as monoidal categories with a multiplication (Example~\ref{exple_eh_2}). Here it may be worth noting that Joyal and Street's 
definition of the braiding is given by the composition
\[\rho_{A,B}\colon A\otimes B\xrightarrow{\chi_{I,A,B,I}}B\otimes A\xrightarrow{\chi_{B,I,I,A}^{-1}}B\otimes A\]
rather than by the morphism $\chi_{I,A,B,I}$ only. The correctional term $\chi_{B,I,I,A}^{-1}$ is made redundant in the proof of 
Lemma~\ref{lemma_braid=E_1_back} by way of the extra equation $\chi_{-,I,I,-}=1$. This extra equation is imposed by using the 
replacement of a general $\mathsf{E}_1$-monoid structure on a cofibrant Gray monoid $\mathcal{C}$ provided by 
Lemma~\ref{lemma_cofbraidunit}. If one wishes to define the braiding associated to an $\mathsf{E}_1$-monoid structure on a Gray monoid
$\mathcal{C}$ directly in terms of the underlying monoidal structure $(\chi,\omega,\iota,\zeta,\kappa)$ on its multiplication $m$, then 
this correctional term reappears.
\end{remark}

This proves Theorem~\ref{thm_braid=E_1}. The proof of Proposition~\ref{prop_braid=E_1_functors} is very similar in methodology.

\begin{lemma}\label{lemma_braid=E_1_morphisms_forth}
Let $(\mathcal{C},\beta)$ and $(\mathcal{D},\beta)$ be braided Gray monoids, and let $F\colon\mathcal{C}\rightarrow\mathcal{D}$ be a 
monoidal 2-functor of underlying Gray monoids. Let $f\colon\mathcal{C}\rightarrow\mathcal{D}$ be the associated morphism of $\mathsf{E}_1$-
monoids in $\hotwocat$ from Lemma~\ref{lemma_halgtrans}. Let $\mathcal{C}^{+}$ and $\mathcal{D}^{+}$ be the associated $\mathsf{E}_2$-
monoids in $\hotwocat$ from Theorem~\ref{thm_braid=E_1}. Then every braiding on $F$ extends $f$ to a morphism
$f^{+}\colon\mathcal{C}^{+}\rightarrow\mathcal{D}^{+}$ of $\mathsf{E}_2$-monoids.
\end{lemma}
\begin{proof}
Without loss of generality $\mathcal{C}$ and $\mathcal{D}$ are cofibrant.
By way of Lemma~\ref{prop_inisegloc} to extend $f$ to a morphism of $\mathsf{E}_2$-monoids it suffices to give a morphism from the diagram
\begin{align}\label{diag_braid=E_1_morphisms_forth_1}
\begin{gathered}
\xymatrix{
\mathcal{C}^{\otimes_{\mathrm{Gr}}4}\ar@<-.5ex>@/_1pc/[d]_{m\otimes 1\otimes 1}\ar[d]|{1\otimes m\otimes 1}\ar@<.5ex>@/^1pc/[d]^{1\otimes 1\otimes m} & & & \\
\mathcal{C}^{\otimes_{\mathrm{Gr}}3}\ar@<-.5ex>[d]_{m\otimes 1}\ar@<.5ex>[d]^{1\otimes m}\ar@{}[dr]|{\omega_{\beta}} & (\mathcal{C}^{\otimes_{\mathrm{Gr}}2})^{\otimes_{\mathrm{Gr}} 3}\ar[l]_{(m, m, m)}\ar@<-.5ex>[d]_{m^{\otimes}\otimes 1}\ar@<.5ex>[d]^{1\otimes m^{\otimes}} & & & \\
\mathcal{C}^{\otimes_{\mathrm{Gr}}2}\ar[d]|{m} \ar[d]\ar@{}[dr]|{\chi_{\beta}} & (\mathcal{C}^{\otimes_{\mathrm{Gr}}2})^{\otimes_{\mathrm{Gr}} 2}\ar[d]\ar[l]_{(m, m)}\ar[d]^{m^{\otimes}}\ar@{}[dr]|{(\omega_{\beta^{\mathrm{rev}}})^{\rho}} & (\mathcal{C}^{\otimes_{\mathrm{Gr}}3})^{\otimes_{\mathrm{Gr}} 2}\ar@<.5ex>[l]^{(m\otimes 1)^{\otimes 2}}\ar@<-.5ex>[l]_{(1\otimes m)^{\otimes 2}}\ar[d]^{m^{\otimes}} & \\
\mathcal{C} & \mathcal{C}^{\otimes_{\mathrm{Gr}}2}\ar[l]|{m} & \mathcal{C}^{\otimes_{\mathrm{Gr}}3}\ar@<.5ex>[l]^{m\otimes 1}\ar@<-.5ex>[l]_{1\otimes m} & \mathcal{C}^{\otimes_{\mathrm{Gr}}4}\ar@<-.5ex>@/_1pc/[l]_{m\otimes 1\otimes 1}\ar[l]|{1\otimes m\otimes 1}\ar@<.5ex>@/^1pc/[l]^{1\otimes 1\otimes m} 
}
\end{gathered}
\end{align}
to the diagram
\begin{align}\label{diag_braid=E_1_morphisms_forth_2}
\begin{gathered}
\xymatrix{
\mathcal{D}^{\otimes_{\mathrm{Gr}}4}\ar@<-.5ex>@/_1pc/[d]_{m\otimes 1\otimes 1}\ar[d]|{1\otimes m\otimes 1}\ar@<.5ex>@/^1pc/[d]^{1\otimes 1\otimes m} & & & \\
\mathcal{D}^{\otimes_{\mathrm{Gr}}3}\ar@<-.5ex>[d]_{m\otimes 1}\ar@<.5ex>[d]^{1\otimes m}\ar@{}[dr]|{\omega_{\beta}} & (\mathcal{D}^{\otimes_{\mathrm{Gr}}2})^{\otimes_{\mathrm{Gr}} 3}\ar[l]_{(m, m, m)}\ar@<-.5ex>[d]_{m^{\otimes}\otimes 1}\ar@<.5ex>[d]^{1\otimes m^{\otimes}} & & & \\
\mathcal{D}^{\otimes_{\mathrm{Gr}}2}\ar[d]|{m} \ar[d]\ar@{}[dr]|{\chi_{\beta}} & (\mathcal{D}^{\otimes_{\mathrm{Gr}}2})^{\otimes_{\mathrm{Gr}} 2}\ar[d]\ar[l]_{(m, m)}\ar[d]^{m^{\otimes}}\ar@{}[dr]|{(\omega_{\beta^{\mathrm{rev}}})^{\rho}} & (\mathcal{D}^{\otimes_{\mathrm{Gr}}3})^{\otimes_{\mathrm{Gr}} 2}\ar@<.5ex>[l]^{(m\otimes 1)^{\otimes 2}}\ar@<-.5ex>[l]_{(1\otimes m)^{\otimes 2}}\ar[d]^{m^{\otimes}} & \\
\mathcal{D} & \mathcal{D}^{\otimes_{\mathrm{Gr}}2}\ar[l]|{m} & \mathcal{D}^{\otimes_{\mathrm{Gr}}3}\ar@<.5ex>[l]^{m\otimes 1}\ar@<-.5ex>[l]_{1\otimes m} & \mathcal{D}^{\otimes_{\mathrm{Gr}}4}\ar@<-.5ex>@/_1pc/[l]_{m\otimes 1\otimes 1}\ar[l]|{1\otimes m\otimes 1}\ar@<.5ex>@/^1pc/[l]^{1\otimes 1\otimes m} 
}
\end{gathered}
\end{align}
that on the left vertical axis is given by $f$. As stated in \cite[p.125]{daystreet}, every braiding $u$ on
$F=(f,\chi,\omega,1,1,1)$ induces a monoidal structure $(\chi_2,1)$ on the $2$-cell $\chi$ with respect to the monoidal structure 
on the two multiplications $m_{\beta}$ of $\mathcal{C}$ and $\mathcal{D}$ induced by their respective braiding. By 
Lemma~\ref{lemma_halgtrans} the pair $(\chi_2,1)$ determines a natural equivalence between the two multiplications $m_{\beta}$ as morphisms 
of $\mathsf{E}_1$-monoids. This in turn means we are given a 
morphism of non-unital $\mathsf{A}_2$-monoids in $\mathrm{Alg}_{\mathsf{E}_1}(\hotwocat)$ (the two left vertical blocks in (\ref{diag_braid=E_1_morphisms_forth_1}) and (\ref{diag_braid=E_1_morphisms_forth_2}), respectively). To extend this morphism two the 
entirety of the two diagrams, we use the same trick as in Lemma~\ref{lemma_braid=E_1_forth}: The braided monoidal 2-functor 
$(F,u)\colon(\mathcal{C},\beta)\rightarrow(\mathcal{D},\beta)$ induces a canonical reversely braided monoidal 2-functor
$(F,u^{\mathrm{rev}})\colon(\mathcal{C},\beta^{\mathrm{rev}})\rightarrow(\mathcal{D},\beta^{\mathrm{rev}})$ by 
Lemma~\ref{lemma_braid_rev_morphism}. The induced morphism
$f^{\mathrm{rev}}\colon(\mathcal{C},\chi_{\beta^{\mathrm{rev}}},\omega_{\beta^{\mathrm{rev}}})\rightarrow(\mathcal{D},\chi_{\beta^{\mathrm{rev}}},\omega_{\beta^{\mathrm{rev}}})$ of induced
$\mathsf{A}_2^{\mathrm{nu}}\otimes_{\mathrm{BV}} \mathsf{E}_1$-monoids induces a diagonal reflection
$(f^{\mathrm{rev}})^{\rho}\colon\colon(\mathcal{C},\chi_{\beta^{\mathrm{rev}}},\omega_{\beta^{\mathrm{rev}}})^{\rho}\rightarrow(\mathcal{D},\chi_{\beta^{\mathrm{rev}}},\omega_{\beta^{\mathrm{rev}}})^{\rho}$. By way of the explicit formula for the monoidal structure
$\chi_2^{\mathrm{rev}}$ on $\chi$ induced by $u^{\mathrm{rev}}$ given on \cite[p.125]{daystreet}, it is not hard to compute that
\[(\chi_2^{\mathrm{rev}})_{A,B,C,D}=(\chi_2^{-1_h})_{A,C,B,D}.\]
Here, we consider $\chi_2$ as a natural 2-cell of the form
\[\xymatrix{
f(A)f(B)f(C)f(D)\ar[rr]^{1\otimes\rho_{f(B),f(C)}\otimes 1}\ar[d]_{\chi_{A,B}\otimes \chi_{C,D}}\ar@{}[ddrr]|{\underset{\chi_2}{\Rightarrow}} & & f(A)f(C)f(B)f(D)\ar[d]^{\chi_{A,C}\otimes\chi_{B,D}} \\
f(AB)f(CD)\ar[d]_{\chi_{AB,CD}} & & f(AC)f(BD)\ar[d]^{\chi_{AC,BD}}\\
f(ABCD)\ar[rr]_{f(1\otimes\rho_{B,C}\otimes 1)} & & f(ACBD)
}\]
That means, the horizontal block $(f^{\mathrm{rev}})^{\rho}$ coincides with $f$ on
$\mathsf{A}_2^{\mathrm{nu}}\otimes\mathsf{A}_2^{\mathrm{nu}}$. It hence extends the non-unital $\mathsf{A}_2$-monoid morphism $f$ to a 
morphism from (\ref{diag_braid=E_1_morphisms_forth_1}) to (\ref{diag_braid=E_1_morphisms_forth_2}) by way of Example~\ref{exple_diagrefl2}. 
It hence gives rise to a morphism of non-unital $\mathsf{E}_2$-monoids as stated. Unitality follows again by way of 
Lemma~\ref{lemma_unitff}.
\end{proof}

\begin{lemma}\label{lemma_braid=E_1_morphisms_back}
Let $(\mathcal{C},\beta)$ and $(\mathcal{D},\beta)$ be braided Gray monoids, and let $F\colon\mathcal{C}\rightarrow\mathcal{D}$ be a monoidal 
2-functor. Let $f\colon\mathcal{C}\rightarrow\mathcal{D}$ be the associated morphism of $\mathsf{E}_1$-monoids in $\hotwocat$ 
from Lemma~\ref{lemma_halgtrans}. Suppose $f^+\colon(\mathcal{C}^{+},\beta)\rightarrow(\mathcal{D}^{+},\beta)$ is a morphism of associated
$\mathsf{E}_2$-monoids extending $f$. Then there is a braiding on $f$ with respect to the given braidings on $\mathcal{C}$ and $\mathcal{D}$.
\end{lemma}
\begin{proof}
Without loss of generality we may assume that $\mathcal{C}$ and $\mathcal{D}$ are cofibrant. By Lemma~\ref{lemma_cofbraidunit} we may 
further assume that $F\colon\mathcal{C}\rightarrow\mathcal{D}$ is a monoidal 2-functor of the form $(f,\chi,\omega,1,1,1)$, so it exhibits 
degenerate unitality laws. The morphism $f^+\colon(\mathcal{C},\beta)\rightarrow(\mathcal{D},\beta)$ comes together with a 3-cell $\chi_2$ 
of the form
\begin{align}\label{diag_braid=E_1_morphisms_back0}
\begin{gathered}
\xymatrix{
 & \mathcal{C}^{\otimes_{\mathrm{Gr}} 2}\ar[rr]_{f^{\otimes_{\mathrm{Gr}} 2}}\ar[dd]^(.3){m}|\hole & & \mathcal{D}^{\otimes_{\mathrm{Gr}} 2} \ar[dd]^{m} \\
(\mathcal{C}^{\otimes_{\mathrm{Gr}} 2})^{\otimes_{\mathrm{Gr}} 2}\ar[rr]^(.3){f^{\otimes_{\mathrm{Gr}} 4}}\ar[dd]^{m^{\otimes}}\ar[ur]^{(m,m)} & & (\mathcal{D}^{\otimes_{\mathrm{Gr}} 2})^{\otimes_{\mathrm{Gr}} 2}\ar[dd]^(.3){m^{\otimes}}|(.52)\hole\ar[ur]_{(m,m)} \\
 & \mathcal{C}\ar[rr]_(.3)f & & \mathcal{D} \\
\mathcal{C}^{\otimes_{\mathrm{Gr}} 2}\ar[rr]_{f^{\otimes_{\mathrm{Gr}} 2}}\ar[ur]^{m} & & \mathcal{D}^{\otimes_{\mathrm{Gr}} 2}\ar[ur]_{m} & \\
}
\end{gathered}
\end{align} 
whose back and bottom face are both $\chi$ (by way of Lemma~\ref{cor_E_latching}), whose left and right face are 
the 2-cell $\chi_{\beta}=1\otimes\rho\otimes 1$ associated to the braidings on $\mathcal{C}$ and $\mathcal{D}$, respectively, and whose front 
and top faces are given by the corresponding product of $\chi$. For every tuple $(A,B,C,D)$ of objects in $\mathcal{C}$, this corresponds to 
a natural 2-cell of the form
\begin{align}\label{diag_braid=E_1_morphisms_back}
\begin{gathered}
\xymatrix{
f(A)f(B)f(C)f(D)\ar[d]_{\chi_{A,B}\otimes \chi_{C,D}}\ar[rrr]^{1_{f(A)}\otimes\rho_{f(B),f(C)}\otimes 1_{f(D)}} & & & f(A)f(C)f(B)f(D)\ar[rr]^(.55){\chi_{A,C}\otimes \chi_{B,D}}\ar@{}[d]|{\underset{\chi_{2 A,B,C,D}}{\Rightarrow}} & & f(AC)f(BD)\ar[d]^{\chi_{AC,BD}} \\
f(AB)f(CD)\ar[rrr]_{\chi_{AB,CD}} & & & f(ABCD)\ar[rr]_{f(1_A\otimes\rho_{B,C}\otimes 1_D)} & & f(ACBD).
}
\end{gathered}
\end{align}
It equips the pseudo-natural equivalence
\begin{align}\label{equ_braid=E_1_morphisms_back0}
\begin{gathered}
\xymatrix{
\mathcal{C}\otimes_{\mathrm{Gr}}\mathcal{C}\ar[r]^{f\otimes_{\mathrm{Gr}}f}\ar[d]_{m}\ar@{}[dr]|{\Downarrow \chi_2} & \mathcal{D}\otimes_{\mathrm{Gr}}\mathcal{D}\ar[d]^m
\\
\mathcal{C}\ar[r]_f & \mathcal{D}
}
\end{gathered}
\end{align}
together with the identity $\chi_0:=1_{1_I}\colon 1_I\rightarrow 1_I$ with a monoidal structure in the sense of
\cite[Definition 3]{daystreet}. Before we construct the associated syllepsis, we again apply a semi-strictification procedure to
$(\chi,\chi_2,1)$. Therefore, we transport the monoidal structure $(\chi_2,1)$ along the natural 2-cell
\begin{align}\label{equ_braid=E_1_morphisms_back}
\xymatrix{
f(A)f(B)\ar@/^1pc/[rr]^{\chi_{A,B}}\ar@/_1pc/[rr]_{\chi_{A,B}}\ar@{}[rr]|{\chi_{2 A,I,I,B}} & & f(AB)
}
\end{align}
to define a new natural 2-cell $\chi_2\sprime$ given by the conjugation
\begin{align}\label{diag_braid=E_1_morphisms_back2}
\begin{gathered}
\xymatrix{
f(A)f(B)f(C)f(D)\ar@{}[d]|{\underset{\chi_{2 A,I,I,B}\otimes\chi_{2 C,I,I,D}}{\Rightarrow}}\ar@/^4pc/[d]^{\chi_{A,B}\otimes \chi_{C,D}}\ar@/_4pc/[d]_{\chi_{A,B}\otimes \chi_{C,D}}\ar[rrr]^{1_{f(A)}\otimes\rho_{f(B),f(C)}\otimes 1_{f(D)}} & & & f(A)f(C)f(B)f(D)\ar[rr]^(.55){\chi_{A,C}\otimes \chi_{B,D}}\ar@{}[d]|{\underset{\chi_{2 A,B,C,D}}{\Rightarrow}} & & f(AC)f(BD)\ar@/_3pc/[d]_{\chi_{AC,BD}}\ar@/^3pc/[d]^{\chi_{AC,BD}}\ar@{}[d]|{\underset{\chi_{2 AC,I,I,BD}^{-1}}{\Rightarrow}} \\
f(AB)f(CD)\ar[rrr]_{\chi_{AB,CD}} & & & f(ABCD)\ar[rr]_{f(1_A\otimes\rho_{B,C}\otimes 1_D)} & & f(ACBD).
}
\end{gathered}
\end{align}
Then $(\chi,\chi_2\sprime,1)$ defines a monoidal pseudo-natural equivalence, and (\ref{equ_braid=E_1_morphisms_back}) defines an invertible 
monoidal modification in the sense of \cite[Definition 3]{daystreet} between the monoidal pseudo-natural equivalences $(\chi,\chi_2,1)$ and
$(\chi,\chi_2\sprime,1)$ by construction. By Lemma~\ref{lemma_halgtrans}, the latter can hence be extended (uniquely) to a 3-cell in
$\mathrm{Alg}_{\mathsf{E}_1}(\hotwocat)$.

Now, whenever $A=D=I$, Diagram~(\ref{diag_braid=E_1_morphisms_back2}) is a natural 2-cell
\[\xymatrix{
f(B)f(C)\ar[rr]^{\rho_{f(B),f(C)}}\ar[d]_{\chi_{B,C}} & \ar@{}[d]|{\underset{\chi_2(I,A,B,I)}{\Rightarrow}} & f(C)f(B)\ar[d]^{\chi_{C,B}} \\
f(BC)\ar[rr]_{f(\rho_{B,C})} & & f(CB).
}\]
We claim that $u_{A,B}:=\chi_2(I,A,B,I)$ is a braiding on $F$. To prove the claim we are left to show the two due equations. Therefore, in 
Lemma~\ref{lemma_braid=E_1_morphisms_forthback} we will again see that the 3-cell of the form (\ref{diag_braid=E_1_morphisms_back0}) 
induced from the tuple $(F,u)$ by way of Lemma~\ref{lemma_braid=E_1_morphisms_forth} (that is, the pasted 3-cell
$\chi^{u}_2$ associated to $u$ given on \cite[p.125]{daystreet}) is exactly the natural 2-cell 
(\ref{diag_braid=E_1_morphisms_back2}). That means, $(\chi,\chi^{u}_2,1)$ is a monoidal pseudo-natural equivalence. 
According to \cite[p.125]{daystreet} it follows that $u$ is a braiding on $F$.
\end{proof}

\begin{lemma}\label{lemma_E_2ext_hset}
Let $(\mathcal{C},\beta)$ and $(\mathcal{D},\beta)$ be braided Gray monoids, and let $f\colon\mathcal{C}\rightarrow\mathcal{D}$ be a monoidal 
2-functor of underlying Gray monoids. Then the space
\begin{align}\label{equ_E_2ext_hset}
\mathrm{Alg}_{\mathsf{E}_2}(2\text{-Cat}^{\times})(\mathcal{C},\mathcal{D})\times_{\mathrm{Alg}_{\mathsf{E}_1}(2\text{-Cat}^{\times})(\mathcal{C},\mathcal{D})}\{f\}.
\end{align}
of $\mathsf{E}_2$-monoidal extensions of $f$ is an $h$-set.
\end{lemma}
\begin{proof}
This follows directly from Corollary~\ref{cor_truncmor} for $n=3$ and $k=1$.
\end{proof}

\begin{lemma}\label{lemma_braid=E_1_morphisms_forthback}
The two constructions of Lemma~\ref{lemma_braid=E_1_morphisms_forth} and Lemma~\ref{lemma_braid=E_1_morphisms_back} are mutually inverse 
functions of h-sets.
\end{lemma}
\begin{proof}
We first can reduce to cofibrant Gray monoids $\mathcal{C}$ and $\mathcal{D}$ and strictly unital monoidal $2$-functors
$F\colon\mathcal{C}\rightarrow\mathcal{D}$. If we start with a braiding $u$ on $F$, then a straightforward computation shows that
$\chi_2^{u}(I,A,B,I)=u(A,B)$. Here one uses that $\chi_{-,I}=\chi_{I,-}=1$ by assumption, and that further
$\omega_{-,I,-}=\omega_{I,-,-}=\omega_{-,-,I}$. 
This proves one direction.

For the other direction, let $f^+\colon(\mathcal{C},\beta)\rightarrow(\mathcal{D},\beta)$ be a morphism of associated $\mathsf{E}_2$-monoids 
extending $F=(f,\chi,\omega,1,1,1)$. Let  $(\chi_2,1)$ be the underlying monoidal structure on the pseudo-natural equivalence $\chi$. It 
follows that
\begin{itemize}
\item $\chi_{2 I,I,A,B}=\chi_{2 A,B,I,I}=\chi_{2 A,I,B,I}=\chi_{2 I,A,I,B}=1$ by the vertical and horizontal unitality laws.
\end{itemize}
By way of the replacement of $\chi_2$ by $\chi_2\sprime$ in the proof of Lemma~\ref{lemma_braid=E_1_morphisms_back} we can further without loss of generality assume that
\begin{itemize}
\item $\chi_{2 A,I,I,B}=1$.
\end{itemize}
The fact that $\chi_2$ is part of a monoidal structure on the pseudo-natural equivalence $\chi$ in (\ref{equ_braid=E_1_morphisms_back0}) 
gives rise to the following equation of natural 2-cells:
\begin{align*}
& \adjustbox{scale=0.7}{%
\xymatrix{
 & & f(A)f(B)f(CE)f(DF)\ar[dr] & & \\
 & f(A)f(B)f(C)f(E)f(D)f(F)\ar[ur]\ar[dr]\ar@{}[rr]|{\Uparrow 1_{f(A)}\otimes \rho^{-1}_{1_{f(B)},\chi_{C,E}}\otimes 1_{\chi_{D,F}}} & & f(A)f(CE)f(B)f(DF)\ar[dr] & \\
f(A)f(B)f(C)f(D)f(E)f(F)\ar[ur]\ar[dr]\ar[ddd]\ar@{}[rr]|{\Uparrow 1_{f(A)}\otimes\bar{\omega}_{f(B),f(C),f(D),f(E)}\otimes 1_{f(E)f(F)}} & & f(A)f(C)f(E)f(B)f(D)f(F)\ar[dr]\ar[ur]\ar@{}[rr]|{\Uparrow \omega_{A,C,E}\otimes\omega_{B,D,F}} & & f(ACE)f(BDF)\ar[ddd] \\
 & f(A)f(C)f(B)f(D)f(E)f(F)\ar[ur]\ar[dr]\ar@{}[rr]|{\Uparrow 1_{\chi_{A,C}}\otimes\rho_{\chi_{B,D},1_{f(E)}}\otimes 1_{f(F)}} & & f(AC)f(E)f(BD)f(F)\ar[ur] & \\
 & & f(AC)f(BD)f(E)f(F)\ar[ur]\ar[d] & & \\
f(AB)f(CD)f(EF)\ar[r]\ar@{}[uurr]|{\underset{(1\otimes\chi_{E,F})\ast(\chi_{2 A,B,C,D}\otimes 1_{f(E)f(F)})}{\Rightarrow}} & f(ABCD)f(EF)\ar[r] & f(ACBD)f(EF)\ar[r] & f(ACBDEF)\ar[r] & f(ACEBDF)\ar@{}[uull]|{\underset{\chi_{2 AC,BD,E,F}}{\Rightarrow}}\\
 & & \rotatebox{90}{$\scriptstyle = $} & & \\
f(A)f(B)f(C)f(D)f(E)f(F)\ar[r]\ar[ddd]\ar@{}[ddrr]|{\underset{(1_{f(A)f(B)}\otimes\chi_{2 C,D,E,F})\ast(\chi_{A,B}\otimes 1)}{\Rightarrow}} & f(A)f(B)f(C)f(E)f(D)f(F)\ar[r] & f(A)f(B)f(CE)f(DF)\ar[r]\ar[d]\ar@{}[ddrr]|{\underset{\chi_{2 A,B,CE,DF}}{\Rightarrow}} & f(A)f(CE)f(B)f(DF)\ar[r] & f(ACE)f(BDF)\ar[ddd] \\
 & & f(AB)f(CEDF)\ar[dr] & & \\
 & f(AB)f(CDEF)\ar[ur]\ar[dr]\ar@{}[rr]|{\Uparrow \chi^{-1}_{1_{AB},1_{C}\otimes\rho_{D,E}\otimes 1_{F}}} & & f(ABCEDF)\ar[dr] & \\
f(AB)f(CD)f(EF)\ar[ur]\ar[dr]\ar@{}[rr]|{\Uparrow \omega_{A,B,C,D,E,F}} & & f(ABCDEF)\ar[dr]\ar[ur]\ar@{}[rr]|{\Uparrow f(1_A\otimes\bar{\omega}_{B,C,D,E}\otimes 1_F)} & & f(ACEBDF) \\
 & f(ABCD)f(EF)\ar[ur]\ar[dr]\ar@{}[rr]|{\Uparrow \chi_{1_A\otimes\rho_{B,C}\otimes 1_D,1_{EF}}} & & f(ACBDEF)\ar[ur] & \\
 & & f(ACBD)f(EF)\ar[ur] & & \\
}} 
\end{align*}
In particular, we obtain
\begin{itemize}
\item $\chi_{2 A,D,E,F}=\omega_{A,D,EF}\ast(1_A\otimes\chi_{2 I,D,E,F})\ast\chi_{2 A,I,E,DF}$ (with $B=C=1$), and so
\item $\chi_{2 A,D,I,F}=\omega_{A,D,F}$;
\item $\chi_{2 A,B,C,F}=\omega^{-1}_{AB,C,F}\ast(\chi_{2 A,B,C,1}\otimes 1_F)\ast\chi_{2 AC,B,I,F}$, and so
\item $\chi_{2 A,I,C,F}=\omega^{-1}_{A,C,F}$.
\end{itemize}
Putting these together yields
\[\chi_{2 A,B,C,D}=\omega_{A,B,CD}\ast(1_A\otimes\omega^{-1}_{B,C,D})\ast(1_A\otimes\chi_{2 I,B,C,I}\otimes 1_D)\ast(1_A\otimes\omega_{C,B,D})\ast\omega^{-1}_{A,C,BD}\]
up to the two non-trivial instances of the naturality squares associated to the pseudo-natural equivalence $\chi$ that we have left 
implicit. This is precisely the cell associated to $u:=\chi_2(I,-,-,I)$ on \cite[p.125]{daystreet}.

That means, the two morphisms $f^{+}$ and $(F,u)$ of $\mathsf{E}_2$-monoids are extensions of the same 3-cell 
(\ref{diag_braid=E_1_morphisms_back0}) in $\mathrm{Alg}_{\mathsf{E}_1}(\hotwocat)$. It follows that the two morphisms
$f^{+},(F,u)\colon(\mathcal{C},\beta)\rightarrow(\mathcal{D},\beta)$ of $\mathsf{E}_2$-monoids are equivalent by way of 
Proposition~\ref{prop_inisegloc}.
\end{proof}

This proves Proposition~\ref{prop_braid=E_1_functors}.

\begin{remark}
Let $\mathcal{C}$ be a Gray monoid. The fiber
\[\mathrm{Alg}_{\mathsf{E}_2}(2\text{-Cat}^{\times})\times_{\mathrm{Alg}_{\mathsf{E}_1}(2\text{-Cat})^{\times}}\{\mathcal{C}\}\]
is an h-groupoid by Corollary~\ref{cor_truncobj}. With some more work one can extract from the above that this groupoid is equivalent to the
groupoid of braidings on $\mathcal{C}$ as given by way of Definition~\ref{def_equiv_braids} as well.
\end{remark}

\subsection{Syllepses and $\mathsf{E}_3$-structures}\label{sec_sub_translation3}

In this section we prove the following characterisation of a syllepsis on a braided Gray monoid.

\begin{theorem}\label{thm_syll=E_2}
Let $(\mathcal{C},\beta)$ be a braided Gray monoid and $\zeta(\mathcal{C},\beta)$ be the set of syllepses on $(\mathcal{C},\beta)$. Then 
there is an equivalence
\[\zeta(\mathcal{C},\beta)\simeq\mathrm{Alg}_{\mathsf{E}_3}(2\text{-Cat}^{\times})\times_{\mathrm{Alg}_{\mathsf{E}_2}(2\text{-Cat}^{\times})}\{(\mathcal{C},\beta)\}\]
of spaces. In particular, the latter is an h-set. This is to say, every syllepsis on $(\mathcal{C},\beta)$ induces an essentially unique 
$\mathsf{E}_1$-monoid structure on $(\mathcal{C},\beta)$ in $\mathrm{Alg}_{\mathsf{E}_2}(\mathrm{Ho}_{\infty}(2\text{-Cat}))^{\times}$, and 
vice versa.
\end{theorem}
\begin{proof}
Let $(\mathcal{C},\beta)$ be a braided Gray monoid, and let $m\colon\mathcal{C}\otimes_{\mathrm{Gr}}\mathcal{C}\rightarrow\mathcal{C}$ be 
equipped with the associated monoidal structure $(\chi_{\beta},\omega_{\beta})$ from Remark~\ref{rem_braidtomono}. Without loss of 
generality the Gray monoid $\mathcal{C}$ is cofibrant. By \cite[Section 5]{daystreet}, the set $\zeta(\mathcal{C},\beta)$ is bijective to 
the set $\beta(m)$ of braidings on $(m,\chi_{\beta},\omega_{\beta})$. By way of Proposition~\ref{prop_braid=E_1_functors} the latter is 
equivalent to the h-set
\begin{align}\label{equ_thm_syll=E_2_1}
\mathrm{Alg}_{\mathsf{E}_2}(2\text{-Cat}^{\times})(\mathcal{C}\times\mathcal{C},\mathcal{C})\times_{\mathrm{Alg}_{\mathsf{E}_1}(2\text{-Cat}^{\times})(\mathcal{C}\times\mathcal{C},\mathcal{C})}\{(m,\chi_{\beta},\omega_{\beta})\}
\end{align}
of $\mathsf{E}_2$-monoid extensions of $(m,\chi_{\beta},\omega_{\beta})$. We show that this h-set is equivalent to the space
\begin{align}\label{equ_thm_syll=E_2_2}
\mathrm{Alg}_{\mathsf{E}_3}(2\text{-Cat}^{\times})\times_{\mathrm{Alg}_{\mathsf{E}_2}(2\text{-Cat})^{\times}}\{(\mathcal{C},\beta)\}.
\end{align}
Therefore, we first note that the latter space is also an h-set by Corollary~\ref{cor_truncobj} indeed. 
Now, for one, every $\mathsf{E}_1$-monoid structure on $(\mathcal{C},\beta)$ with underlying diagram
\[\xymatrix{
(\mathcal{C},\beta) & (\mathcal{C},\beta)^{2}\ar[l]|{m\sprime} & (\mathcal{C},\beta)^{3}\ar@<.5ex>[l]^{m\sprime\times 1}\ar@<-.5ex>[l]_{1\times m\sprime} & (\mathcal{C},\beta)^{4}\ar@<-.5ex>@/_1pc/[l]_{m\times 1\times 1}\ar[l]|{}\ar@<.5ex>@/^1pc/[l]^{1\times 1\times m\sprime} 
}\]
can be considered as an $\mathsf{E}_2$-monoid structure on $\mathcal{C}$ in
$\mathrm{Alg}_{\mathsf{E}_1}(\hotwocat)$ (with underlying $\mathsf{E}_1$-monoid $(\mathcal{C},\beta)$). By Proposition~\ref{prop_eh} 
applied to this $\mathsf{E}_2$-monoid in $\mathrm{Alg}_{\mathsf{E}_1}(\hotwocat)$ it follows (in particular) that
$m\sprime\simeq(m,\chi_{\beta},\omega_{\beta})$ as monoidal 2-functors. The morphism $m\sprime$ is equipped with an extension to a morphism 
of $\mathsf{E}_2$-monoids by assumption, and hence so is $(m,\chi_{\beta},\omega_{\beta})$.

Vice versa, suppose we are given an $\mathsf{E}_2$-monoidal extension of the monoidal 2-functor $(m,\chi_{\beta},\omega_{\beta})$, or 
equivalently, a syllepsis $v$ on $(\mathcal{C},\beta)$.
We want to extend this to an $\mathsf{E}_3$-monoid structure on $\mathcal{C}$. To motivate the construction, we can think of an
$\mathsf{E}_3$-monoid structure $(\mathcal{C},\beta)^{v}\colon\mathsf{E}_1\otimes\mathsf{E}_1\otimes\mathsf{E}_1\rightarrow\mathcal{C}$ 
extending the $\mathsf{E}_2$-monoid $(\mathcal{C},\beta)$ as a 3-dimensional (infinitely long) cube-shaped diagram whose three boundary 
edges
$(\mathcal{C},\beta)^{v}|_{\mathsf{E}_1\otimes\{1\}\otimes\{1\}}$,
$(\mathcal{C},\beta)^{v}|_{\{1\}\otimes\mathsf{E}_1\otimes\{1\}}$ and 
$(\mathcal{C},\beta)^{v}|_{\{1\}\otimes\{1\}\otimes\mathsf{E}_1}$ are each given by the $\mathsf{E}_1$-monoid
$\mathcal{C}$, and whose three boundary surfaces 
$(\mathcal{C},\beta)^{v}|_{\mathsf{E}_1\otimes\mathsf{E}_1\otimes\{1\}}$,
$(\mathcal{C},\beta)^{v}|_{\mathsf{E}_1\otimes\{1\}\otimes\mathsf{E}_1}$, and
$(\mathcal{C},\beta)^{v}|_{\{1\}\otimes\mathsf{E}_1\otimes\mathsf{E}_1}$
are each given by the $\mathsf{E}_1$-monoid $(\mathcal{C},\beta)$ (via Proposition~\ref{prop_eh}). The assumption determines the boundary
of this cube together with a ``filler'' of the block
$(\mathcal{C},\beta)^{v}|_{\mathsf{A}_2^{\mathrm{nu}}\otimes\mathsf{E}_1\otimes\mathsf{E}_1}$. In the following we want to extend this 
block to a filler of the entire boundary $(\mathcal{C},\beta)^{v}|_{L_3(\mathsf{E}_\bullet)}$. We do so in essentially the same way as we 
did in Lemma~\ref{lemma_braid=E_1_forth}: We first suitably reverse the block
$(\mathcal{C},\beta)^{v}|_{\mathsf{A}^{\mathrm{nu}}_2\otimes\mathsf{E}_1\otimes\mathsf{E}_1}$ given by the assumed $\mathsf{E}_2$-monoid 
extension of the monoidal 2-functor $(m,\chi_{\beta},\omega_{\beta})$, then reflect it along a suitable diagonal, and third 
show that the $(\mathcal{C},\beta)^{v}|_{\mathsf{E}_1\otimes\mathsf{A}^{\mathrm{nu}}_2\otimes\mathsf{E}_1}$-block thus obtained coincides 
with the original one on their common fragment
$(\mathcal{C},\beta)^{v}|_{\mathsf{A}^{\mathrm{nu}}_2\otimes\mathsf{A}^{\mathrm{nu}}_2\otimes\mathsf{E}_1}$. This yields an initial segment 
in the sense of Section~\ref{sec_sub_iniseg} that extends to a full $\mathsf{E}_3$-monoid structure by way of 
Proposition~\ref{prop_inisegfib} and Lemma~\ref{lemma_unitff}.

More precisely, the $\mathsf{E}_2$-monoid extension of the monoidal 2-functor $(m,\chi_{\beta},\omega_{\beta})$ defines a 
monoidal structure $\chi_2$ (and without loss of generality $\chi_0=1$) on the pseudo-natural equivalence
\begin{align}\label{diag_syll=E_2_1}
\begin{gathered}
\xymatrix{
\mathcal{C}^{\otimes_{\mathrm{Gr}}2}\ar[d]_{(m,\chi_{\beta},\omega_{\beta})} & \mathcal{C}^{\otimes_{\mathrm{Gr}}4}\ar[d]^{(m,\chi_{\beta},\omega_{\beta})^{\otimes_{\mathrm{Gr}}2}}\ar[l]_{(m^{\otimes},\chi_{\beta}^{\otimes},\omega_{\beta}^{\otimes})} \\
\mathcal{C} & \mathcal{C}^{\otimes_{\mathrm{Gr}}2}\ar[l]^{(m,\chi_{\beta},\omega_{\beta})}\ar@{}[ul]|{\Uparrow\chi_{\beta}}
}
\end{gathered}
\end{align}
that extends to a non-unital $\mathsf{A}_2$-monoid structure on $(\mathcal{C},\beta)$ in $\mathrm{Alg}_{\mathsf{E}_2}(\hotwocat)$. This
$\mathsf{A}_2$-monoid structure is given by a 3-cell
\begin{align}\label{diag_syll=E_2_2}
\begin{gathered}
\xymatrix{
& \mathcal{C}^4\ar[dl]_{m^{\otimes}}\ar[dd]|(.3){m^2}|\hole & & \mathcal{C}^8\ar[ll]_{m^4}\ar[dd]^{(m^{\otimes},m^{\otimes})}\ar[dl]|{(m^{\otimes})^{\otimes}} \\
\mathcal{C}^2\ar[dd]_m  & & \mathcal{C}^4\ar[ll]_(.3){m^2}\ar[dd]^(.3){m^{\otimes}}  & \\
& \mathcal{C}^2\ar[dl]_{m} & & \mathcal{C}^4\ar[dl]^{m^{\otimes}}\ar[ll]_(.7){m^2}|\hole \\
\mathcal{C} & & \mathcal{C}^2\ar[ll]^{m} &
}
\end{gathered}
\end{align}
in $2\text{-Cat}^{\times}$ that extends to the back and to the top in essentially uniquely fashion (these two extensions yield ``the 
block'' referred to by $(\mathcal{C},\beta)^{v}|_{\mathsf{A}_2^{\mathrm{nu}}\otimes\mathsf{E}_1\otimes\mathsf{E}_1}$ above). Its left, 
front and bottom face are all given by $\chi_{\beta}$, the opposite faces are a corresponding product. As a modification in $\mathcal{C}$ 
it is given by a natural 2-cell
\[\xymatrix{
A_1A_2B_1B_2C_1C_2D_1D_2\ar[rr]^{\chi_{\beta}(\vec{A},\vec{B})\otimes\chi_{\beta}(\vec{C},\vec{D})}\ar[d]|{\chi_{\beta}(A_1A_2,B_1B_2,C_1C_2,D_1D_2)} & & A_1B_1A_2B_2C_1D_1C_2D_2\ar[rr]^{\chi_{\beta}(\vec{A}\otimes\vec{B},\vec{C}\otimes\vec{D})}\ar@{}[d]|{\Downarrow\chi_2} & & A_1B_1C_1D_1A_2B_2C_2D_2\ar[d]|{\chi_{\beta}(A_1,B_1,C_1,D_1)\otimes\chi_{\beta}(A_2,B_2,C_2,D_2)}\\
A_1A_2C_1C_2B_1B_2D_1D_2\ar[rr]_{\chi_{\beta}(\vec{A},\vec{C})\otimes\chi_{\beta}(\vec{B},\vec{D})}   & & A_1C_1A_2C_2B_1D_1B_2D_2\ar[rr]_{\chi_{\beta}(\vec{A}\otimes\vec{C},\vec{B}\otimes\vec{D})}  & & A_1C_1B_1D_1A_2C_2B_2D_2,
}\]
or in terms of (\ref{diag_syll=E_2_2}),
\begin{align}\label{diag_syll=E_2_2_formula}
\begin{gathered}
\xymatrix{
A_1A_2B_1B_2C_1C_2D_1D_2\ar[rr]^{\text{back}}\ar[d]_{\text{left}} & & A_1B_1A_2B_2C_1D_1C_2D_2\ar[rr]^{\text{bottom}}\ar@{}[d]|{\Downarrow\chi_2} & & A_1B_1C_1D_1A_2B_2C_2D_2\ar[d]^{\text{right}}\\
A_1A_2C_1C_2B_1B_2D_1D_2\ar[rr]_{\text{top}} & & A_1C_1A_2C_2B_1D_1B_2D_2\ar[rr]_{\text{front}}  & & A_1C_1B_1D_1A_2C_2B_2D_2.
}
\end{gathered}
\end{align}
To extend this to an $\mathsf{E}_1$-monoid structure on $(\mathcal{C},\beta)$ in $\mathrm{Alg}_{\mathsf{E}_2}(\hotwocat)$ --- and hence to 
an element of the h-set 
(\ref{equ_thm_syll=E_2_2}) --- it suffices to extend it to an element of $I^{(2)}_{(2,1,2)}(2\text{-Cat}^{\times})$ by 
Proposition~\ref{prop_inisegfib} (and Lemma~\ref{lemma_unitff}). That means, we are given the
$(\mathsf{A}_2^{\mathrm{nu}}\otimes \mathsf{E}_1\otimes \mathsf{E}_1$)-algebra $\mathsf{A}_{(2,\infty,\infty)}:=(m,\chi_{\beta},\omega_{\beta},(\chi_{\beta})_2)$ in
$2\text{-Cat}^{\times}$, and are to construct an $(\mathsf{E}_1\otimes \mathsf{A}_2^{\mathrm{nu}}\otimes \mathsf{E}_1)$-algebra $\mathsf{A}_{(\infty,2,\infty)}$ such that 
\begin{enumerate}
\item $\mathsf{A}_{(2,\infty,\infty)}=\mathsf{A}_{(\infty,2,\infty)}$ on $\mathsf{A}_2^{\mathrm{nu}}\otimes \mathsf{A}_2^{\mathrm{nu}}\otimes \mathsf{E}_1^{\mathrm{nu}}$, and such that
\item $\mathsf{A}_{(\infty,2,\infty)}=(m,\chi_{\beta},\omega_{\beta})$ on $\mathsf{E}_1^{\mathrm{nu}}\otimes \mathsf{E}_1^{\mathrm{nu}}\otimes \mathsf{E}_{0}^{\mathrm{nu}}$.\footnote{It would be enough to have matching equivalences in (1) and (2), but we can provide equalities.}
\end{enumerate}
We do so by constructing another $(\mathsf{A}_2^{\mathrm{nu}}\otimes \mathsf{E}_1\otimes \mathsf{E}_1$)-algebra $\mathsf{A}_{(2,\infty,\infty)}^{\mathrm{rev}}$ in
$2\text{-Cat}^{\times}$ whose (2-dimensional) diagonal reflection 
\begin{align}\label{equ_syll=E_2}
(-)^{\rho_{x_1}}\colon\mathrm{Alg}_{\mathsf{A}_2^{\mathrm{nu}}\otimes \mathsf{E}_1}(\mathrm{Alg}_{\mathsf{E}_1}(2\text{-Cat}^{\times}))\rightarrow\mathrm{Alg}_{\mathsf{E}_1\otimes \mathsf{A}_2^{\mathrm{nu}}}(\mathrm{Alg}_{\mathsf{E}_1}(2\text{-Cat}^{\times}))
\end{align}
computed in $\mathrm{Alg}_{\mathsf{E}_1}(2\text{-Cat}^{\times})$ will give the desired algebra $\mathsf{A}_{(\infty,2,\infty)}$ (Example~\ref{exple_diagrefl3}). Towards 
this end, the monoidal pseudo-natural equivalence (\ref{diag_syll=E_2_1}) may be inverted to yield a pseudo-natural equivalence 
\begin{align}\label{diag_syll=E_2_3}
\begin{gathered}
\xymatrix{
\mathcal{C}^{\otimes_{\mathrm{Gr}}2}\ar[d]_{(m,\chi_{\beta},\omega_{\beta})}\ar@{}[drr]|{\Downarrow\chi_{\beta}^{-1}} & & \mathcal{C}^{\otimes_{\mathrm{Gr}}4}\ar[d]^{(m,\chi_{\beta},\omega_{\beta})^2}\ar[ll]_{(m^{\otimes},\chi_{\beta}^{\otimes},\omega_{\beta}^{\otimes})} \\
\mathcal{C} & & \mathcal{C}^{\otimes_{\mathrm{Gr}}2}\ar[ll]^{(m,\chi_{\beta},\omega_{\beta})}
}
\end{gathered}
\end{align}
with monoidal structure $(\chi_{\beta})_2^{-1_h}$ by Lemma~\ref{lemma_mnt_inv}. In terms of the natural 2-cell (\ref{diag_syll=E_2_2_formula}) it computes the horizontal inverse when
inverting the ``front'' and ``back'' arrows. The isomorphism
$\mathcal{C}^{\sigma_{(\langle 2\rangle,\langle 2\rangle)}}:=1\otimes_{\mathrm{Gr}}\mathrm{swap}\otimes_{\mathrm{Gr}} 1\colon C^{\otimes_{\mathrm{Gr}^4}}\rightarrow C^{\otimes_{\mathrm{Gr}^4}}$
is Gray monoidal, and so we may pre-compose it with (\ref{diag_syll=E_2_3}) to obtain a pseudo-natural equivalence 
\begin{align}\label{diag_syll=E_2_4}
\begin{gathered}
\xymatrix{
\mathcal{C}^{\otimes_{\mathrm{Gr}}2}\ar[d]_{(m,\chi_{\beta},\omega_{\beta})}\ar@{}[drr]|{\Downarrow\chi_{\beta^{\mathrm{rev}}}} & & \mathcal{C}^{\otimes_{\mathrm{Gr}}4}\ar[d]^{(m^{\otimes},\chi_{\beta}^{\otimes},\omega_{\beta}^{\otimes})}\ar[ll]_{(m,\chi_{\beta},\omega_{\beta})^2} \\
\mathcal{C} & & \mathcal{C}^{\otimes_{\mathrm{Gr}}2}\ar[ll]^{(m,\chi_{\beta},\omega_{\beta})}
}
\end{gathered}
\end{align}
with monoidal structure given by $\mathcal{C}^{\sigma_{(\langle 2\rangle,\langle 2\rangle)}}\ast(\chi_{\beta})_2^{-1_h}$.
This forms an $\mathsf{A}_2^{\mathrm{nu}}\otimes \mathsf{A}_2^{\mathrm{nu}}\otimes \mathsf{E}_1$-monoid in $2\text{-Cat}^{\times}$. We claim 
that that this extends (essentially uniquely) to a $\mathsf{A}_2^{\mathrm{nu}}\otimes \mathsf{E}_1\otimes \mathsf{E}_1$-monoid, and hence to 
a morphism from $(\mathcal{C},\beta)^2$ to $(\mathcal{C},\beta)$ in $\mathrm{Alg}_{\mathsf{E}_2}(2\text{-Cat}^{\times})$.
If the claim holds, then diagonal reflection of this $(\mathsf{A}_2^{\mathrm{nu}}\otimes \mathsf{E}_1\otimes \mathsf{E}_1)$-algebra under 
(\ref{equ_syll=E_2}) gives an $(\mathsf{E}_1\otimes \mathsf{A}_2^{\mathrm{nu}}\otimes \mathsf{E}_1)$-algebra that satisfies (1) and (2) as 
desired. This finishes the proof.

To show the claim, by Lemma~\ref{lemma_syll_rev} we may construct a sylleptic Gray monoid
$(\mathcal{C},\beta^{\mathrm{rev}},v^{\mathrm{rev}})$. This corresponds to a braiding $u_{v^{\mathrm{rev}}}$ on the monoidal
2-functor $(m,\chi_{\beta^{\mathrm{rev}}},\omega_{\beta^{\mathrm{rev}}})\colon\mathcal{C}\times\mathcal{C}\rightarrow\mathcal{C}$ respective the 
braiding $\beta^{\mathrm{rev}}$ on $\mathcal{C}$. Subsequently we obtain a monoidal pseudo-natural equivalence
\begin{align}\label{diag_syll=E_2_5}
\begin{gathered}
\xymatrix{
\mathcal{C}^{\otimes_{\mathrm{Gr}}2}\ar[d]_{(m,\chi_{\beta^{\mathrm{rev}}},\omega_{\beta^{\mathrm{rev}}})} & & \mathcal{C}^{\otimes_{\mathrm{Gr}}4}\ar[d]^{(m,\chi_{\beta^{\mathrm{rev}}},\omega_{\beta^{\mathrm{rev}}})^2}\ar[ll]_{(m^{\otimes},\chi_{\beta^{\mathrm{rev}}}^{\otimes},\omega_{\beta^{\mathrm{rev}}}^{\otimes})} \\
\mathcal{C} & & \mathcal{C}^{\otimes_{\mathrm{Gr}}2}\ar[ll]^{(m,\chi_{\beta^{\mathrm{rev}}},\omega_{\beta^{\mathrm{rev}}})}.\ar@{}[ull]|{\Uparrow (\chi_{\beta^{\mathrm{rev}}},(\chi_{\beta^{\mathrm{rev}}})_2)}
}
\end{gathered}
\end{align}
Again by Lemma~\ref{lemma_mnt_inv} this monoidal pseudo-natural equivalence may be inverted and subsequently be pre-composed with the Gray 
monoidal isomorphism $\mathcal{C}^{\sigma_{(\langle 2\rangle,\langle 2\rangle)}}\colon C^{\otimes_{\mathrm{Gr}^4}}\rightarrow C^{\otimes_{\mathrm{Gr}^4}}$ to yield a monoidal pseudo-natural equivalence 
\begin{align}\label{diag_syll=E_2_6}
\begin{gathered}
\xymatrix{
\mathcal{C}^{\otimes_{\mathrm{Gr}}2}\ar[d]_{(m,\chi_{\beta^{\mathrm{rev}}},\omega_{\beta^{\mathrm{rev}}})}\ar@{}[drr]|{\Downarrow} & & \mathcal{C}^{\otimes_{\mathrm{Gr}}4}\ar[d]^{(m^{\otimes},\chi_{\beta^{\mathrm{rev}}}^{\otimes},\omega_{\beta^{\mathrm{rev}}}^{\otimes})}\ar[ll]_{(m,\chi_{\beta^{\mathrm{rev}}},\omega_{\beta^{\mathrm{rev}}})^2} \\
\mathcal{C} & & \mathcal{C}^{\otimes_{\mathrm{Gr}}2}\ar[ll]^{(m,\chi_{\beta^{\mathrm{rev}}},\omega_{\beta^{\mathrm{rev}}})}
}
\end{gathered}
\end{align}
given by $(\chi_{\beta},\mathcal{C}^{\sigma_{(\langle 2\rangle,\langle 2\rangle)}}\ast(\chi_{\beta^{\mathrm{rev}}})_2^{-1_h})$.
This defines an $\mathsf{A}_{2}^{\mathrm{nu}}\otimes \mathsf{A}_{2}^{\mathrm{nu}}$-algebra in
$\mathrm{Alg}_{\mathsf{E}_1}(2\text{-Cat}^{\times})$. Its diagonal reflection respective
\begin{align}\label{equ_syll=E_2_2}
(-)^{\rho_{x_3}}=\mathrm{Alg}_{\mathsf{A}_2^{\mathrm{nu}}}((-)^{\rho})\colon\mathrm{Alg}_{\mathsf{A}_2^{\mathrm{nu}}\otimes \mathsf{A}_2^{\mathrm{nu}}\otimes \mathsf{E}_1}(2\text{-Cat}^{\times}))\rightarrow\mathrm{Alg}_{\mathsf{A}_2^{\mathrm{nu}}\otimes \mathsf{E}_1\otimes \mathsf{A}_2^{\mathrm{nu}}}(2\text{-Cat}^{\times}))
\end{align}
gives rise to a $(\mathsf{A}_2^{\mathrm{nu}}\otimes \mathsf{E}_1\otimes \mathsf{A}_2^{\mathrm{nu}})$-monoid. By way of the formula for the 
braiding on $(m,\chi_{\beta},\omega_{\beta})$ associated to the corresponding syllepsis $v$ given on \cite[p.129]{daystreet}, and 
subsequently by the formula for the monoidal structure on $\chi_{\beta}$ associated to this braiding given on \cite[p.125]{daystreet}, one 
computes that the underlying $(\mathsf{A}_2^{\mathrm{nu}}\otimes \mathsf{A}_2^{\mathrm{nu}}\otimes \mathsf{A}_2^{\mathrm{nu}})$-monoid of
the diagonal reflection (\ref{equ_syll=E_2_2}) of (\ref{diag_syll=E_2_6}) is precisely the 3-cell underlying 
the monoidal pseudo-natural equivalence~(\ref{diag_syll=E_2_4}). Together they hence form a morphism in
$I^{(1)}_{(3,2)}(2\text{-Cat}^{\times})$ from $(\mathcal{C},\beta)^2$ to $(\mathcal{C},\beta)$. By way of Proposition~\ref{prop_inisegloc} we 
obtain an (essentially unique) extension to a morphism of $\mathsf{E}_2$-monoids, which is to say we obtain an
$\mathsf{A}_2^{\mathrm{nu}}$-monoid in $\mathrm{Alg}_{\mathsf{E}_2}(2\text{-Cat}^{\times})$.

We have thus constructed a map from the h-set (\ref{equ_thm_syll=E_2_1}) to the h-set (\ref{equ_thm_syll=E_2_2}) and vice versa. It is 
straightforward to show that the composition
$(\ref{equ_thm_syll=E_2_1}) \rightarrow(\ref{equ_thm_syll=E_2_2})\rightarrow(\ref{equ_thm_syll=E_2_1})$ is the identity.
The other direction acts (essentially) by restriction of a given $\mathsf{E}_3$-monoid structure on $\mathcal{C}$ to its underlying
$\mathsf{E}_2$-monoid structure on $m$, and subsequent extension of that restriction. However, such extensions are essentially unique 
whenever they exist by way of Proposition~\ref{prop_inisegfib}.
\end{proof}

We further prove the following relativisation.

\begin{proposition}\label{prop_syll=E_2_functors}
Let $(\mathcal{C},\beta,v)$ and $(\mathcal{D},\beta,v)$ be sylleptic Gray monoids, and let
$(F,u)\colon(\mathcal{C},\beta)\rightarrow(\mathcal{D},\beta)$ be a braided monoidal 2-functor 
of underlying braided Gray monoids. Then the space
\begin{align}\label{equ_syll=E_2_functors} 
\mathrm{Alg}_{\mathsf{E}_3}(2\text{-Cat}^{\times})(\mathcal{C},\mathcal{D})\times_{\mathrm{Alg}_{\mathsf{E}_2}(2\text{-Cat}^{\times})(\mathcal{C},\mathcal{D})}\{(F,u)\}
\end{align}
is $(-1)$-truncated. It is non-empty if and only if $(F,u)$ is sylleptic.
\end{proposition}
\begin{proof}
We first note that the space is $(-1)$-truncated by Corollary~\ref{cor_truncobj}.
Now, let $F=(f,\chi,\omega)$ denote the monoidal structure underlying $(F,u)$. Let $\chi^{fm}$ and $\chi^{m(f,f)}$ denote the 
corresponding pseudo-natural equivalences associated to the induced monoidal structures on the compositions $fm$ and $m(f,f)$, 
respectively. Then the natural 2-cell $u$ induces a natural 2-cell
\[\xymatrix{
f(A)f(B)f(C)f(D)\ar[rr]^(.55){\chi^{m(f,f)}_{A,B,C,D}}\ar[d]_{\chi_{A,B}\otimes \chi_{C,D}}\ar@{}[drr]|{\underset{\chi_2}{\Rightarrow}} & & f(AC)f(BD)\ar[d]^{\chi_{AC,BD}}\\
f(AB)f(CD)\ar[rr]_{\chi^{fm}_{A,B,C,D}} & & f(ACBD)
}\]
such that $u$ is a braiding for $F$ if and only if $(\chi_2,1)$ is a monoidal structure on the pseudo-natural equivalence $\chi$ 
\cite[p.125]{daystreet}. One computes that the braided monoidal 2-functor $(F,u)$ is sylleptic if and only if the 
monoidal pseudo-natural equivalence $(\chi,\chi_2,1)$ is braided in the sense of \cite[Definition 14]{daystreet}. This in turn holds if and 
only if the modification $\chi_2$ is monoidal by \cite[p.126]{daystreet}.

By Lemma~\ref{lemma_halgtrans}, the monoidal 2-functor $F\colon\mathcal{C}\rightarrow\mathcal{D}$ extends essentially uniquely 
to a morphism $f\colon\mathcal{C}\rightarrow\mathcal{D}$ of associated $\mathsf{E}_1$-monoids in $\hotwocat$. Subsequently, in 
Proposition~\ref{prop_braid=E_1_functors}, we have seen that the braided monoidal 2-functor
$(F,u)\colon(\mathcal{C},\beta)\rightarrow(\mathcal{C},\beta)$ extends essentially uniquely to a morphism $f^{u}$ of 
associated $\mathsf{E}_2$-monoids in $\hotwocat$ by way of the associated monoidal pseudo-natural equivalence $(\chi,\chi_2,1)$.
Let $m_{\beta}\colon\mathcal{C}\times\mathcal{C}\rightarrow\mathcal{C}$ and
$m_{\beta}\colon\mathcal{D}\times\mathcal{D}\rightarrow\mathcal{D}$ denote the two multiplications of $\mathcal{C}$ and $\mathcal{D}$ in
$\mathrm{Alg}_{\mathsf{E}_1}(\hotwocat)$ induced by their respective braidings $\beta$. Let $\chi^{v}_2$ denote the natural 2-cell 
induced by the syllepsis $v$ such that the pseudo-natural equivalence $(\chi_{\beta},(\chi^v)_2,1)$ becomes monoidal.
Then, by way of Proposition~\ref{prop_inisegloc}, the space (\ref{equ_syll=E_2_functors}) of extensions of $(F,u)$ to a morphism of
$\mathsf{E}_3$-monoids is equivalent to the space of morphisms from the diagram
\begin{align}
\begin{gathered} 
\xymatrix{
\mathcal{C}^{\otimes_{\mathrm{Gr}}3}\ar@<-.5ex>[d]_{m_{\beta}\otimes 1}\ar@<.5ex>[d]^{1\otimes m_{\beta}} & & & \\
\mathcal{C}^{\otimes_{\mathrm{Gr}}2}\ar[d]_{m_{\beta}} \ar[d]\ar@{}[dr]|{\chi^v_2} & (\mathcal{C}^{\otimes_{\mathrm{Gr}}2})^{\otimes_{\mathrm{Gr}} 2}\ar[d]\ar[l]_(.55){(m_{\beta}, m_{\beta})}\ar[d]^{m_{\beta}^{\otimes}} & \\
\mathcal{C} & \mathcal{C}^{\otimes_{\mathrm{Gr}}2}\ar[l]|{m_{\beta}} & \mathcal{C}^{\otimes_{\mathrm{Gr}}3}\ar@<.5ex>[l]^{m_{\beta}\otimes 1}\ar@<-.5ex>[l]_{1\otimes m_{\beta}}  
}
\end{gathered}
\end{align}
to the diagram
\begin{align}
\begin{gathered}
\xymatrix{
\mathcal{D}^{\otimes_{\mathrm{Gr}}3}\ar@<-.5ex>[d]_{m_{\beta}\otimes 1}\ar@<.5ex>[d]^{1\otimes m_{\beta}} & & & \\
\mathcal{D}^{\otimes_{\mathrm{Gr}}2}\ar[d]|{m_{\beta}} \ar[d]\ar@{}[dr]|{\chi^v_2} & (\mathcal{D}^{\otimes_{\mathrm{Gr}}2})^{\otimes_{\mathrm{Gr}} 2}\ar[d]\ar[l]_(.55){(m_{\beta}, m_{\beta})}\ar[d]^{m_{\beta}^{\otimes}} & \\
\mathcal{D} & \mathcal{D}^{\otimes_{\mathrm{Gr}}2}\ar[l]|{m_{\beta}} & \mathcal{D}^{\otimes_{\mathrm{Gr}}3}\ar@<.5ex>[l]^{m_{\beta}\otimes 1}\ar@<-.5ex>[l]_{1\otimes m_{\beta}} & 
}
\end{gathered}
\end{align}
in $\mathrm{Alg}_{\mathsf{E}_1}(\hotwocat)$ that on the left row and bottom column is given by $f^{u}$. The space of such extensions is hence the 
space of fillers of the cube
\begin{align}\label{diag_syll=E_2_functors}
\begin{gathered}
\xymatrix{
 & \mathcal{C}^{\otimes_{\mathrm{Gr}} 2}\ar[rr]^{f^{\otimes_{\mathrm{Gr}} 2}}\ar[dd]^(.3){m_{\beta}}|\hole & & \mathcal{D}^{\otimes_{\mathrm{Gr}} 2} \ar[dd]^{m_{\beta}} \\
(\mathcal{C}^{\otimes_{\mathrm{Gr}} 2})^{\otimes_{\mathrm{Gr}} 2}\ar[rr]^(.4){f^{\otimes_{\mathrm{Gr}} 4}}\ar[dd]_{m_{\beta}^{\otimes}}\ar[ur]^{(m_{\beta},m_{\beta})} & & (\mathcal{D}^{\otimes_{\mathrm{Gr}} 2})^{\otimes_{\mathrm{Gr}} 2}\ar[dd]^(.3){m_{\beta}^{\otimes}}\ar[ur]_{(m_{\beta},m_{\beta})} \\
 & \mathcal{C}\ar[rr]_(.3)f|\hole & & \mathcal{D} \\
\mathcal{C}^{\otimes_{\mathrm{Gr}} 2}\ar[rr]_{f^{\otimes_{\mathrm{Gr}} 2}}\ar[ur]^{m_{\beta}} & & \mathcal{D}^{\otimes_{\mathrm{Gr}} 2}\ar[ur]_{m_{\beta}} & \\
}
\end{gathered}
\end{align} 
in $\mathrm{Alg}_{\mathsf{E}_1}(\hotwocat)$, where the bottom and back face are given by the monoidal pseudo-natural transformation
$(\chi,\chi_2,1)$, the top and front face are given by the its corresponding products, and the left and right face are given by the two 
respective instances of the monoidal pseudo-natural transformation $(\chi_{\beta},\chi^v_2,1)$. 
Thus, a 3-cell of the form (\ref{diag_syll=E_2_functors}) corresponds under Lemma~\ref{lemma_halgtrans} precisely to a proof that its 
underlying modification --- given by its bottom face $\chi_2$ --- is monoidal. 
That means the fiber (\ref{equ_syll=E_2_functors}) is inhabited if and only if $(f,u)$ is sylleptic.
\end{proof}

\subsection{Symmetries and $\mathsf{E}_{\infty}$-structures}

In this section we prove the following characterisation of a symmetry on a sylleptic Gray monoid.

\begin{theorem}
Let $(\mathcal{C},\beta,v)$ be a sylleptic Gray monoid. Then the space
\[\mathrm{Alg}_{\mathsf{E}_{\infty}}(2\text{-Cat}^{\times})\times_{\mathrm{Alg}_{\mathsf{E}_3}(2\text{-Cat})^{\times}}\{(\mathcal{C},\beta,v)\}\]
is $(-1)$-truncated. It is inhabited if and only if $(\mathcal{C},\beta,v)$ is symmetric.
\end{theorem}
\begin{proof}
Let $(\mathcal{C},\beta,v)$ be a sylleptic Gray monoid, and let
$m\colon\mathcal{C}\otimes_{\mathrm{Gr}}\mathcal{C}\rightarrow\mathcal{C}$ be equipped with its associated braided monoidal structure
$u_v$ \cite[Section 5]{daystreet}. Without loss of generality the Gray monoid $\mathcal{C}$ is cofibrant. Again by
\cite[Section 5]{daystreet}, the syllepsis $v$ on $(\mathcal{C},\beta)$ is symmetric if and only if the braided monoidal 2-functor
$(m,u_v)$ is sylleptic. By way of Proposition~\ref{prop_syll=E_2_functors} the latter holds if and only if the h-proposition
\begin{align}\label{equ_thm_sym=E_3_1}
\mathrm{Alg}_{\mathsf{E}_3}(2\text{-Cat}^{\times})(\mathcal{C}\times\mathcal{C},\mathcal{C})\times_{\mathrm{Alg}_{\mathsf{E}_2}(2\text{-Cat}^{\times})(\mathcal{C}\times\mathcal{C},\mathcal{C})}\{(m,u_v)\}
\end{align}
of $\mathsf{E}_3$-monoid extensions of $(m,u_v)$ is inhabited. We are left to show that this h-proposition is equivalent to the space
\begin{align}\label{equ_thm_sym=E_3_2}
\mathrm{Alg}_{\mathsf{E}_4}(2\text{-Cat}^{\times})\times_{\mathrm{Alg}_{\mathsf{E}_3}(2\text{-Cat})^{\times}}\{(\mathcal{C},\beta,v)\}.
\end{align}
The latter fiber is an h-proposition again by Corollary~\ref{cor_truncobj}.
The implication $(\ref{equ_thm_sym=E_3_1})\rightarrow(\ref{equ_thm_sym=E_3_2})$ is given by the fact that any point in
$(\ref{equ_thm_sym=E_3_1})$ is a non-unital $\mathsf{A}_2$-monoid structure on $(\mathcal{C},\beta,v)$ in
$\mathrm{Alg}_{\mathsf{E}_3}(\hotwocat)$. By way of Corollary~\ref{cor_E_latching} this gives an element of
$I^{(3}_{(1,1,1,2)}(2\text{-Cat}^{\times})$. The forgetful functors
\[\mathrm{Alg}_{\mathsf{A}_{m+1}}^{\mathrm{nu}}(\mathrm{Alg}_{\mathsf{E}_3}(\mathcal{C}^{\otimes}))\rightarrow I^{(3)}_{(1,1,1,m)}(2\text{-Cat}^{\times})\]
are equivalences for all $m\geq 2$ by Lemma~\ref{lemma_iniseg_qu}. We 
hence obtain an extension of any given point in $(\ref{equ_thm_sym=E_3_1})$ to $(\ref{equ_thm_sym=E_3_2})$ (Lemma~\ref{lemma_unitff} again 
takes care of unitality).
The other direction follows directly from Proposition~\ref{prop_eh}. Lastly, the projection
\[\mathrm{Alg}_{\mathsf{E}_{\infty}}(2\text{-Cat}^{\times})\rightarrow\mathrm{Alg}_{\mathsf{E}_4}(2\text{-Cat}^{\times})\]
is a trivial fibration by the Baez--Dolan--Lurie Stabilization Theorem \cite[Corollary 5.1.1.7]{lurieha}.
\end{proof}

\begin{proposition}
Let $(\mathcal{C},\beta,v)$ and $(\mathcal{D},\beta,v)$ be symmetric Gray monoids, and let
$(F,u)\colon(\mathcal{C},\beta,v)\rightarrow(\mathcal{D},\beta,v)$ be a sylleptic monoidal 2-functor of underlying sylleptic 
braided Gray monoids. Then the space
\[\mathrm{Alg}_{\mathsf{E}_{\infty}}(2\text{-Cat}^{\times})(\mathcal{C},\mathcal{D})\times_{\mathrm{Alg}_{\mathsf{E}_3}(2\text{-Cat}^{\times})(\mathcal{C},\mathcal{D})}\{(F,u)\}.\]
is contractible.
\end{proposition}
\begin{proof}
Immediate by Corollary~\ref{cor_truncmor}.
\end{proof}

\section*{Conclusive remarks}

The reader may note that the proofs regarding the equivalence of fully algebraic structures and their corresponding homotopical structures on 
a bicategory become less long and less complex with increasing complexity of the structures they are about. This is because at each 
stage we fix a low dimensional structure first and only let the higher dimensional structure vary over it fiberwise. The collections of these 
higher dimensional structures over a \emph{fixed} lower dimensional algebraic base become more and more truncated and as such become easier 
to manage. This is directly reflected by the underlying fact that the complexity of the definitions of Gray monoidal structures over
2-categories, braidings over Gray monoidal structures, syllepses over braidings, and symmetries over syllepses become less long and complex 
with each step. This is in contrast to the approach taken in \cite{gurski_braid} and \cite{gurskiosorno_sym}, which essentially considers 
$\mathsf{E}_n$-monoids in the 3-category of bicategories directly. This adds a new dimension at each step rather than to take one away, and
the geometry of the involved configuration spaces becomes more complicated rather than simpler.

\bibliographystyle{amsalpha}
\bibliography{BSBib}

\newcommand{\noopsort}[1]{}
\providecommand{\bysame}{\leavevmode\hbox to3em{\hrulefill}\thinspace}
\providecommand{\MR}{\relax\ifhmode\unskip\space\fi MR }
\providecommand{\MRhref}[2]{%
  \href{http://www.ams.org/mathscinet-getitem?mr=#1}{#2}
}
\providecommand{\href}[2]{#2}
\begin{thebibliography}{GFH24}

\bibitem[BD98]{baezdolan_cat}
J.~Baez and J.~Dolan, \emph{Categorification}, Contemp. Math. \textbf{230}
  (1998), 1–36.

\bibitem[BG17]{bourkegurski}
J.~Bourke and N.~Gurski, \emph{The {G}ray tensor product via factorisation},
  Applied Categorical Structures (2017).

\bibitem[BN96]{baezneuchl}
J.~Baez and M.~Neuchl, \emph{Higher-dimensional algebra {I}: {B}raided monoidal
  2-categories}, Adv. in Math. \textbf{121} (1996), 196--244.

\bibitem[Cra98]{crans_centers}
S.~E. Crans, \emph{Generalized centers of braided and sylleptic monoidal
  2-categories}, Adv. in Math. \textbf{136} (1998), 183--223.

\bibitem[DS97]{daystreet}
B.~Day and R.~Street, \emph{Monoidal bicategories and hopf algebroids},
  Advances in Mathematics \textbf{129} (1997), 99--157.

\bibitem[FN62]{nf_config}
E.~Fadell and L.~Neuwirth, \emph{Configuration spaces}, Mathematica
  Scandinavica \textbf{10} (1962), 111--118.

\bibitem[Gar10]{garner_cofmndthy}
R.~Garner, \emph{Homomorphisms of higher categories}, {A}dvances in
  {M}athematics \textbf{224} (2010), 2269--2311.

\bibitem[GFH24]{fioregambinohyland}
N.~Gambino, M.~Fiore, and M.~Hyland, \emph{Monoidal bicategories, differential
  linear logic, and analytic functors}, arXiv:2405.05774, 2024.

\bibitem[GO13]{gurskiosorno_sym}
N.~Gurski and A.~M. Osorno, \emph{Infinite loop spaces, and coherence for
  symmetric monoidal bicategories}, Advances in Mathematics \textbf{246}
  (2013), 1--32.

\bibitem[GPS95]{gps_coherence}
R.~Gordon, A.~J. Power, and R.~Street, \emph{Coherence for tricategories},
  Memoirs of the AMS \textbf{117} (1995), no.~558, 1--85.

\bibitem[Gur11]{gurski_braid}
N.~Gurski, \emph{Loop spaces, and coherence for monoidal and braided monoidal
  bicategories}, Advances in Mathematics \textbf{226} (2011), 4225--4265.

\bibitem[Hau23]{haugseng_hanotes}
R.~Haugseng, \emph{An allegedly somewhat friendly introduction to
  $\infty$-operads}, \url{https://runegha.folk.ntnu.no/iopd.pdf}, 2023, Notes
  from a series of 7 lectures in Sevilla (October 2022).

\bibitem[JS93]{joyalstreet}
A.~Joyal and R.~Street, \emph{Braided tensor categories}, Advances in
  Mathematics \textbf{102} (1993), 20--78.

\bibitem[JY21]{johnsonyau}
N.~Johnson and D.~Yau, \emph{2-{D}imensional {C}ategories}, Oxford University
  Press, 2021.

\bibitem[KV94a]{kapranovvoevodskyI}
M.~Kapranov and V.~Voevodsky, \emph{2-{C}ategories and {Z}amolodchikov
  tetrahedra}, Proc. Symp. Pure Math. \textbf{56} (1994), 177--260.

\bibitem[KV94b]{kapranovvoevodskyII}
\bysame, \emph{Braided monoidal 2-categories and {M}anin--{S}chechtman higher
  braid groups}, J. Pure Appl. Algebra \textbf{92} (1994), 241--267.

\bibitem[Lac02]{lack2catms}
S.~Lack, \emph{A {Q}uillen model structure for 2-categories}, K-Theory
  \textbf{26} (2002), no.~2, 171--205.

\bibitem[Lac04]{lackbicatms}
\bysame, \emph{A {Q}uillen model structure for bicategories}, K-Theory
  \textbf{33} (2004), no.~3, 185–197.

\bibitem[Lac07]{lack_2bitricat}
\bysame, \emph{Bicat is not equivalent to {G}ray}, Theory and {A}pplications of
  {C}ategories \textbf{18} (2007), no.~21, 1--3.

\bibitem[Lac11]{lackgraycatms}
\bysame, \emph{A {Q}uillen model structure for {G}ray-categories}, Journal of
  {K}-Theory \textbf{8} (2011), no.~2, 183--221.

\bibitem[Lur09]{luriehtt}
J.~Lurie, \emph{Higher topos theory}, Annals of Mathematics Studies, no. 170,
  Princeton University Press, 2009.

\bibitem[Lur17]{lurieha}
\bysame, \emph{Higher algebra},
  \url{http://www.math.harvard.edu/~lurie/papers/HA.pdf}, 2017, Last update 18
  September 2017.

\bibitem[May72]{may_loop}
J.~P. May, \emph{The geometry of iterated loop spaces}, Lecture Notes in
  Mathematics, no. 271, Springer-Verlag, Berlin, New York, 1972.

\bibitem[Rez10]{rezkhtytps}
C.~Rezk, \emph{Toposes and homotopy toposes (version 0.15)},
  \url{https://www.researchgate.net/publication/255654755_Toposes_and_homotopy_toposes_version_015},
  2010.

\bibitem[SP09]{schommerpries_thesis}
C.~J. Schommer-Pries, \emph{The classification of two-dimensional extended
  topological field theories}, Ph.D. thesis, University of California,
  Berkeley, Berkeley, CA 94720-3840, United States, 2009.

\bibitem[SS00]{schwedeshipley_mon}
S.~Schwede and B.~Shipley, \emph{Algebras and modules in monoidal model
  categories}, Proc. London Math. Soc. \textbf{80} (2000), no.~3, 491--511.

\bibitem[Ste26]{rs_monbicat}
R.~Stenzel, \emph{Symmetry shifting for monoidal bicategories}, to appear,
  2026.

\bibitem[SY19]{schlankyanovski_eh}
T.~M. Schlank and L.~Yanovski, \emph{The $\infty$–categorical
  {E}ckmann--{H}ilton argument}, Algebraic \& Geometric Topology \textbf{19}
  (2019), no.~6, 3119--3170.

\end{thebibliography}
\end{document}